\documentclass[10pt,twoside, a4paper, english, reqno]{amsart}
\pdfoutput=1
\usepackage{aliases}

\title[MLMC-FDM for random fractional conservation laws]
{Multi-level Monte Carlo Finite Difference Methods for Fractional Conservation Laws with Random Data}

\date{}

\author[U. Koley]{Ujjwal Koley}
\address[Ujjwal Koley]{\newline
 Tata Institute of Fundamental Research, 
 Centre for Applicable Mathematics,
 Post Bag No. 6503, GKVK Post Office,
 Sharada Nagar, Chikkabommasandra,
 Bangalore 560065, India}
\email[]{ujjwal@math.tifrbng.res.in}

\author[D. Ray]{Deep Ray}
\address[Deep Ray]{\newline
Department of Computational and Applied Mathematics,
Rice University, Houston, TX-77005,
USA}
\email[]{deep.ray@rice.edu}

\author[T. Sarkar]{Tanmay Sarkar}
\address[Tanmay Sarkar]{\newline
 Department of Mathematics, 
 Indian Institute of Technology Jammu,
 Jagti, NH-44 Bypass Road, Post Office Nagrota,
 Jammu - 181221, India}
\email[]{tanmay.sarkar@iitjammu.ac.in}



\thanks{}

\begin{document}
\begin{abstract}
We establish a notion of random entropy solution for degenerate fractional conservation laws incorporating randomness in the initial data, convective flux and diffusive flux. In order to quantify the solution uncertainty, we design a multi-level Monte Carlo Finite Difference Method (MLMC-FDM) to approximate the ensemble average of the random entropy solutions. Furthermore, we analyze the convergence rates for MLMC-FDM and compare it with the convergence rates for the deterministic case. Additionally, we formulate error vs. work estimates for the multi-level estimator. Finally, we present several numerical experiments to demonstrate the efficiency of these schemes and validate the theoretical estimates obtained in this work.
\end{abstract}

\maketitle

% ------------------------------------------------------------------------------------------------------------------------------------------

\section{Introduction}
The last decade has witnessed remarkable advances in the area of degenerate non-linear non-local integral partial differential equations. In this paper, we consider the following Cauchy problem in multiple space dimensions
\begin{align}
\label{fdcd}
\begin{cases}
\partial_t u(t,x) + \nabla \cdot f(u(t,x)) = - (- \Delta)^{\ld/2}[A(u(t,\cdot))](x), &\quad (t,x)\in Q_T :=(0,T)\times\R^d,\\
 u(0,x) = u_0(x),  \qquad \qquad & \quad x\in \R^d,
\end{cases}\end{align}
where $T>0$ is fixed, $u: Q_T\rightarrow \R$ is the unknown function, $u_0$ is the initial condition, $f= (f_1, f_2,...,f_d):\R\rightarrow\R^d$ is the flux function and $A:\R\rightarrow\R$ is the nonlinear diffusion term. We assume that $f\in W^{1,\infty}(\R,\R^d)$ and $A\in W^{1,\infty}(\R)$. Furthermore, we assume that $A(\cdot)$ is non-decreasing with $A(0)=0$, thus allowing \eqref{fdcd} to be degenerate. Additionally, we make the following assumption on the initial condition
\begin{align}
\label{initial_condn}
u_0\in L^{\infty}(\R^d)\cap L^1(\R^d)\cap BV(\R^d).
\end{align}
The non-local operator $-(-\bt)^{\ld/2}$ is the fractional Laplacian, defined for all $\vp(t,\cdot)\in C_c^{\infty}(\R^d)$ by
\begin{align}\label{dif_integralform}
-(-\bt)^{\ld/2}[\vp(t,\cdot)](x) = c_{\ld}\text{P.V.}\int_{|z|>0} \frac{\vp(t,x+z)-\vp(t,x)}{|z|^{d+\ld}}~ dz,
\end{align} 
for some constant $c_{\ld}>0,~\ld\in (0,2)$. More precisely, the constant $c_{\ld}$ is given by \cite{alibaud2014optimal, droniou2006fractal, droniou2010numerical}
\begin{align}
\label{c-lambda}
c_{\ld} = \frac{2^{\ld-1}\ld \Gamma(\frac{d+\ld}{2})}{\pi^{\frac{d}{2}}\Gamma(1-\frac{\ld}{2})},
\end{align}
where $\Gamma$ is the Gamma function defined as follows
\begin{align*}
\Gamma(\xi) := \int_0^{\infty} e^{-p}p^{\xi-1}~dp, \quad \xi>0.
\end{align*}
A large number of phenomena in physics and finance are modeled by equations of type \eqref{fdcd}, see \cite{applebaum}. In particular, non-local operator appears in mathematical models for viscoelastic materials \cite{Grandmont}, fluid flows and acoustic propagation in porous media \cite{Pablo} and pricing derivative securities in financial markets \cite{black}. Observe that model described by \eqref{fdcd} encompasses scalar conservation laws ($A=0$), linear fractional conservation laws ($A(u)=u$), and fractional porous medium equations ($f=0, A(u)=u |u|^{m-1}, m\ge 1$). In fact, \eqref{fdcd} is an extension to the fractional diffusion setting of the degenerate convection diffusion equations 
\begin{align}
\label{degenerate}
\partial_t u(t,x) + \nabla \cdot f(u(t,x)) = \Delta A(u(t,x)).
\end{align}

The non-linearity of the flux function and possible degeneracy of the diffusion term in \eqref{fdcd}, can lead to loss of regularity in the solution, even with smooth initial conditions. Thus, \emph{weak solutions} to \eqref{fdcd} must be sought. The notion of weak solution is defined as follows:
\begin{defi}(Weak solution)
A function $u(t,x)\in L^{\infty}(Q_T) \cap C([0,T]; L^1(\R^d))$ is called a weak solution of the non-local initial value problem \eqref{fdcd} if 
\begin{Definitions}
\item $u(t,x)$ satisfies the following integral formulation: 
\begin{equation}
\iint_{Q_T} \Big(u(t,x)\partial_t\vp(t,x) + f(u(t,x))\cdot \nabla\vp(t,x) - A(u(t,x))\,(-\bt)^{\ld/2}[\vp(t,\cdot)](x)\Big)~dxdt = 0,
\end{equation}
for all test functions $\vp\in C_c^{\infty}((0,T)\times\R^d)$.
\item For almost every $x \in \R^d$, $u(0,x)= u_0(x)$. 
\end{Definitions}
\end{defi}
It is well-known that weak solutions to \eqref{fdcd} need not be unique \cite{burger2000}. Consequently, an \emph{entropy admissibility condition} must be imposed to single out the physically relevant solution. To describe the entropy framework for \eqref{fdcd}, we introduce Kru\v{z}kov's convex entropy-entropy flux pair \cite{Kruzkov}: 
\begin{align*}
\eta_k(u)= |u-k|, \qquad q_k(u) = \eta'_k(u)(f(u)-f(k)),
\end{align*}
where $\eta'_k(u)= \sg(u-k)$, and $k\in\R$ is a constant. 
For technical reasons (cf. \cite{alibaud2007entropy, cifani2011entropy}), we decompose the non-local operator $g:=-(-\bt)^{\ld/2}$ into two parts. For each $r>0$, we write $g[\vp]=g_r[\vp]+g^r[\vp]$, where
\begin{align*}
g_r[\vp(t,\cdot)](x) & = c_{\ld} \,\text{P.V.} \int_{|z|\le r}\frac{\vp(t,x+z)-\vp(t,x)}{|z|^{d+\ld}}~ dz,\\
g^r[\vp(t,\cdot)](x) & = c_{\ld} \int_{|z|>r}\frac{\vp(t,x+z)-\vp(t,x)}{|z|^{d+\ld}}~ dz,
\end{align*} 
where $c_{\ld}$ is given by \eqref{c-lambda}. We are now ready to define the notion of an \emph{entropy solution} for \eqref{fdcd}.
\begin{defi}(Entropy solution)
A function $u(t,x)$ is said to be an entropy solution of the initial value problem \eqref{fdcd} provided 
\begin{Definitions}
\item $u\in L^{\infty}(Q_T)\cap C([0,T]; L^1(\R^d))$.
\item For all $k\in\R$, all $r>0$, and all test functions $0\leq \vp\in C_c^{\infty}([0,T)\times\R^d)$,
\begin{align*}
\iint_{Q_T} \left(\eta_k(u)\partial_t\vp + q_k(u)\cdot\nabla\vp + \eta_{A(k)}(A(u))g_r[\vp] + \eta'_k(u) g^r[A(u)]\vp)\right)~dxdt
& + \int_{\R^d}\eta_k(u_0(x))\vp(0,x)~dx \geq 0.
\end{align*}
\end{Definitions}
\end{defi}
For scalar conservation laws, the entropy framework was introduced by Kru\v{z}kov \cite{Kruzkov} and Vol'pert \cite{Volpert}, while entropy solutions for the degenerate parabolic equations \eqref{degenerate} were first considered by Vol'pert and Hudjaev \cite{VolpertHudajev}. Uniqueness of entropy solutions to \eqref{degenerate} was first proved by Carrillo \cite{Carrillo}. Numerical approximation of entropy solutions for both hyperbolic and degenerate hyperbolic equations are quite well developed in literature. We mention a few references, which by no means is exhaustive. Finite difference schemes, for hyperbolic problem, have been studied by Ole\u{i}nik \cite{Oleinik}, Harten \textit{et al}.~\cite{Hartenetal}, and several others. Finite difference schemes for degenerate equations were analysed by Evje and Karlsen \cite{EvjeKarlsen1}, and Karlsen et al. \cite{ujjwal}. Several efficient numerical schemes have been proposed and analyzed to solve the non-local model \eqref{fdcd}. Finite difference/volume schemes have been developed in \cite{cifani2011entropy, droniou2010numerical} (see references therein), with error estimates for such schemes obtained in \cite{cifani2014numerical}. Discontinuous Galerkin methods for \eqref{fdcd} have been proposed and analyzed by Cifani et al. \cite{cifani2010discontinuous, cifani2011discontinuous}, and also by Xu and Hesthaven \cite{xu2014discontinuous}.

The \emph{classical} paradigm for designing efficient numerical schemes assumes that data for \eqref{fdcd}, i.e., initial data $u_0$, convective flux and diffusive flux, are known exactly. 
In many situations of practical interest, however, deterministic data is unavailable due to inherent uncertainty in modeling and
measurements of physical parameters, such as the coefficients of specific heat in the equation of state for compressible gases, or the relative permeabilities in models of multi-phase flow in porous media.  Often, the initial data is known only up to certain statistical quantities of interest like the mean, variance, higher moments, and in some
cases, the law of the stochastic initial data. Thus, a mathematical formulation is necessary for \eqref{fdcd} which allows randomness in the initial data, as well as in the convective and diffusive fluxes. The problem of random initial data was considered in \cite{mishra2012sparse}, where the existence and uniqueness of a
random entropy solution was shown along with a convergence analysis for Multi-Level Monte-Carlo Finite Volume (MLMCFV) discretizations. In \cite{mishra2016numerical} a mathematical framework was developed for scalar conservation laws with random flux functions. The Multi-Level Monte-Carlo (MLMC) discretization of random degenerate parabolic equation $(\lambda=2)$ was investigated in \cite{koley2013multilevel}.
However, as per our knowledge, MLMC discretization of \eqref{fdcd} for the values of $\lambda\in (0,2)$ has not been addressed in the literature. In addition, due to the fact that the equation \eqref{fdcd} changes its nature (from hyperbolic to parabolic) for the ascending values of $\lambda\in (0,2)$, the convergence rates of certain numerical schemes heavily depend on the values of $\lambda$ in the deterministic setup. As a consequence, the convergence analysis in MLMC discretization for the random case differs from the existing literature for values of $\lambda\in (0,2)$.

More precisely, the main contributions of this paper are listed below:
\begin{enumerate}
\item We develop an appropriate mathematical framework of random entropy solution for the non-local equation \eqref{fdcd}. By generalizing the classical well-posedness results for the fractional degenerate convection-diffusion equation \eqref{fdcd} to the case of random initial data and random convective and diffusive fluxes, we define random entropy solutions and develop well-posedness results. We remark that our solution concept is different from the \emph{stochastic entropy solution} for randomly forced fractional conservation laws with multiplicative noise. Several results are available in that direction. For well-posedness theory of stochastic conservation laws, we refer to \cite{BaVaWit_2012, BisKoleyMaj,KoleyMajval2}. For the degenerate stochastic conservation laws, interested reader can consult \cite{KoleyMajval1} (see also references therein), and finally for stochastic degenerate fractional conservation laws, consult \cite{Koleyval1,Koleyval2}.
\item We design and analyze robust algorithms for computing the random entropy solutions. We begin with by describing the explicit and explicit-implicit schemes in deterministic case. Under $\lambda$-dependent CFL condition, we analyze the convergence rates and work estimates. We generalize the error analysis to the case of random input data.

\item To compute random entropy solutions, we rely on a method based on Monte Carlo (MC) sampling. In a MC method, the probability space is sampled and the non-local PDE is solved for each sample. However, the main drawback of MC method is that it converges as $1/\sqrt{M}$ where $M$ is the number of MC samples. Since this asymptotic rate can not be improved due to the central limit theorem, we require a large number of samples  in order to obtain low statistical errors. To overcome this drawback of MC methods, we propose a multi-level Monte Carlo (MLMC) method based on the explicit/explicit-implicit schemes for the deterministic non-local equations. We demonstrate that the resulting schemes converge. In addition, we determine the optimal number of MC samples needed at each mesh level to minimize the overall computational work.
\end{enumerate}

The rest of the paper is organized as follows: in Section~\ref{review}, we recapitulate the existence and stability results for the deterministic degenerate fractional conservation laws \eqref{fdcd}, following which we generalize the results for a random input data. We present finite difference schemes for the non-local equation in Section~\ref{numerical}, along with their associated work estimates. We describe the MC method in Section~\ref{mlmc-mc}, and analyze the convergence rates which are inferior to the deterministic convergence rates. We also describe MLMC method and obtain the corresponding convergence rate estimates. In addition, we determine the optimal number of samples for a fixed error tolerance. Finally, Section~\ref{results} is devoted to numerical experiments which confirm the theoretical estimates.

% ------------------------------------------------------------------------------------------------------------------------------------------

\section{Random degenerate fractional conservation laws}
\label{review}
Our aim is to develop a framework for random entropy solutions of degenerate fractional convection-diffusion equations, with a particular class of random initial data and random flux functions. We begin by first stating the known results with deterministic data.
\subsection{Entropy solution}
Under the assumptions on $f$ and $A$ described in Section 1, the Cauchy problem \eqref{fdcd} with deterministic data admits a unique entropy solution $u(t,\cdot)\in L^1(\R^d)\cap L^{\infty}(\R^d)$ for every $t>0$ corresponding to each $u_0\in L^1(\R^d)\cap L^{\infty}(\R^d)\cap BV(\R^d)$ (for details, consult \cite{cifani2011entropy}). Let us define the data-to-solution operator as
\begin{align}\label{datatosol}
S: (u_0,f,A)\mapsto u(t,\cdot)= S(t)(u_0,f,A), \quad t>0.
\end{align}
We also introduce the following notation to describe the well-posedness results:
\begin{align}\label{defn_E}
E(u_0) = |u_0|_{BV(\R^d)}\Big[1+ \Big(\ln \frac{\|u_0\|_{L^1(\R^d)}}{|u_0|_{BV(\R^d)}}\Big) \Big]\chi_{\frac{\|u_0\|_{L^1(\R^d)}}{|u_0|_{BV(\R^d)}}>1},
\end{align}
where $\chi$ denotes the characteristics function, with the convention that $E(u_0)=0$ whenever $|u_0|_{BV(\R^d)}=0$. Furthermore, we denote $\displaystyle\esssup_{I(u_0)}|f'|$ by $\|f'\|_{L^{\infty}(\R^d)}$, where $I(u_0) = (\essinf u_0,\esssup u_0)$.
The following theorem summarizes some of the fundamental results from  \cite{alibaud2012continuous} regarding the entropy solution $u$ of \eqref{fdcd}.
%%%%%%%%%%%%%%% deterministic existence thm %%%%%%%%%%%%%%%%%%%%%
\begin{thm}
\label{determ_exis}
Let $f$ and $A$ be locally Lipschitz continuous functions. Then
\begin{itemize}
\item [(i)]for every $u_0\in L^1(\R^d)\cap L^{\infty}(\R^d)\cap BV(\R^d)$, the initial value problem \eqref{fdcd} admits a unique BV entropy weak solution $u\in L^{\infty}(Q_T)\cap C([0,T];L^1(\R^d))\cap L^{\infty}([0,T];BV(\R^d))$.
\item [(ii)] For every $t>0$, the data-to-solution map $S(t)$ given by \eqref{datatosol}
satisfies the following estimates
\begin{align}
\|S(t)(u_0,f,A)\|_{L^{\infty}(\R^d)} & \leq \|u_0\|_{L^{\infty}(\R^d)},\label{det_estim_0}\\
\|S(t)(u_0,f,A)\|_{L^1(\R^d)} & \leq \|u_0\|_{L^1(\R^d)}, \label{det_estim_1}\\
\|S(t)(u_0,f,A)\|_{BV(\R^d)} & \leq |u_0|_{BV(\R^d)} \label{det_estim_2}.
\end{align}
\item [(iii)] Modulus of continuity in time: for all $t_1,t_2\geq 0$,
\begin{align}
\|S(t_1)(u_0,f,A) - S(t_2)(u_0,f,A)\|_{L^{1}(\R^d)} \leq |u_0|_{BV(\R^d)} \|f'\|_{L^{\infty}(\R^d)} |t_1 - t_2| + C\E^{t_1 - t_2}_{\ld,u_0,A}
\end{align}
with $C= C(d,\ld)$, and $\E^{t_1 - t_2}_{\ld,u_0,A}$ is given by
\[
\E^{t_1 - t_2}_{\ld,u_0,A} = 
\begin{cases}
|u_0|_{BV(\R^d)}\|(A')^{1/\ld}\|_{L^{\infty}(\R)} |t_1^{1/\ld} - t_2^{1/\ld}|, & \mathrm{if} \,\, \ld\in (1,2),\\[1.5mm]
E(u_0)\|A'\|_{L^{\infty}(\R)}|t_1 - t_2|
+|u_0|_{BV(\R^d)}\|A'\|_{L^{\infty}(\R)}(1+\|\ln A'\|_{L^{\infty}(\R)})|t_1 - t_2|\\
\qquad \qquad \qquad \qquad +|u_0|_{BV(\R^d)}\|A'\|_{L^{\infty}(\R)} |t_1\ln t_1 - t_2\ln t_2|, & \mathrm{ if  }\,\, \ld=1,\\[1.5mm]
\|u_0\|_{L^1(\R^d)}^{1-\ld}|u_0|_{BV(\R^d)}^{\ld}\|A'\|_{L^{\infty}(\R)}|t_1 - t_2|, & \mathrm{ if  }\,\, \ld\in (0,1),
\end{cases}
\]
where $E(u_0)$ is defined in \eqref{defn_E}.
\end{itemize}
\end{thm}
\begin{proof}
The existence and uniqueness results (i) follows from [\cite{cifani2011entropy}, Theorem 5.3]. For the estimates presented in (ii), one can refer to [\cite{alibaud2012continuous}, Theorem 2.2]. Finally, the modulus of continuity result is demonstrated in [\cite{alibaud2014optimal}, Corollary 3.6].
\end{proof}
\begin{rem}($L^2$-estimates)\\
It is not difficult to observe that under the assumptions of Theorem \ref{determ_exis}, the following energy estimate holds for every $u_0\in L^1(\R^d)\cap L^{\infty}(\R^d)\cap BV(\R^d)$,
\begin{align}\label{p-estimate}
\|u(t,\cdot)\|^2_{L^2(\R^d)} \leq \|u_0\|^2_{L^2(\R^d)}.
\end{align}
\end{rem}
We also require continuous dependence results for the degenerate fractional convection-diffusion equations with respect to the initial data, convective flux and diffusive flux. For the detailed proof, we refer to [\cite{alibaud2014optimal}, Theorem 3.1]. Let $v$ be the entropy solution for the problem 
\begin{align}\label{cont_dep}
\begin{cases}
\partial_t v(t,x) + \nabla\cdot \tilde{f}(v(t,x)) = - (- \Delta)^{\ld/2} [B(v(t,\cdot))](x), & \quad (t,x)\in Q_T=(0,T)\times\R^d,\\
 v(0,x) = v_0(x),  \qquad \qquad & \quad x\in \R^d,
\end{cases}
\end{align}
where $v_0$, $\tilde{f}$ and $B$ undertake the same assumptions of $u_0$, $f$ and $A$ respectively. For convenience, we will use the following notations:
\begin{align*}
\|f'-\tilde{f}'\|_{L^{\infty}(\R^d)} & := \esssup_{I(u_0)} |f'-\tilde{f}'|,\\
\|A'- B'\|_{L^{\infty}(\R)} & := \esssup_{I(u_0)} |A' - B'|.
\end{align*}
%%%%%%%%%%%%%%%%%%%%%%%%%%%%%%%%%%%%%%%%%%%%%%%%%%%%%%%%%%%%%%%%%
\begin{thm}[see \cite{alibaud2014optimal}]
\label{stab_det}
Assume that $u_0,v_0\in L^1(\R^d)\cap L^{\infty}(\R^d)\cap BV(\R^d)$; $f(\cdot),\tilde{f}(\cdot)\in \text{Lip}_{\mathrm{loc}}(\R;\R^d)$, and $A(\cdot),B(\cdot)\in \text{Lip}_{\mathrm{loc}}(\R)$ with $A',B'\geq 0$. 
Let $u(t,\cdot)=S(t)(u_0,f,A)$ and $v(t,\cdot)=S(t)(v_0,\tilde{f},B)$ be the unique entropy solutions of \eqref{fdcd} and \eqref{cont_dep} respectively. Then the following a priori continuity estimate holds
\begin{equation}\label{estim_cont_dep}
\|u-v\|_{C([0,T];L^1(\R^d))} \leq \|u_0-v_0\|_{L^1(\R^d)} + T|u_0|_{BV(\R^d)}\|f'-\tilde{f}'\|_{L^{\infty}(\R^d)} + C\,\E^{A-B}_{T,\ld,u_0},
\end{equation}
where the constant $C$ depends only on $d$ and $\ld$, and $\E^{A-B}_{T,\ld,u_0}$ is given by
\[
\E^{A-B}_{T,\ld,u_0} = 
\begin{cases}
T^{1/\ld}|u_0|_{BV(\R^d)} \|(A')^{1/\ld}- (B')^{1/\ld}\|_{L^{\infty}(\R)}, & \mathrm{ if  }\,\, \ld\in (1,2),\\[1.5mm]
TE(u_0) \|A'-B'\|_{L^{\infty}(\R)}+ T(1 + |\ln T|)|u_0|_{BV(\R^d)}\|A'-B'\|_{L^{\infty}(\R)} \\
\qquad \qquad \qquad \qquad + T |u_0|_{BV(\R^d)} \|A'\ln A^{'} - B'\ln B'\|_{L^{\infty}(\R)},  & \mathrm{ if  }\,\, \ld=1,\\[1.5mm]
T \|u_0\|^{1-\ld}_{L^1(\R^d)} |u_0|^{\ld}_{BV(\R^d)} \|A' - B'\|_{L^{\infty}(\R)}, & \mathrm{ if  }\,\, \ld\in (0,1).
\end{cases}
\]
\end{thm}
\subsection{Entropy solutions with random data}
The existence and uniqueness of the entropy solution for \eqref{fdcd} in the absence of a diffusive flux and with random initial data, has been shown in \cite{mishra2016numerical}. For $\ld=2$, the well-posedness of \eqref{fdcd} with a random diffusive flux has been studied in \cite{koley2013multilevel}. In this work, we focus on the random entropy solutions for the fractional degenerate convection-diffusion equation \eqref{fdcd}, where the initial data $u_0$, flux function $f$ and diffusive flux $A$ are all random, taking values in the Banach spaces $L^1(\R^d)\cap BV(\R^d)\cap L^{\infty}(\R^d)$, $W^{1,\infty}(\R;\R^d)$ and $W^{1,\infty}(\R)$ respectively.

For a probability space $(\Omega,\F,\Pb)$, we consider the following strongly measurable maps defined on the sample space $\Omega$: 
\begin{itemize}
\item [(a)] random initial data $u_0:\Omega\rightarrow L^1(\R^d)\cap BV(\R^d)\cap L^{\infty}(\R^d)$.
\item [(b)] random convective flux $f:\Omega\rightarrow \Lip(\R;\R^d)$.
\item [(c)] random diffusive flux $A:\Omega\rightarrow \Lip(\R;\R^d)$.
\end{itemize}
We proceed with an approach similar to that described in \cite{koley2013multilevel}.
\begin{defi}
Random data $(u_0,f,A)$ for the fractional degenerate convection-diffusion equation \eqref{fdcd} is a random variable taking values from
\begin{align*}
E_R = (BV(\R^d)\cap L^1(\R^d)\cap L^{\infty}(\R^d)) \times W^{1,\infty}(\R;\R^d) \times W^{1,\infty}(\R;\R^d).
\end{align*}
The set $E_R$ is a Banach space endowed with the norm
\begin{align*}
\|(u_0,f,A)\|_{E_R} = \|u_0\|_{L^1(\R^d)} + TV(u_0) + \|u_0\|_{L^{\infty}(\R^d)} + \|f\|_{W^{1,\infty}(\R;\R^d)} + \|A\|_{W^{1,\infty}(\R;\R^d)}.
\end{align*}
In particular, random data $(u_0,f,A)$ for \eqref{fdcd} is a strongly measurable function
\begin{align*}
(u_0,f,A): (\Omega,\F)\mapsto (E_R,\B(E_R)).
\end{align*}
\end{defi}
We are interested in the random solutions of the following random degenerate fractional convection-diffusion equation
\begin{equation}\label{rfdcd}
\begin{cases}
\partial_t u(\omega;t,x) + \nabla\cdot f(\omega; u(\omega;t,x)) = - (-\triangle)^{\ld/2} [A(\omega;u(\omega;t,\cdot))](x), & \quad t>0,~x\in\R^d, \, \omega\in \Omega,\\
u(\omega;0,x)= u_0(\omega;x), & \quad x\in\R^d, \omega\in \Omega.
\end{cases}
\end{equation}
We assume the following bounds to carry out the convergence analysis:
\begin{Assumptions}
\item $u_0(\omega;\cdot)$ satisfies $\Pb$-a.s.
\begin{align}
\|u_0(\omega;\cdot)\|_{L^{\infty}(\R^d)} & \leq \mathcal{M}, \label{assum_1}\\
|u_0(\omega;\cdot)|_{BV(\R^d)} & \leq C_{TV} < \infty.
\end{align}
\item The flux functions $f$ and $A$ satisfy $\Pb$-a.s.
\begin{align}
\|f(\omega;\cdot)\|_{W^{1,\infty}(I(u_0);\R^d)} & \leq C_f,\\
\|A(\omega;\cdot)\|_{W^{1,\infty}(I(u_0);\R^d)} & \leq C_A
\end{align}
with finite $C_f,C_A$, and $A'(\omega;\cdot)\geq 0$.
Driven by the above bounds on the flux functions, we refer to $f$ and $A$ as a bounded random flux and bounded random diffusion operator respectively. \\
\item Furthermore, let us assume that
\begin{align}\label{assum_5}
\|u_0\|_{L^{q}(\Omega;L^1(\R^d)\cap L^2(\R^d))}<\infty, \quad q=1,2.
\end{align}
Hereby, we observe that for every $\omega\in\Omega$, the map
\begin{align*}
\omega \mapsto \Big(\|u_0(\omega;\cdot)\|_{L^1(\R^d)}, TV(u_0(\omega;\cdot)), \|u_0(\omega;\cdot)\|_{L^{\infty}(\R^d)}, \|f\|_{W^{1,\infty}(\R;\R^d)}, \|A\|_{W^{1,\infty}(\R;\R^d)} \Big)
\end{align*}
is in $L^2(\Omega;\R^5)$.
\end{Assumptions}
Under the above assumptions, we define the notion of weak solution and entropy solution for \eqref{rfdcd}.
%%%%%%%%%%%%%%%%%%% random weak/entropy soln %%%%%%%%%%%%%%%%%%%%%%%%%%%%%%%
\begin{defi}(Random weak solution)
A random field $u: \Omega\ni\omega \rightarrow u(\omega;t,x)$, i.e., a measurable function from $(\Omega,\F)$ to $C([0,T];L^1(\R^d))$, is called a random weak solution of \eqref{rfdcd} with random initial data $u_0$, random flux function $f$, and random diffusive flux $A$ satisfying \eqref{assum_1}-\eqref{assum_5}, provided the following integral formulation holds $\Pb$-a.s.:
\begin{align*}
\int_0^{\infty}\int_{\R^d} \big(u(\omega;t,x)\partial_t\vp + f(\omega,u(\omega;t,x))\cdot\nabla\vp + A(\omega;u(\omega;t,x))\,g[\vp] \big)~dxdt
+ \int_{\R^d} u_0(\omega;x)\vp(0,x)~dx = 0,
\end{align*}
for all test functions $\vp\in C_0^1([0,\infty)\times\R^d)$.
\end{defi}
\begin{defi}(Random entropy solution)
A random field $u: \Omega\ni\omega \rightarrow u(\omega;t,x)$, i.e., a measurable function from $(\Omega,\F)$ to $C([0,T];L^1(\R^d))$ is called a random entropy solution of \eqref{rfdcd} with random initial data, flux function and diffusive flux satisfying \eqref{assum_1}-\eqref{assum_5}, if for all $k\in\R$, and any pair consisting of a (deterministic) entropy $\eta_k$ and (stochastic) entropy flux $q_k(\omega;\cdot)$ such that $\eta_k$ is convex, $q'_{i,k}(\omega;\cdot) = \eta_k'f'_{i}(\omega,\cdot)$, and $u(\omega;t,x)$ satisfies $\Pb$-a.s. the following inequality:
\begin{align*}
\int_0^{\infty}\int_{\R^d} \Big(\eta_k(u(\omega;t,x))\partial_t\vp & + \nabla q_k(\omega;u(\omega;t,x))\cdot\nabla\vp
+ \eta_{A(k)}A(\omega;u(\omega;t,x))\,g_r[\vp] \\
& + \eta_k'(u)g^r[A(\omega;u(\omega;t,\cdot))](x) \, \vp \Big)\,dx\,dt
+ \int_{\R^d} \eta_k(u_0(\omega;x))\vp(0,x)\,dx \geq 0,
\end{align*}
for all $r>0$ and all test functions $0\leq \vp\in C_0^1([0,\infty)\times\R^d)$.
\end{defi}
The following theorem generalizes the existence of random entropy solutions for random initial data from \cite{mishra2012sparse}, random convective flux function \cite{mishra2016numerical}, and random diffusive flux \cite{koley2013multilevel} to the non-local case.
\begin{thm}\label{random_entropy}
Consider the fractional degenerate convection-diffusion equation \eqref{rfdcd} with random initial data $u_0$, random flux function $f$, and random diffusion operator $A$ satisfying \eqref{assum_1}-\eqref{assum_5}. Then there exists a unique random entropy solution $u:\Omega\ni\omega\rightarrow C([0,T];L^1(\R^d))$ such that $\Pb$-a.s.
\begin{align*}
u(\omega;t,\cdot) = S(t)(u_0(\omega;\cdot),f(\omega;\cdot),A(\omega;\cdot)), \quad t>0.
\end{align*}
Moreover, for every $t\in [0,T]$, $0<T<\infty$;
\begin{align}
\|u\|_{L^2(\Omega; C([0,T]; L^2(\R^d)))} & \leq \|u_0\|_{L^2(\Omega;L^2(\R^d))}, \label{estim_1}\\
\|S(t)(u_0,f,A)(\omega)\|_{L^1(\R^d)\cap L^{\infty}(\R^d)} & \leq \|u_0(\omega;\cdot)\|_{L^1(\R^d)\cap L^{\infty}(\R^d)} \label{estim_2}.
\end{align}
Furthermore $\Pb$-a.s.
\begin{align}\label{estim_3}
|S(t)(u_0,f,A)(\omega)|_{BV(\R^d)} \leq |u_0(\omega;\cdot)|_{BV(\R^d)} 
\end{align}
such that $\Pb$-a.s., we have
\begin{align}\label{estim_4}
\sup_{0\leq t\leq T} \|u(\omega;t,\cdot)\|_{L^{\infty}(\R^d)} \leq \mathcal{M},
\end{align}
where $\mathcal{M}$ is defined by \eqref{assum_1}.
\end{thm}
\begin{proof}
Motivated by the Theorem \ref{determ_exis}, for $\omega\in\Omega$, we define $\Pb$-a.s. a random function $u(\omega; t,x)$ by
\begin{align}\label{dcd_rf}
u(\omega;t,\cdot) = S(t)(u_0,f,A)(\omega).
\end{align}
Note that the well-definedness of \eqref{dcd_rf} follows from the properties of $(S(t))_{t\geq 0}$ (refer to the Theorem \ref{determ_exis}). As a consequence, $\Pb$-a.s. $u(\omega;\cdot)$ is a weak entropy solution of \eqref{rfdcd}. 

All the estimates \eqref{det_estim_0}-\eqref{det_estim_2} hold $\Pb$-a.s. from the Theorem \ref{determ_exis}.
The measurability of the function for $0\leq t\leq T$, $\Omega\ni\omega\mapsto u(\omega;t,\cdot)\in L^1(\R^d)$ is demonstrated from the fact that composition of a continuous function with strongly measurable function becomes strongly measurable function and the assumption that $\Omega\ni\omega\mapsto (u_0,f,A)(\omega)\in E_R$ is a random variable along with the continuity estimates \eqref{estim_cont_dep}.

Assuming that $u_0\in L^2(\Omega;L^2(\R^d))$, we have for $t\in [0,T]$
\begin{align*}
\int_{\Omega} \|u(\omega;t,\cdot)\|^2_{L^2(\R^d)} \Pb(d\omega)
& = \int_{\Omega}\|S(t)u_0(\omega;\cdot)\|^2_{L^2(\R^d)} \Pb(d\omega)\\
& \leq \int_{\Omega}\|u_0(\omega;\cdot)\|^2_{L^2(\R^d)} \Pb(d\omega) = \|u_0\|^2_{L^2(\Omega;L^2(\R^d))},
\end{align*}
where we have used \eqref{p-estimate}. The estimate \eqref{p-estimate} additionally gives us
\begin{align*}
\|u\|^2_{L^2(\Omega;C(0,T;L^2(\R^d)))} & = \int_{\Omega}\big(\max_{t\in[0,T]}
\|S(t)u_0(\omega;\cdot)\|_{L^2(\R^d)} \big)^2 ~\Pb(d\omega)\\
& \leq \int_{\Omega}\|u_0(\omega;\cdot)\|^2_{L^2(\R^d)} \Pb(d\omega) = \|u_0\|^2_{L^2(\Omega;L^2(\R^d))}.
\end{align*}
Hence the estimate \eqref{estim_1} follows from the assumption \eqref{assum_5}.

We obtain the estimates \eqref{estim_2} and \eqref{estim_3} $\Pb$-a.s. from the estimates \eqref{det_estim_0}-\eqref{det_estim_2}. Moreover, the boundedness of $u(\omega;t,\cdot)$ follows from the assumption \eqref{assum_1}.
\end{proof}

\begin{rem}(Stability estimates)\\
The continuous dependence results of random degenerate fractional convection-diffusion equation \eqref{rfdcd} follow in an analogous way from the deterministic case (described in Theorem \ref{stab_det}).
\end{rem}

% ------------------------------------------------------------------------------------------------------------------------------------------

\section{Numerical approximations of degenerate fractional conservation laws}
\label{numerical}
In this section we derive efficient numerical schemes for the fractional degenerate convection diffusion equation \eqref{fdcd}. The following analysis can be generalized for higher space dimensions since the convergence rates remain same for all spatial dimensions (refer to \cite{cifani2014numerical}). However, for simplicity, we carry out our analysis for the following one-dimensional model
\begin{equation}\label{fdcd_1d}
\begin{cases}
\partial_t u(t,x) + \partial_x f(u(t,x)) = -(-\Delta)^{\lambda/2}[A(u(t,\cdot))](x), \quad t>0,~x\in \mathbb{R},\\
u(0,x) = u_0(x),\quad x\in \R.
\end{cases}
\end{equation}

We consider a uniform discretization of the space and time domains. Let $\Delta t>0$ be the time-step and $\Delta x>0$ be the spatial mesh size. The spatial grid consists of the points $x_i=i\Delta x$ for $i\in\mathbb{Z}$, while the temporal grid is given by the points $t^n=n\Delta t$ for $n=0,1,2,...,N=\frac{T}{\Delta t}$. Furthermore, we use the notation $x_\iph$ to represent the interface $(\iph )\Delta x$ in the space domain. We consider the following explicit numerical scheme
\begin{align}
U_i^0 & = \frac{1}{\Delta x} \int_{x_\imh}^{x_\iph} u_0(x)~dx,\\
U_i^{n+1} & = U_i^n - \Delta t D^{-} F(U_i^n,U_{i+1}^n) + \Delta t \sum \limits_{j\neq 0} G_j \bigl(A(\U^n_{i+j})-A(\U^n_i)\bigr), \label{eqn:fdm_explicit}%\label{explicit_sc}
\end{align}
where $U_i^n$ is the approximate solution of equation \eqref{fdcd_1d} in the cell $[t^n,t^\npo) \times [x_\imh,x_\iph)$, $\U^n := [...,U^n_\imo,U^n_i,U^n_\ipo,...]^\top$ is the solution vector in the time slab $[t^n,t^\npo)$ and $D^{-}$ is the spatial difference operator given by $D^{-}(\cdot)_i := \bigl((\cdot)_i -(\cdot)_{i-1}\bigr)/\Delta x$. The numerical solution $u_\Delta(t,x)$ is a piecewise constant function given by
\begin{align}\label{num_sol}
u_{\Delta}(t,x) = U_i^n, \quad \text{for all } (t,x)\in [t_n,t_{n+1})\times [x_\imh,x_\iph).
\end{align}

The convective numerical flux $F:\R^2\rightarrow\R$ is chosen to be (i) Lipschitz continuous with Lipschitz constant $L_F$, (ii) consistent with the flux $f$, i.e., $F(u,u) = f(u)$, and (iii) monotone i.e.,
\begin{align*}
\frac{\partial}{\partial u}F(u,v)\geq 0, \quad \frac{\partial}{\partial v}F(u,v)\leq 0.
\end{align*}
For instance, the Lax-Friedrichs flux given by
\begin{align*}
F(U_i^n,U_{i+1}^n) = \frac{1}{2}(f(U_i^n)+f(U_{i+1}^n)) - \frac{\Delta x}{2\Delta t}(U_{i+1}^n - U_i^n)
\end{align*}
satisfies these conditions. Finally, using the integral formulation \eqref{dif_integralform} of the diffusion term and noting that the numerical solution is a piecewise constant function, the approximate diffusion operator is given in \eqref{eqn:fdm_explicit}, where
\begin{align}\label{numdif}
G_i = c_{\lambda} \int \limits_{x_\imh}^{x_\iph} \frac{dz}{|z|^{1+\lambda}},  \quad i \in \Z \setminus\{0\}.
\end{align}
Further details about the discretization of diffusion term are given in Section \ref{sec:frac_diss}. It is easy to observe that $G_i$ is positive and finite for $i\neq 0$. The numerical scheme \eqref{eqn:fdm_explicit} is monotone under the following CFL condition \cite{cifani2011entropy}
\begin{equation}\label{CFL_1}
2L_F \frac{\Delta t}{\Delta x} + \left(c_{\lambda}2^{\lambda}L_A \int_{|z|>1} \frac{dz}{|z|^{1+\lambda}}\right)
\frac{\Delta t}{\Delta x^{\lambda}}  \leq 1,
\end{equation}
where $L_F$ and $L_A$ are the Lipschitz constants of $F$ and $A$ respectively.

We also consider the following explicit-implicit scheme:
\begin{align}
U_i^{n+1} & = U_i^n - \Delta t D^{-} F(U_i^n,U_{i+1}^n) + \Delta t \sum_{j\neq 0} G_j\left(A(U_{i+j}^{n+1})-A(U_i^{n+1})\right). \label{eqn:fdm_impexp}%\label{implicit_explicit}
\end{align}
For the scheme \eqref{eqn:fdm_impexp}, we need the CFL condition \cite{cifani2014numerical}
\begin{align}\label{CFL_imp_1}
2L_F \frac{\Delta t}{\Delta x} \leq 1.
\end{align}
%%%%%%%%%%%%%%%Theorem %%%%%%%%%%%%%%%%%%%%%%%
\begin{thm}\label{conv_esti}(Convergence to the entropy solution and a priori estimates)\\
Assume that $u_0\in L^{\infty}(\R)\cap L^1(\R)\cap BV(\R)$.
Let $u_{\Delta}$ be a sequence of solutions either of the explicit scheme \eqref{eqn:fdm_explicit} or explicit-implicit scheme \eqref{eqn:fdm_impexp}. Furthermore, assume that the CFL conditions \eqref{CFL_1} and \eqref{CFL_imp_1} hold for the schemes \eqref{eqn:fdm_explicit} and \eqref{eqn:fdm_impexp} respectively.
Then the approximations $u_{\Delta}$ converge up to a subsequence to $u$ in $C([0,T];L^1(\R))$ as $\Delta x\rightarrow 0$ with,
\begin{align*}
u\in L^{\infty}(Q_T)\cap C([0,T];L^1(\R))\cap L^{\infty}([0,T];BV(\R)).
\end{align*}
Moreover, $u$ is the unique entropy solution of \eqref{fdcd_1d} and the following estimates hold:
\begin{itemize}
\item [(i)] $\|u_{\Delta}(t,\cdot)\|_{L^{\infty}(\R)}\leq \|u_0\|_{L^{\infty}(\R)}$,
\item [(ii)] $\|u_{\Delta}(t,\cdot)\|_{L^1(\R)}\leq \|u_0\|_{L^1(\R)}$,
\item [(iii)] $\|u_{\Delta}(t,\cdot)\|_{BV(\R)}\leq \|u_0\|_{BV(\R)}$.
\end{itemize}
Additionally, the following time-regularity estimate holds for the schemes  \eqref{eqn:fdm_explicit} and  \eqref{eqn:fdm_impexp}
\begin{equation}
\|u_{\Delta}(s,\cdot) - u_{\Delta}(t,\cdot)\|_{L^1(\R)} \leq \sigma(|s-t|+\Delta t), \quad \forall s,t\geq 0,
\end{equation}
where the function $\sigma$ is given by
\[
\sigma(r) = 
\begin{cases}
C|r|, & \lambda\in (0,1),\\
C|r\ln r|, & \lambda = 1,\\
C|r|^{1/\lambda}, & \lambda\in(1,2),
\end{cases}
\]
for some constant $C>0$.
\end{thm}
\begin{proof}
The convergence of the approximate solution to the entropy solution of \eqref{fdcd_1d} is proved in [\cite{cifani2011entropy}, Theorem 4.4 and Lemma 4.6]. For the estimates on approximate solution, one can refer to [\cite{cifani2014numerical}, Lemma 5.1] and the time regularity estimates can be found in [\cite{cifani2011entropy}, Lemma 4.5] or [\cite{cifani2014numerical},  Lemma 5.3].
\end{proof}

For the convenience of further analysis, we replace the CFL conditions \eqref{CFL_1} and \eqref{CFL_imp_1} with the following  simplified condition
\begin{align}\label{CFL_2}
C\frac{\Delta t}{\Delta x^{1\vee\lambda}} \leq 1, \qquad \text{for }\lambda\in(0,2),
\end{align}
where the constant $C$ may depend on $\lambda$, and $\vee$ is defined by $a\vee b=\max\{a,b\}$.
%%%%%%%%%%% Theorem %%%%%%%%%%%%%%%%
\begin{thm}\label{conv_rate}(Convergence rate for approximate solution)\\
Let $u_0\in L^{\infty}(\R)\cap L^1(\R)\cap BV(\R)$ and $u_{\Delta}$ be the approximate solution obtained either by the explicit scheme \eqref{eqn:fdm_explicit} or by the explicit-implicit scheme \eqref{eqn:fdm_impexp} under the CFL condition \eqref{CFL_2}. 
\begin{enumerate}[(a)]
\item The following estimate holds with the scheme \eqref{eqn:fdm_explicit} for all $\lambda\in (0,2)$,
\begin{align}\label{conv_estimate}
\|u(T,\cdot) - u_{\Delta}(T,\cdot)\|_{L^1(\R)} \leq C_T \sigma^{EX}_{\lambda}(\Delta x)
\end{align}
for some constant $C_T$ independent of $\Delta x$, where the function $\sigma^{EX}_{\lambda}$ is given by
\begin{equation}\label{sigma_defn_explicit}
\sigma^{EX}_{\lambda}(r) = 
\begin{cases}
r^{\frac{1}{2}}, & \quad \text{if } \lambda\in (0,\frac{2}{3}],\\
r^{\frac{2-\lambda}{2+\lambda}}, & \quad \text{if }\lambda\in (\frac{2}{3},1)\cap (1,2).
\end{cases}
\end{equation}
For $\lambda = 1$, under the stronger CFL condition 
\begin{align*}
C\frac{\Delta t}{\Delta x^{\alpha}} \leq 1, \qquad \alpha\in (1,2),
\end{align*}
the following estimate holds:
\begin{align*}
\|u(T,\cdot) - u_{\Delta}(T,\cdot)\|_{L^1(\R)} \leq C_T \sigma^{EX}_{\alpha}(\Delta x).
\end{align*}
\item With the explicit-implicit scheme \eqref{eqn:fdm_impexp} for $\lambda\in (0,2)$, we have
\begin{align}\label{conv_estimate_implicit}
\|u(T,\cdot) - u_{\Delta}(T,\cdot)\|_{L^1(\R)} \leq C_T \sigma^{EI}_{\lambda}(\Delta x),
\end{align}
where the function $\sigma^{EI}_{\lambda}$ is given by
\begin{equation}\label{sigma_defn_implicit}
\sigma^{EI}_{\lambda}(r) = 
\begin{cases}
r^{\frac{1}{2}}, & \quad \text{if } \lambda\in (0,1),\\
r^{\frac{1}{2}}|\log r|, & \quad \text{if }\lambda = 1,\\
r^{\frac{2-\lambda}{2}}, & \quad \text{if }\lambda\in (1,2).
\end{cases}
\end{equation}
\end{enumerate}
\end{thm}
\begin{proof}
For a detailed proof we refer to [\cite{cifani2014numerical}, Theorem 6.3],  and [\cite{cifani2014numerical}, Corollary 6.4] for $\lambda =1$.
\end{proof}
The CFL condition \eqref{CFL_2} is instrumental to get the convergence rates. In order to simplify the discussion of the main ideas used in this paper, we ignore the critical case $\lambda = 1$ for the remainder of this paper. However, our ideas can be applied to the case $\lambda = 1$ at the expense of presumably long computations. Furthermore, we consider the non-local problem \eqref{fdcd_1d} in a bounded domain $I\subset\R$ and take into account $I$-periodic solution in $u$. It is straightforward to observe that all the above presented estimates in Theorem \ref{conv_esti} and Theorem \ref{conv_rate} also hold for $x\in I$.
\begin{rem}
As a consequence of the Theorem \ref{conv_rate}, we have the following estimate: 
\begin{align}
\|u(t^n,\cdot) - u_{\Delta}(t^n,\cdot)\|_{L^1(\R)} \leq \|u_{\Delta}(0,\cdot) - u_0\|_{L^1(\R)} + C_T \sigma_{\lambda}(\Delta x),
\end{align}
where $\sigma_{\lambda}(r)$ is defined by
\begin{align}\label{sigma_defn}
\sigma_{\lambda}(r) = 
\begin{cases}
\sigma^{EX}_{\lambda}(r), & \text{ for explicit scheme \eqref{eqn:fdm_explicit}}\\
\sigma^{EI}_{\lambda}(r), & \text{ for explicit-implicit scheme \eqref{eqn:fdm_impexp}}.
\end{cases}
\end{align}
\end{rem}
Moreover, we have the following result with a simple application of H{\"o}lder's inequality:
\begin{cor}\label{rate_conv_p}($L^2$-estimates)\\
Let the assumptions of Theorem \ref{conv_rate} hold. Then we have the following estimate for the rate of convergence of the schemes \eqref{eqn:fdm_explicit} and \eqref{eqn:fdm_impexp} in $L^2(I)$:
\begin{align}
\|u(t^n,\cdot) - u_{\Delta}(t^n,\cdot)\|_{L^2(I)} \leq C \big(\|u_{\Delta}(0,\cdot) - u_0\|^{1/2}_{L^1(I)} 
+ C_T^{1/2} (\sigma_{\lambda}(\Delta x))^{1/2} \big),
\end{align}
where the constant $C$ does not depend on $\Delta x$.
\end{cor}
In order to analyze the efficiency of the MC and MLMC methods, we need to estimate the computational work performed to compute the approximate solution with the FD-schemes in the deterministic case. In addition, we need to analyze how the computational work scales with respect to mesh refinement.
\subsection{Work bounds}
The computational work/cost can be obtained by evaluating the number of floating point operations performed during the execution of the algorithm. We compute the work estimate for the explicit scheme \eqref{eqn:fdm_explicit} as well as for the explicit-implicit scheme \eqref{eqn:fdm_impexp}.
Since the actual numerical simulations are performed on bounded domains, the number of grid cells in one dimension scales as $1/\Delta x$. 
\subsubsection{Work estimate for explicit scheme \eqref{eqn:fdm_explicit}}
It can be easily observed that due to the non-local term, the number of operations per time-step scales quadratically with the number of cells in spatial domain for the explicit scheme. Since the scale for the spatial domain is $1/\Delta x$, the work for the explicit scheme \eqref{eqn:fdm_explicit} can be estimated by 
\begin{align*}
W_{\Delta}^{EX} \leq C \Delta t^{-1} \Delta x^{-2}.
\end{align*}
Incorporating the CFL condition \eqref{CFL_2}, we obtain the following work bound:
\begin{equation}
W_{\Delta}^{EX} = 
\begin{cases}
\mathcal{O}(\Delta x^{-3}), & \quad \text{if } \lambda \in (0,1),\\
\mathcal{O}(\Delta x^{-\lambda -2}), & \quad \text{if }\lambda \in (1,2).
\end{cases}
\end{equation}
\subsubsection{Work estimate for explicit-implicit scheme \eqref{eqn:fdm_impexp}}
The scheme \eqref{eqn:fdm_impexp} requires a non-linear solver for $\U^\npo$ in each time-step, which can be computationally expensive if solved exactly. Thus, we approximately solve the equation using a suitable iterative method. In particular, we consider the Newton iteration method. We continue to do the iteration until the residual is $\mathcal{O}(\Delta t\Delta x)$ since the mapping $U^n\mapsto U^{n+1}=:\Psi(U^n)$ ($\Psi$ can be easily obtained from \eqref{eqn:fdm_impexp}) is a contraction thanks to sufficiently small $\Delta t$ and CFL constant \eqref{CFL_imp_1}.

It is possible to show that the additional error introduced due to the finite termination of the iterative solver does not contribute to an increment in the overall error. To see this, let $\widetilde{u}^{n,0}$ denote the approximate solution at time $t=t^n$ obtained by solving \eqref{eqn:fdm_impexp} exactly in each time-step and $\widetilde{u}^{n,m}$ represents the approximation of \eqref{eqn:fdm_impexp} by Newton iteration in the first $m$ time-steps and afterwards exactly. Hence we have $\widetilde{u}^{n,n}= u_{\Delta}(t^n,\cdot)$ and by triangle inequality and $L^1$-contraction of the numerical scheme
\begin{align*}
\|u_{\Delta}(t^n,\cdot) - u(t^n,\cdot)\|_{L^1(I)} & = \big\|\sum_{m=0}^{n-1}(\widetilde{u}^{n,m+1} - \widetilde{u}^{n,m}) 
+ \widetilde{u}^{n,0} - u(t^n,\cdot) \big\|_{L^1(I)}\\
& \leq \sum_{m=0}^{n-1}\big\|\widetilde{u}^{n,m+1} - \tilde{u}^{n,m} \big\|_{L^1(I)} 
+ \|\widetilde{u}^{n,0} - u(t^n,\cdot)\|_{L^1(I)}\\
&{ \leq \sum_{m=0}^{n-1}\big\|\widetilde{u}^{n,m+1}-\widetilde{u}^{m+1,m+1}\big\|_{L^1(I)} + \sum_{m=0}^{n-1} \big\|\widetilde{u}^{m+1,m+1}- \widetilde{u}^{m+1,m} \big\|_{L^1(I)} }\\
& \ \ {+ \sum_{m=0}^{n-1} \big\|\widetilde{u}^{n,m}- \widetilde{u}^{m+1,m} \big\|_{L^1(I)} + C_T \sigma_{\lambda}^{EI}(\Delta x)} \\
& \leq \sum_{m=0}^{n-1}\big\|\widetilde{u}^{m+1,m+1} - \widetilde{u}^{m+1,m} \big\|_{L^1(I)} + C_T \sigma_{\lambda}^{EI}(\Delta x)\\
& \leq n\Delta t\Delta x + C_T \sigma_{\lambda}^{EI}(\Delta x) 
\leq t_n\Delta x + C_T \sigma_{\lambda}^{EI}(\Delta x) \leq C \sigma_{\lambda}^{EI}(\Delta x),
\end{align*} 
where we have used the estimates \eqref{conv_estimate_implicit}. Now we focus on the work estimation of \eqref{eqn:fdm_impexp}.

It is well-known that the Newton method converges (locally) quadratically provided the initial approximation is in the small neighbourhood of the fixpoint. It is sufficient to perform $\mathcal{O}(\log(\log(\Delta t^{-1}\Delta x^{-1}))$ iterations to achieve the error bound $C\Delta t\Delta x$ in one time-step. Incorporating the CFL condition \eqref{CFL_2}, we need to perform the following number of iterations $(N)$:
\begin{align*}
N = 
\begin{cases}
\mathcal{O}(\log(\log(\Delta x^{-2})), & \quad \text{if }\lambda\in (0,1),\\[1.5mm]
\mathcal{O}(\log(\log(\Delta x^{-\lambda-1})), & \quad \text{if }\lambda\in (1,2).
\end{cases}
\end{align*}
The presence of the non-local diffusion term leads to the inversion of a full-matrix in each Newton iteration, which corresponds to $\mathcal{O}(\Delta x^{-3})$ floating point operations. Thus, the work estimate for one full Newton solve is given by
\begin{align*}
\mathcal{O}(\Delta x^{-3}\log(\log(\Delta x^{-1})), & \quad \text{if }\lambda\in (0,1)\cup (1,2)
\end{align*}
since $\mathcal{O}(\log(\log(\Delta x^{-2}))) = \mathcal{O}(\log(\log(\Delta x^{-1}))) = \mathcal{O}(\log(\log(\Delta x^{-\lambda-1}))$. Finally, since there are $T \Delta t^{-1}$ time-steps, we obtain the following work estimate for the explicit-implicit scheme \eqref{eqn:fdm_impexp}
\begin{align}\label{work_implicit}
W^{EI}_{\Delta} =
\begin{cases}
\mathcal{O}(\Delta x^{-4}\log(\log(\Delta x^{-1})), & \quad \text{if }\lambda\in (0,1),\\[1.5mm]
\mathcal{O}(\Delta x^{-3-\lambda}\log(\log(\Delta x^{-1})), & \quad \text{if }\lambda\in (1,2).
\end{cases}
\end{align}
\subsection{Application to random data}
We are interested in the following scalar random degenerate convection diffusion equation:
\begin{equation}\label{rfdcd_1d}
\begin{cases}
\partial_t u(\omega;t,x) +\partial_x f(\omega; u(\omega;t,x)) = - (-\triangle)^{\ld/2} [A(\omega;u(\omega;t,\cdot))](x), & \quad t>0,~x\in\R, \omega\in \Omega,\\
u(\omega;0,x)= u_0(\omega;x),& \quad x\in\R, \omega\in \Omega.
\end{cases}
\end{equation}
In order to develop MC-FDMs, we need to combine MC sampling to the FDMs \eqref{eqn:fdm_explicit} and \eqref{eqn:fdm_impexp} with random input data. FDMs incorporating random input data will be instrumental to perform the convergence analysis of the MC-FDM/MLMC-FDM algorithms.

Given a draw $(u_0(\omega;\cdot),f(\omega;\cdot),A(\omega;\cdot))$ of $(u_0,f,A)$, let $u_{\Delta}(\omega;t,x)$ define a family of grid function corresponding to the schemes \eqref{eqn:fdm_explicit} or \eqref{eqn:fdm_impexp} for \eqref{rfdcd_1d}. The following result consists of stability estimates and rate of convergence of approximate solutions in a random setup.
%%%%%% Proposition %%%%%%%%%%%%%%%%%%%%%%%%%%%%%%%%
\begin{prop}\label{discrete_random_prop}
Let us assume that $(u_0,f,A)\in L^2(\Omega;E_R^1)$, where $E_R^1$ is given by
\begin{align*}
E_R^1 = (BV(\R)\cap L^1(\R)\cap L^{\infty}(\R)) \times W^{1,\infty}(\R;\R) \times W^{1,\infty}(\R;\R).
\end{align*}
Consider the finite difference schemes \eqref{eqn:fdm_explicit} and \eqref{eqn:fdm_impexp} for the approximations of the entropy solution of \eqref{rfdcd_1d} corresponding to the random data $(u_0,f,A)(\omega)$.

Then, the random grid functions $\Omega\ni \omega\mapsto u_{\Delta}(\omega;t,x)$ defined by \eqref{num_sol} satisfy the following stability bounds for every $0<\widetilde{t}<\infty$ and $0<\Delta x<1$:
\begin{align*}
\|u_{\Delta}(\cdot;\widetilde{t},\cdot)\|_{L^2(\Omega;L^{\infty}(I))} & \leq \|u_0\|_{L^2(\Omega;L^{\infty}(I))},\\
\|u_{\Delta}(\cdot;\widetilde{t},\cdot)\|_{L^2(\Omega;L^{1}(I))} & \leq \|u_0\|_{L^2(\Omega;L^{1}(I))}.
\end{align*} 
Furthermore, we have the consistency bound: there exist a positive constant $C_T$ such that
\begin{align}\label{discrete_estimate}
\|u(\cdot;\widetilde{t},\cdot) - u_{\Delta}(\cdot;\widetilde{t},\cdot)\|_{L^2(\Omega;L^1(I))} \leq \|u_0 - u_{\Delta}(\cdot;0,\cdot)\|_{L^2(\Omega;L^1(I))} + C_T \sigma_{\lambda}(\Delta x),
\end{align} 
where $\sigma_{\lambda}(\Delta x)$ is defined by \eqref{sigma_defn}.
\end{prop}
Following the similar arguments in \cite{mishra2016numerical}, we can obtain the estimates in Proposition \ref{discrete_random_prop}. One can also refer to \cite{mishra2012sparse, mishra2013multi, zhang2014monte} for further details.
\begin{rem}($L^2$-estimates)\\
Under the assumptions of Proposition \ref{discrete_random_prop} and using H{\"o}lder's inequality, we have the following rate of convergence in $L^2(I)$,
\begin{align}\label{discrete_estimate_1}
\|u(\cdot;\widetilde{t},\cdot) - u_{\Delta}(\cdot;\widetilde{t},\cdot)\|_{L^2(\Omega;L^2(I))} 
\leq C \mathcal{M}^{1/2}\|u_0 - u_{\Delta}(\cdot;0,\cdot)\|^{1/2}_{L^2(\Omega;L^1(I))} 
+ C\mathcal{M}^{1/2} (C_T\sigma_{\lambda}(\Delta x))^{1/2},
\end{align}
where $\mathcal{M}$ is defined by \eqref{assum_1} and $\sigma_{\lambda}(\Delta x)$ is defined by \eqref{sigma_defn}.
\end{rem}
\begin{rem}
In a similar way to Proposition \ref{discrete_random_prop}, we can also obtain the following estimate:
\begin{align}\label{discrete_estimate_2}
\|u(\cdot;\widetilde{t},\cdot) - u_{\Delta}(\cdot;\widetilde{t},\cdot)\|_{L^1(\Omega;L^2(I))} 
\leq C \mathcal{M}^{1/2}\|u_0 - u_{\Delta}(\cdot;0,\cdot)\|^{1/2}_{L^1(\Omega;L^1(I))} 
+ C\mathcal{M}^{1/2} (C_T\sigma_{\lambda}(\Delta x))^{1/2},
\end{align}
since the assumption \eqref{assum_5} holds.
\end{rem}

% ------------------------------------------------------------------------------------------------------------------------------------------

\section{Multilevel Monte Carlo Finite Difference Method}
\label{mlmc-mc}
Our aim is to compute certain properties such as expectation, variance and higher moments of solution of \eqref{rfdcd_1d}. In order to do this, we have to discretize the stochastic domain $\omega\in \Omega$ as well as physical domain $(t,x)\in Q_T$. There are several approaches that one can follow. A popular approach is the use of Stochastic Galerkin methods with generalized polynomial chaos (see \cite{chen2005uncertainty, lin2004UQ, wan2006pc,burger2014uq, poette2009uq,  tryoen2012uq} and references therein). However, these methods are highly intrusive, requiring the restructuring of existing deterministic numerical codes. Stochastic collocation methods \cite{xiu2005high} provide an alternative class of methods which are non-intrusive. Both stochastic Galerkin and stochastic collocation methods can suffer from deterioration in performance due to loss of regularity of the solution with respect to the stochastic variable. We focus on yet another class of methods based on statistical sampling methods, to quantify the uncertainty in numerical solutions. In particular, we consider Monte Carlo (MC) sampling and multi-level Monte Carlo (MLMC).  These methods are non-intrusive and easy to parallelize.
\subsection{Monte Carlo Method}
Let us assume that for $\mathbb{P}$-a.s. $\omega$ the data $(u_0(\omega;\cdot), f(\omega;\cdot), A(\omega;\cdot))\in E_R^1$, and the assumptions \eqref{assum_1}-\eqref{assum_5} hold. We wish to statistically estimate $\Ex[u]$, which is the expectation (or first moment) of $u$. The MC approximation of $\Ex[u]$ is defined as follows:
\begin{defi}
Given $M$ independent, identically distributed (i.i.d.) samples $(\widehat{u}_0^i,\widehat{f}^i,\widehat{A}^i),i=1,2,...,M$, of initial data, flux function and diffusion operator, the MC estimate $\Ex[u(\cdot;t,\cdot)]$ at time $t$ is given by
\begin{align}\label{MC_estimate}
E_M [u(t,\cdot)]: = \frac{1}{M}\sum_{i=1}^M \widehat{u}^i(t,\cdot),
\end{align}
where $\widehat{u}^i(t,\cdot)$ corresponds to the unique entropy solution for the $i$-th data sample.
\end{defi}
%%%%%%%%%%%%% Lemma %%%%%%%%%%%%%%%
\begin{lem}
If the samples $\{(\widehat{u}_0^i,\widehat{f}^i,\widehat{A}^i),i=1,2,...,M\}$ are i.i.d., then
$\Ex\big[\|E_M[u(t,\cdot)]\|_{L^2(\R)}\big]$ is finite.
\end{lem}
\begin{proof}
Since we observe that
\begin{align*}
\widehat{u}^i (t,\cdot) = S(t)(\widehat{u}_0^i,\widehat{f}^i,\widehat{A}^i),
\end{align*}
for every $M$ and for every $0<t<\infty$,
\begin{align*}
\|E_M[u(\omega; t,\cdot)]\|_{L^2(\R)} & = \Big\|\frac{1}{M}\sum_{i=1}^M S(t)(\widehat{u}_0^i,\widehat{f}^i,\widehat{A}^i)\Big\|_{L^2(\R)}\\
& \leq \frac{1}{M}\sum_{i=1}^M \|S(t)(\widehat{u}_0^i,\widehat{f}^i,\widehat{A}^i)(\omega)\|_{L^2(\R)}
 \leq \frac{1}{M} \sum_{i=1}^M \|\widehat{u}_0^i(\omega;\cdot)\|_{L^2(\R)},
\end{align*}
where we have taken into account the estimate \eqref{p-estimate}. As a consequence, we have
\begin{align*}
\Ex\big[\|E_M[u(\omega;t,\cdot)]\|_{L^2(\R)}\big] \leq \Ex\Big[\frac{1}{M}\sum_{i=1}^M \|\widehat{u}_0^i(\omega;\cdot)\|_{L^2(\R)} \Big]
= \Ex [\|u_0\|_{L^2(\R)}] = \|u_0\|_{L^1(\Omega;L^2(\R))} <\infty.
\end{align*}
\end{proof}
%%%%%% Theorem %%%%%%%%%%%%%%%
\begin{thm}\label{main_estimate}
Let us consider the equation \eqref{rfdcd_1d} in which the random variable $(u_0,f,A)(\omega)$ satisfies the assumption \eqref{assum_1}-\eqref{assum_5}. Moreover, assume that $u_0\in L^2(\Omega;L^2(\R))$.
Then the MC approximations $E_M[u(t,\cdot)]$ defined in \eqref{MC_estimate} converges in $L^2(\Omega;L^2(\R))$ as $M\rightarrow\infty$, to $ \Ex[u(t,\cdot)]$. In addition, for any $M\in\mathbb{N}$, $0<t<\infty$, there holds the error bound
\begin{equation}\label{MC_error_bnd}
\|\Ex[u(t,\cdot)] - E_M[u(t,\cdot)]\|_{L^2(\Omega;L^{2}(\R))} \leq C M^{-1/2}\|u_0\|_{L^2(\Omega;L^2(\R))}.
\end{equation}
\end{thm}
\begin{proof}
Consider the $M$ i.i.d. samples $\{(\widehat{u}_0^i,\widehat{f}^i,\widehat{A}^i),i=1,2,...,M\}$ and $\widehat{u}^i (t,\cdot) = S(t)(\widehat{u}_0^i,\widehat{f}^i,\widehat{A}^i)$, for $i=1,2,...,M$.
We have the following equality using the linearity of expectation
\begin{align*}
\|\Ex[u(t,\cdot)] - E_M[u(t,\cdot)]\|^2_{L^2(\Omega;L^{2}(\R))} & = \Ex\Big[\|\Ex[u(t,\cdot)] - E_M[u(t,\cdot)]\|^2_{L^{2}(\R)} \Big] \\
& = \Ex\Big[\Big\|\frac{1}{M}\sum_{i=1}^M \big(\Ex[u(t,\cdot)] - \widehat{u}^i(t,\cdot)\big) \Big\|^2_{L^2(\R)} \Big].
\end{align*}
For convenience, denote $\Ex[u(t,\cdot)] - \widehat{u}^i(t,\cdot)$ as $Y_i$. Observe that $Y_i$ are i.i.d. random variables with zero mean. Hence we have
\begin{align*}
\Ex\Big[\Big\|\frac{1}{M}\sum_{i=1}^M \big(\Ex[u(t,\cdot)] - \widehat{u}^i(t,\cdot)\big) \Big\|^2_{L^2(\R)} \Big]
= \Ex\Big[\Big\|\frac{1}{M}\sum_{i=1}^M Y_i \Big\|^2_{L^2(\R)}\Big].
\end{align*}
A crucial estimate for $L^p$ spaces [\cite{koley2013multilevel}, Corollary 2.5] gives us
\begin{align*}
\Ex\Big[\Big\|\frac{1}{M}\sum_{i=1}^M Y_i \Big\|^2_{L^2(\R)}\Big] \leq C M^{-1} \Ex\Big[\|\Ex[u(t,\cdot)] - u(t,\cdot)\|^2_{L^2(\R)} \Big]
\leq CM^{-1} \Ex\big[\|u(t,\cdot)\|^2_{L^2(\R)} \big].
\end{align*}
With the help of estimate \eqref{p-estimate} we obtain the required error bound
\begin{align*}
CM^{-1} \Ex\big[\|u(t,\cdot)\|^2_{L^2(\R)} \big] \leq CM^{-1}\Ex\big[\|u_0\|^2_{L^2(\R)} \big] = CM^{-1}\|u_0\|^2_{L^2(\Omega;L^2(\R))}.
\end{align*}
\end{proof}
\subsection{MC-FDM}
We combine the MC tools with the finite difference methods (referred to as MC-FDM) to approximate statistical quantities associated with the solution of the non-local equation \eqref{rfdcd_1d}. The main idea of MC-FDM is to generate independent samples of initial data, flux function and diffusion operator and then, for each sample to perform an FD simulation. For the remainder of this work, we restrict the discussions to the bounded interval $I$ instead of an unbounded domain.

\begin{defi}(Statistical estimates for random entropy solutions)\\
Consider the initial value problem \eqref{rfdcd_1d} with random data $(u_0,f,A)$ satisfying \eqref{assum_1}-\eqref{assum_5}. Given $M\in\mathbb{N}$, generate $M$ i.i.d. samples $\{(\widehat{u}_0^i,\widehat{f}^i,\widehat{A}^i),i=1,2,...,M\}$. Let $\{\widehat{u}^i(t,\cdot)\}_{i=1}^M$ denote the unique entropy solution of  \eqref{rfdcd_1d} corresponding to the data sample $(\widehat{u}_0^i,\widehat{f}^i,\widehat{A}^i)$. Then the MC-FDM approximation of $\Ex[u(t,.)]$ is defined as the statistical estimate of the ensemble $\{\widehat{u}^i_{\Delta}(t,\cdot)\}_{i=1}^M$ obtained from the FD approximation either by \eqref{eqn:fdm_explicit} or by \eqref{eqn:fdm_impexp} with data samples $\{(\widehat{u}_0^i,\widehat{f}^i,\widehat{A}^i),i=1,2,...,M\}$. More precisely, the first moment of the random solution $u(\omega;t,\cdot)$ at time $t>0$, is estimated as
\begin{align}\label{MC_estimate_defn}
\Ex[u(t,.)] \thickapprox E_M[u_{\Delta}(t,\cdot)]:=\frac{1}{M}\sum_{i=1}^M \widehat{u}^i_{\Delta}(t,\cdot).
\end{align}
\end{defi}
\subsubsection{Convergence analysis of MC-FDM}
We analyze the convergence of $E_M[u_{\Delta}(t,\cdot)]$ to the mean $\Ex[u(t,\cdot)]$. In order to do this, we have the following result concerning the error bound. It is important to note that the error with the MC-FDM approach is due to statistical/sampling error and discretization error.
%%%%%%%%%%%%%%%%% Theorem %%%%%%%%%%%%%%%%%%%%%%%%%%%%%
\begin{thm}(MC-FDM Error bound)
Let us assume that $I$ is a bounded interval and the assumptions \eqref{assum_1}-\eqref{assum_5} hold. Furthermore, assume that
\begin{align*}
u_0\in L^2(\Omega; L^1(I)\cap BV(I)\cap L^{\infty}(I))
\end{align*}
and the deterministic FD schemes \eqref{eqn:fdm_explicit}-\eqref{eqn:fdm_impexp} converge at rate $\sigma_{\lambda}(\Delta x)$ in $L^{\infty}(0,T;L^1(I))$ for every $0<T<\infty$, where $\sigma_{\lambda}(\Delta x)$ is defined in \eqref{sigma_defn}.
Then, for every $M$, the MC estimate $E_M[u_{\Delta}(t,\cdot)]$ defined in \eqref{MC_estimate_defn} satisfies the following error bound:
\begin{equation}\label{MC_error}
\begin{split}
\|\Ex[u(t,\cdot)] -  E_M[u_{\Delta}(\omega;t,\cdot)]\|_{L^2(\Omega;L^2(I))} 
& \leq C \Big\{ \mathcal{M}^{1/2}\|u_0 - u_{\Delta}(\cdot;0,\cdot)\|_{L^2(\Omega;L^1(I))}\\
& + M^{-1/2}\|u_0\|_{L^2(\Omega;L^2(I))}
+ \mathcal{M}^{1/2}\sigma_{\lambda}(\Delta x)^{1/2} \Big\},
\end{split}\end{equation}
where $\mathcal{M}$ is defined in \eqref{assum_1} and the non-negative constant $C$ is independent of $M$ and $\Delta x$.
\end{thm}
\begin{proof}
For arbitrary $t>0$, using the triangle inequality we obtain
\begin{align*}
\|\Ex[u(t,\cdot)] -  E_M[u_{\Delta}(\omega;t,\cdot)]\|_{L^2(\Omega;L^2(I))} 
\leq & \|\Ex[u(\omega;t,\cdot)] -  E_M[u(\omega;t,\cdot)]\|_{L^2(\Omega;L^2(I))} \\
& + \|E_M[u(\omega;t,\cdot)] - E_M[u_{\Delta}(\omega;t,\cdot)]\|_{L^2(\Omega;L^2(I))}\\
=&: T_1 + T_2.
\end{align*}
Using the Theorem \ref{main_estimate}, we have the following estimate for $T_1$
\begin{align}\label{MC_estim_1}
T_1 \leq C M^{-1/2} \|u_0\|_{L^2(\Omega;L^2(I))}.
\end{align}
Thus, we focus on the term $T_2$. Noting the linearity of the estimator $E_M[\cdot]$, we obtain
\begin{align*}
T_2 & = \|E_M[u(\omega;t,\cdot)] - E_M[u_{\Delta}(\omega;t,\cdot)]\|_{L^2(\Omega;L^2(I))}\\
& = \Big\| \frac{1}{M}\sum_{i=1}^M \widehat{u}^i(\omega;t,\cdot) - \frac{1}{M}\sum_{i=1}^M \widehat{u}^i_{\Delta}(\omega;t,\cdot) \Big\|_{L^2(\Omega;L^2(I))}\\
& \leq \|u(\omega;t,\cdot) - u_{\Delta}(\omega;t,\cdot)\|_{L^2(\Omega;L^2(I))}\\
& \leq \|u(\omega;t,\cdot) - u_{\Delta}(\omega;t,\cdot)\|^{1/2}_{L^2(\Omega;L^1(I))} 
\|u(\omega;t,\cdot) - u_{\Delta}(\omega;t,\cdot)\|^{1/2}_{L^2(\Omega;L^{\infty}(I))}, 
\end{align*}
using Proposition \ref{discrete_random_prop} and \eqref{assum_1}. With the help of the error estimate \eqref{discrete_estimate}, we obtain
\begin{align}\label{MC_estim_2}
T_2 \leq C \mathcal{M}^{1/2} (\|u_0 - u_{\Delta}(\cdot;0,\cdot)\|^{1/2}_{L^2(\Omega;L^1(I))} + C_T \sigma_{\lambda}(\Delta x)^{1/2}).
\end{align}
Combining the estimates \eqref{MC_estim_1} and \eqref{MC_estim_2}, we have the required result.
\end{proof}
\subsubsection{Work estimates}
Our next aim is to obtain the work estimates for MC-FDM with the explicit scheme \eqref{eqn:fdm_explicit} as well as for the explicit-implicit scheme \eqref{eqn:fdm_impexp}. We have observed that as $\Delta x,\Delta t\rightarrow 0$, the computational work estimate for the explicit scheme \eqref{eqn:fdm_explicit} is asymptotically bounded as
\begin{equation*}
W_{\Delta}^{EX} = 
\begin{cases}
\mathcal{O}(\Delta x^{-3}), & \quad \text{if } \lambda \in (0,1),\\
\mathcal{O}(\Delta x^{-\lambda -2}), & \quad \text{if }\lambda \in (1,2).
\end{cases}
\end{equation*}
Hence the work for the computation of the MC estimate $E_M[u_{\Delta}(t,\cdot)]$ is of order
\begin{align}\label{MC_work_1}
W_{\Delta,M}^{EX} \leq
\begin{cases}
CM \Delta x^{-3}, & \quad \text{if } \lambda\in (0,1),\\
CM \Delta x^{-\lambda -2}, & \quad \text{if }\lambda \in (1,2).
\end{cases}
\end{align}
We compute the convergence order in terms of work from the estimate \eqref{MC_error}. To this end we equilibrate in \eqref{MC_error} the two bounds by choosing
\begin{align}\label{eqn:MCsamples}
\begin{cases}
M^{-1/2}\backsim (\Delta x)^{1/4}, \text{ i.e. } M = C\Delta x^{-1/2}, \quad & \lambda\in (0,2/3],\\
M^{-1/2}\backsim (\Delta x)^{\frac{2-\lambda}{2(2+\lambda)}}, \text{ i.e. } M = C\Delta x^{-\frac{2-\lambda}{2+\lambda}},
 \quad & \lambda\in (2/3,1)\cup (1,2).
\end{cases}
\end{align}
Inserting $M$ in \eqref{MC_work_1} yields $W_{\Delta,M}^{EX} \leq \Theta^{EX}_{\lambda}(\Delta x)$, where the function $\Theta^{EX}_{\lambda}(\Delta x)$ is given by
\begin{align}
\Theta^{EX}_{\lambda}(\Delta x) = 
\begin{cases}
C \Delta x^{-3-\frac{1}{2}}, \quad & \lambda \in (0,2/3],\\
C \Delta x^{-3- \frac{2-\lambda}{2+\lambda}}, \quad & \lambda\in(2/3,1),\\
C \Delta x^{-\lambda-2-\frac{2-\lambda}{2+\lambda}}, \quad & \lambda\in (1,2).
\end{cases}
\end{align}
It is straightforward to observe that $\|u_0 - u_{\Delta}(\cdot;0,\cdot)\|_{L^2(\Omega;L^1(I))}$ is of $\mathcal{O}(\Delta x)$,
and as a consequence, the error estimate \eqref{MC_error} becomes
\begin{align*}
\|\Ex[u(t,\cdot)] -  E_M[u_{\Delta}(\omega;t,\cdot)]\|_{L^2(\Omega;L^2(I))} \leq C\sigma^{EX}_{\lambda}(\Delta x)^{1/2},
\end{align*} 
where the constant $C$ is independent of $M$ and $\Delta x$. Hence we obtain
\begin{align}\label{MC_error_mod_explicit}
\|\Ex[u(t,\cdot)] -  E_M[u_{\Delta}(\omega;t,\cdot)]\|_{L^2(\Omega;L^2(I))} \leq 
\begin{cases}
C (W_{\Delta,M}^{EX})^{-\frac{1}{14}}, \quad & \lambda\in(0,2/3],\\[1.5mm]
C(W_{\Delta,M}^{EX})^{-\frac{2-\lambda}{6(2+\lambda)+2(2-\lambda)}}, \quad & \lambda\in(2/3,1),\\[1.5mm]
C(W_{\Delta,M}^{EX})^{-\frac{(2-\lambda)}{2(2+\lambda)^2  + 2(2-\lambda)}}, \quad & \lambda\in (1,2).
\end{cases}
\end{align}
Next we carry out the similar analysis for the explicit-implicit scheme \eqref{eqn:fdm_impexp}. Taking into account the work estimate \eqref{work_implicit}, we obtain the computational work for the MC estimate $E_M[u_{\Delta}(t,\cdot)]$
\begin{align}\label{MC_work_2}
W^{EI}_{\Delta,M} \leq
\begin{cases}
CM \Delta x^{-4}\log(\log(\Delta x^{-1})), & \quad \text{if } \lambda\in (0,1),\\[1.5mm]
CM \Delta x^{-3 -\lambda}\log(\log(\Delta x^{-1})), & \quad \text{if }\lambda\in (1,2).
\end{cases}
\end{align}
In order to equilibrate the terms in the estimate \eqref{MC_error}, we choose
\begin{align*}
\begin{cases}
M^{-1/2}\backsim (\Delta x)^{\frac{1}{4}}, \text{ i.e. } M = C\Delta x^{-\frac{1}{2}}, \quad & \lambda\in (0,1),\\[1.5mm]
M^{-1/2}\backsim (\Delta x)^{\frac{2-\lambda}{4}}, \text{ i.e. } M = C\Delta x^{-\frac{2-\lambda}{2}}, \quad & \lambda\in (1,2).
\end{cases}
\end{align*}
which in turn gives $W^{EI}_{\Delta,M}\leq \Theta^{EI}_{\lambda}(\Delta x)$, where
\begin{align}\label{wrk_temp}
\Theta^{EI}_{\lambda}(\Delta x) = 
\begin{cases}
C \Delta x^{-4-\frac{1}{2}}\log(\log(\Delta x^{-1})), \quad & \lambda \in (0,1),\\[1.5mm]
C \Delta x^{-3-\lambda - \frac{2-\lambda}{2}}\log(\log(\Delta x^{-1})), \quad & \lambda\in (1,2).
\end{cases}
\end{align}
Hence we have
\begin{align*}
\|\Ex[u(t,\cdot)] -  E_M[u_{\Delta}(\omega;t,\cdot)]\|_{L^2(\Omega;L^2(I))} \leq C\sigma^{EI}_{\lambda}(\Delta x)^{1/2}.
\end{align*} 
After incorporating \eqref{sigma_defn_implicit} in \eqref{wrk_temp} we get
\begin{align}\label{MC_error_mod_implicit}
\|\Ex[u(t,\cdot)] -  E_M[u_{\Delta}(\omega;t,\cdot)]\|_{L^2(\Omega;L^2(I))} \leq 
\begin{cases}
C \Big(W_{\Delta,M}^{EI}(\log(W_{\Delta,M}^{EI}))^{-1}\Big)^{-\frac{1}{18}}, \quad & \lambda\in(0,1),\\[1.5mm]
C \Big(W_{\Delta,M}^{EI}(\log(W_{\Delta,M}^{EI}))^{-1}\Big)^{-\frac{(2-\lambda)}{4(3+\lambda) + 2(2-\lambda)}}, \quad & \lambda\in (1,2),
\end{cases}
\end{align}
where we have used the estimate $\log(\log (\Delta x^{-1}))\leq \log (\Delta x^{-1})$  assuming that the space discretization $\Delta x \ll 1$. The constant $C$ may depend on $u_0$ or $p$ but is independent of $M$ and $\Delta x$. 
\begin{rem}
In the deterministic setup, the convergence rate for the explicit scheme \eqref{eqn:fdm_explicit} and explicit-implicit scheme \eqref{eqn:fdm_impexp} with respect to work read
\begin{align}\label{DM_estim_explicit}
\|u(t,\cdot) -  u_{\Delta}(t,\cdot)\|_{L^2(I)} \leq 
\begin{cases}
C (W_{\Delta}^{EX})^{-\frac{1}{12}}, \quad & \lambda\in(0,2/3],\\[1.5mm]
C(W_{\Delta}^{EX})^{-\frac{2-\lambda}{6(2+\lambda)}}, \quad & \lambda\in(2/3,1),\\[1.5mm]
C(W_{\Delta}^{EX})^{-\frac{(2-\lambda)}{2(2+\lambda)^2}}, \quad & \lambda\in (1,2).
\end{cases}
\end{align}
and
\begin{align}\label{DM_estim_implicit}
\|u(t,\cdot) -  u_{\Delta}(t,\cdot)\|_{L^2(I)} \leq 
\begin{cases}
C \Big(W_{\Delta}^{EI}(\log(W_{\Delta}^{EI}))^{-1}\Big)^{-\frac{1}{16}}, \quad & \lambda\in(0,1),\\[1.5mm]
C \Big(W_{\Delta}^{EI}(\log(W_{\Delta}^{EI}))^{-1}\Big)^{-\frac{(2-\lambda)}{4(3+\lambda)}}, \quad & \lambda\in (1,2).
\end{cases}
\end{align}
respectively.
It is straightforward to observe that the asymptotic efficiency (in terms of overall error vs work) of MC-FDM (with the explicit scheme \eqref{eqn:fdm_explicit} as well as the explicit-implicit scheme \eqref{eqn:fdm_impexp}) is, in general, inferior to the deterministic scheme \eqref{eqn:fdm_explicit}.
\end{rem}

\subsection{Multilevel MC-FDM}
In order to achieve an accuracy versus time bound for the stochastic FDM which lies closer to the bound \eqref{DM_estim_explicit} and \eqref{DM_estim_implicit} corresponding to the deterministic problem, we turn towards analyzing the multilevel Monte Carlo finite difference method (MLMC-FDM). The main idea behind the MLMC scheme is the simultaneous MC sampling on different levels of mesh resolution of the FDM, with $M_l$ denoting the number of samples on level $l$. We also determine the number of samples required in each level.
\begin{defi}(MLMC-FDM)
The MLMC-FDM is defined as a multilevel discretization in $x$ and $t$ with level dependent numbers of samples, denoted by $M_l$. Due to the presence of non-local operator (representation involves principle value function), we consider a family of nested grids with cell sizes
\begin{align}\label{MLMC_discretization}
\Delta x_l = 3^{-l}\Delta x_0,\qquad l\in\mathbb{N}_0 = \{0\}\cup\mathbb{N},
\end{align}
for some $\Delta x_0>0$. This ensures that each mesh contains a cell centered at $x=0$. Similarly, we denote the time-step size $\Delta t_l$ for the explicit and explicit-implicit schemes corresponding to grid size $\Delta x_l$ at level $l$. The time-step is determined by the CFL condition
\begin{align*}
\Delta t_l = C \Delta x_l^{1\vee\lambda}.
\end{align*}
The approximate solution of \eqref{fdcd_1d} computed by the scheme \eqref{eqn:fdm_explicit} or \eqref{eqn:fdm_impexp} on the grid with cell and time-step size 
$\Delta_l:=(\Delta t_l,\Delta x_l)$ is denoted by $u_l$.
\end{defi}
\subsubsection{Derivation of MLMC-FDM}
Our aim is to estimate the ensemble average i.e., $\Ex[u(t,\cdot)]$, $0<t<\infty$ of the random entropy solution of \eqref{rfdcd_1d} with the random samples $(u_0,f,A)(\omega),\omega\in\Omega$, satisfying \eqref{assum_1}-\eqref{assum_5}. As was done for MC-FDM, the expectation $\Ex[u(t,\cdot)]$ in MLMC will be estimated by approximating $u(t,\cdot)$ with the help of the FDMs.

Let $\{u_l(t,\cdot)\}_{l=0}^{\infty}$ denote the sequence of approximations of solutions of \eqref{rfdcd_1d} on the nested meshes with cell sizes $\Delta x_l$, time-steps of sizes $\Delta t_l$. Then, for a prescribed target level $L\in\mathbb{N}$ of spatial resolution, we have
\begin{align}\label{MLMC_expec}
\Ex[u_L(t,\cdot)] = \Ex\left[\sum_{l=0}^{L} \big(u_l(t,\cdot)-u_{l-1}(t,\cdot) \big) \right],
\end{align}
where we have set $u_{-1}(t,\cdot)=0$ and used the linearity of the expectation operator. Furthermore, we estimate each term in \eqref{MLMC_expec} statistically by a MC method with level dependent number of samples $M_l$. This leads to the MLMC-FDM estimator
\begin{align}
E^L[u(t,\cdot)] = \sum_{l=0}^L E_{M_l}[u_l(t,\cdot) - u_{l-1}(t,\cdot)],
\end{align}
where $E_M[u_{\Delta}(t,\cdot)]$ is evaluated by \eqref{MC_estimate_defn}.
\subsubsection{Convergence analysis}
We wish to analyze the MLMC-FDM mean field error given by
\begin{align}\label{MLMC_error}
\|\Ex[u(t,\cdot)] - E^L[u(t,\cdot)]\|_{L^2(\Omega;L^2(I))}, \quad 0<t<\infty,~L\in\mathbb{N}.
\end{align}
Our aim is to choose the appropriate sample sizes $\{M_l\}_{l=0}^{\infty}$ such that for every $L\in\mathbb{N}$, the MLMC error \eqref{MLMC_error} is minimized. The principal issue in the design of MLMC-FDM is the optimal choice of $\{M_l\}_{l=0}^{\infty}$ such that for each $L$, an error \eqref{MLMC_error} is achieved with minimal total work which is given as follows:\\
for the explicit scheme \eqref{eqn:fdm_explicit},
\begin{align}\label{work_estim_MLMC_explicit}
W^{EX}_{L,MLMC} = C\sum_{l=0}^{L} M_l W^{EX}_{\Delta_l} = 
\begin{cases}
\mathcal{O}\left(\displaystyle\sum_{l=0}^L M_l \Delta x_l^{-3}\right), \quad & \lambda\in (0,1),\\[1.5mm]
\mathcal{O}\left(\displaystyle\sum_{l=0}^L M_l \Delta x_l^{-\lambda-2}\right), \quad & \lambda \in (1,2),
\end{cases}
\end{align}
for the explicit-implicit scheme \eqref{eqn:fdm_impexp},
\begin{align}\label{work_estim_MLMC_implicit}
W^{EI}_{L,MLMC} = C\sum_{l=0}^{L} M_l W^{EI}_{\Delta_l} = 
\begin{cases}
\mathcal{O}\left(\displaystyle\sum_{l=0}^L M_l \Delta x_l^{-4}\big|\log(\log(\Delta x_l^{-1}))\big|\right), \quad & \lambda\in (0,1),\\[1.5mm]
\mathcal{O}\left(\displaystyle\sum_{l=0}^L M_l \Delta x_l^{-3-\lambda} \big|\log(\log(\Delta x_l^{-1}))\big|\right), \quad & \lambda\in (1,2)
\end{cases}
\end{align}
which are based on \eqref{MC_work_1} and \eqref{MC_work_2} respectively.

%%%%%%%%%%%%%%%%%%%%%%%%%%%%%%%%%%%%%%%%%%%%%%%%%%%%%%%%%
We now establish the following result on MLMC error bounds \eqref{MLMC_error}:
%%%%%%%%%%%%%%%%%%% Theorem %%%%%%%%%%%%%%%%%%%%%%%%%%%%%%%%%%%%%%%
\begin{thm}\label{error_bnd_thm}
Consider the multilevel discretization \eqref{MLMC_discretization} along with the assumptions \eqref{assum_1}-\eqref{assum_5}. Furthermore, consider any sequence of sample sizes $\{M_l\}_{l=0}^{\infty}$ at mesh level $l$. Then, we have the following error bounds for the MLMC-FDM estimate in \eqref{MLMC_error}: for the explicit scheme \eqref{eqn:fdm_explicit},
\begin{equation}\label{error bound explicit}
\begin{split}
\| & \Ex[u(t,\cdot)] - E^L[u(t,\cdot)]\|^2_{L^2(\Omega;L^2(I))} \leq \\ &
\begin{cases}
C M_0^{-1} \|u_0\|^2_{L^2(\Omega;L^2(I))} 
+ C \mathcal{M} \Big\{\Delta x_L \||u_0|_{BV(I)}\|^{2}_{L^2(\Omega)} 
+ \Delta x_L^{1/2} \Big\} \\
+ C\mathcal{M}\Big\{\displaystyle\sum_{l=1}^L M_l^{-1}\Delta x_l^{1/2} \Big\}
\big(1 + \||u_0|_{BV(I)}\|_{L^2(\Omega)} \big), & \quad \lambda\in(0,2/3],\\[2mm]
C M_0^{-1} \|u_0\|^2_{L^2(\Omega;L^2(I))}  
+ C \mathcal{M} \Big\{\Delta x_L \||u_0|_{BV(I)}\|^{2}_{L^2(\Omega)} + \Delta x_L^{\frac{2-\lambda}{2+\lambda}} \Big\}\\
+ C\mathcal{M}\Big\{\displaystyle\sum_{l=1}^L M_l^{-1}\Delta x_l^{\frac{2-\lambda}{2+\lambda}} \Big\}
\big(1 + \||u_0|_{BV(I)}\|_{L^2(\Omega)} \big),  & \quad \lambda\in(2/3,1)\cup(1,2),
\end{cases}
\end{split}
\end{equation}
and for the explicit-implicit scheme \eqref{eqn:fdm_impexp},
\begin{equation}\label{error bound implicit}
\begin{split}
\| & \Ex[u(t,\cdot)] - E^L[u(t,\cdot)]\|^2_{L^2(\Omega;L^2(I))} \leq \\ &
\begin{cases}
C M_0^{-1} \|u_0\|^2_{L^2(\Omega;L^2(I))} 
+ C \mathcal{M} \Big\{\Delta x_L \||u_0|_{BV(I)}\|^{2}_{L^2(\Omega)} 
+ \Delta x_L^{1/2} \Big\} \\
+ C\mathcal{M}\Big\{\displaystyle\sum_{l=1}^L M_l^{-1}\Delta x_l^{1/2} \Big\}
\big(1 + \||u_0|_{BV(I)}\|_{L^2(\Omega)} \big), & \quad \lambda\in(0,1),\\[2mm]
C M_0^{-1} \|u_0\|^2_{L^2(\Omega;L^2(I))}  
+ C \mathcal{M} \Big\{\Delta x_L \||u_0|_{BV(I)}\|^{2}_{L^2(\Omega)} + \Delta x_L^{\frac{2-\lambda}{2}} \Big\}\\
+ C\mathcal{M}\Big\{\displaystyle\sum_{l=1}^L M_l^{-1}\Delta x_l^{\frac{2-\lambda}{2}} \Big\}
\big(1 + \||u_0|_{BV(I)}\|_{L^2(\Omega)} \big),  & \quad \lambda\in(1,2),
\end{cases}
\end{split}
\end{equation}
where the constant $C>0$ is independent of the parameters $l$, $\{M_{l}\}_{l=0}^{\infty}$, and $\Delta x_l$ but may depend on $t$, $u_0$, $f$, $A$, and size of the domain $I$. 
\end{thm}
\begin{proof}
Using the linearity of mathematical expectation $\Ex[\cdot]$ and applying the triangle inequality, we obtain
\begin{align*}
\|\Ex & [u(t,\cdot)] - E^L[u(t,\cdot)]\|^2_{L^2(\Omega;L^2(I))} \\
\leq & C \|\Ex[u(t,\cdot)] - \Ex[u_L(t,\cdot)]\|^2_{L^2(\Omega;L^2(I))}
 + C \|\Ex[u_L(t,\cdot)] - E^L[u(t,\cdot)]\|^2_{L^2(\Omega;L^2(I))}\\
= & C \|\Ex[u(t,\cdot)] - \Ex[u_L(t,\cdot)]\|^2_{L^2(\Omega;L^2(I))}
 + C \Big\|\sum_{l=0}^L \Ex[u_l - u_{l-1}] - E_{M_l}[u_l - u_{l-1}] \Big\|^2_{L^2(\Omega;L^2(I))}\\
= & C \|\Ex[u(t,\cdot)] - \Ex[u_L(t,\cdot)]\|^2_{L^2(\Omega;L^2(I))}
 + C\|\Ex[u_{l=0}] - \Ex_{M_0}[u_{l=0}]\|^2_{L^2(\Omega;L^2(I))}\\
& \qquad \qquad 
+ C \Big\|\sum_{l=1}^L \Ex[u_l - u_{l-1}] - E_{M_l}[u_l - u_{l-1}] \Big\|^2_{L^2(\Omega;L^2(I))}\\
= & : T_1 + T_2 + T_3,
\end{align*}
where we have used the definition of MLMC estimator \eqref{MLMC_expec} and $u_{-1}=0$. To estimate \eqref{MLMC_error}, we consider the terms $T_1$, $T_2$ and $T_3$ separately. With the help of linearity of expectation, term $T_1$ can be estimated as
\begin{align*}
T_1 & = C \|\Ex[u(t,\cdot)] - \Ex[u_L(t,\cdot)]\|^2_{L^2(\Omega;L^2(I))}\\
& = C \|\Ex[u(t,\cdot) - u_L(t,\cdot)]\|^2_{L^2(\Omega;L^2(I))}\\
& = C \|u(t,\cdot) - u_L(t,\cdot)\|^2_{L^1(\Omega;L^2(I))}, 
\end{align*}
which is bounded by \eqref{discrete_estimate_2}. The bound for term $T_2$ is given by \eqref{MC_error_bnd}. Finally, we focus on the term $T_3$. Taking into account the definition \eqref{MC_estimate_defn}, we obtain
\begin{align*}
T_3 & \leq C \Big\|\sum_{l=1}^L \sum_{i=1}^{M_l} \frac{1}{M_l} 
\big(\Ex[u_l - u_{l-1}] - (\widehat{u}^i_l - \widehat{u}^i_{l-1})\big) \Big\|^2_{L^2(\Omega;L^2(I))}\\
& = C \Big\|\sum_{l=1}^L \sum_{i=1}^{M_l} Z_{i,l} \Big\|^2_{L^2(\Omega;L^2(I))},
\end{align*}
where $Z_{i,l}$ is given by 
\begin{align*}
Z_{i,l} = \frac{1}{M_l} \big(\Ex[u_l - u_{l-1}] - (\widehat{u}^i_l - \widehat{u}^i_{l-1})\big), \quad i=1,2,3,\dots,M_l,\quad l=1,2,3,\dots,L.
\end{align*}
Observe that $Z_{i,l}$ are independent, mean zero random variables. For each fixed level $l$, the random variables $Z_{i,l}$, $i=1,...,M_l$ have an identical distribution. In other words, we have
\begin{align*}
\|Z_{i,l}\|^2_{L^2(\Omega;L^2(I))}&=\|Z_{1,l}\|^2_{L^2(\Omega;L^2(I))}, \quad i=1,...,M_l.
\end{align*}
Furthermore, since the space $L^2$ is a Banach space of type 2 (see \cite{koley2013multilevel}) , we get the following estimate 
\begin{align*}
\Big\|\sum_{l=1}^L \sum_{i=1}^{M_l} Z_{i,l} \Big\|^2_{L^2(\Omega;L^2(I))}
& \leq C \sum_{l=1}^L \sum_{i=1}^{M_l}\|Z_{i,l}\|^2_{L^2(\Omega;L^2(I))}
 = C \sum_{l=1}^L M_l \|Z_{1,l}\|^2_{L^2(\Omega;L^2(I))}\\ 
& = C \sum_{l=1}^L M_l \Big\|\frac{1}{M_l}\big(\Ex[u_l - u_{l-1}] - (\widehat{u}_l^1 - \widehat{u}_{l-1}^1) \big) \Big\|^2_{L^2(\Omega;L^2(I))} \\
& = C \sum_{l=1}^L M_l^{-1} \|\Ex[u_l - u_{l-1}] - (\widehat{u}_l^1 - \widehat{u}_{l-1}^1)\|^2_{L^2(\Omega;L^2(I))}\\
& \leq C \sum_{l=1}^L M_l^{-1} \|u_l - u_{l-1}\|^2_{L^2(\Omega;L^2(I))}.
\end{align*} 
Now we estimate 
\begin{align*}
\|u_l - u_{l-1}\|_{L^2(\Omega;L^2(I))} \leq \|u(t,\cdot) - u_{l}(t,\cdot)\|_{L^2(\Omega;L^2(I))} 
+ \|u(t,\cdot) - u_{l-1}(t,\cdot)\|_{L^2(\Omega;L^2(I))}.
\end{align*}
Both terms on the right hand side can be estimated by \eqref{discrete_estimate_1}. Hence we end up with
\begin{align*}
\|u_l - u_{l-1}\|_{L^2(\Omega;L^2(I))} \leq & C\mathcal{M}^{1/2} \Big(\|u_0-u_l(\cdot;0,\cdot)\|^{1/2}_{L^2(\Omega;L^1(I))}\\
& \quad + \|u_0-u_{l-1}(\cdot;0,\cdot)\|^{1/2}_{L^2(\Omega;L^1(I))} \Big) + C \mathcal{M}^{1/2} \sigma_{\lambda}(\Delta x_l)^{1/2}, 
\end{align*}
where $\mathcal{M}$ is defined by \eqref{assum_1}. We can approximate the first two terms on the right hand side as follows:
\begin{align*}
\|u(\cdot;0,\cdot)-u_l(\cdot;0,\cdot)\|_{L^2(\Omega;L^1(I))} \leq \Delta x_l \||u_0|_{BV(I)}\|_{L^2(\Omega)}.
\end{align*}
Thus, we have
\begin{align*}
\|u(\cdot;0,\cdot)-u_l(\cdot;0,\cdot)\|^{1/2}_{L^2(\Omega;L^1(I))} \leq \Delta x_l^{1/2} \||u_0|_{BV(I)}\|^{1/2}_{L^2(\Omega)}.
\end{align*}
Subsequently, we obtain
\begin{align*}
\|u_l - u_{l-1}\|_{L^2(\Omega;L^2(I))} & \leq C\mathcal{M}^{1/2} \big(\Delta x_l^{1/2} \||u_0|_{BV(I)}\|^{1/2}_{L^2(\Omega)} \big) 
+ C \mathcal{M}^{1/2} \sigma_{\lambda}(\Delta x_l)^{1/2}.
\end{align*}
More precisely, for the explicit scheme \eqref{eqn:fdm_explicit}, we have
\begin{align*}
\|u_l - u_{l-1}\|_{L^2(\Omega;L^2(I))} & \leq 
\begin{cases}
C \mathcal{M}^{1/2} \big(1 + \Delta x_l^{\frac{1}{4}}\||u_0|_{BV(I)}\|^{1/2}_{L^2(\Omega)} \big) \Delta x_l^{\frac{1}{4}}, \quad & \lambda\in(0,2/3],\\[1.5mm]
C \mathcal{M}^{1/2} \big(1 + \Delta x_l^{\frac{\lambda}{2+\lambda}}\||u_0|_{BV(I)}\|^{1/2}_{L^2(\Omega)} \big) 
\Delta x_l^{\frac{2-\lambda}{2(2+\lambda)}}, \quad & \lambda\in(2/3,1)\cup(1,2),
\end{cases}
\end{align*}
and similarly, for the numerical scheme \eqref{eqn:fdm_impexp},
\begin{align*}
\|u_l - u_{l-1}\|_{L^2(\Omega;L^2(I))} & \leq 
\begin{cases}
C \mathcal{M}^{1/2} \big(1 + \Delta x_l^{\frac{1}{4}}\||u_0|_{BV(I)}\|^{1/2}_{L^2(\Omega)} \big) \Delta x_l^{\frac{1}{4}}, \quad & \lambda\in(0,1),\\[1.5mm]
C \mathcal{M}^{1/2} \big(1 + \Delta x_l^{\frac{\lambda}{4}}\||u_0|_{BV(I)}\|^{1/2}_{L^2(\Omega)} \big) 
\Delta x_l^{\frac{2-\lambda}{4}}, \quad & \lambda\in (1,2).
\end{cases}
\end{align*}
Taking into account $\Delta x_l \leq \mathcal{O}(1)$ we have for the scheme \eqref{eqn:fdm_explicit},
\begin{align}\label{MLMC_estim_1_expl}
\|u_l - u_{l-1}\|_{L^2(\Omega;L^2(I))} \leq
\begin{cases}
C \mathcal{M}^{1/2} \big(1 + \||u_0|_{BV(I)}\|_{L^2(\Omega)} \big)^{\frac{1}{2}} \Delta x_l^{\frac{1}{4}}, \quad & \lambda\in(0,2/3],\\[1.5mm]
C \mathcal{M}^{1/2} \big(1 + \||u_0|_{BV(I)}\|_{L^2(\Omega)} \big)^{\frac{1}{2}} \Delta x_l^{\frac{2-\lambda}{2(2+\lambda)}}, \quad & \lambda\in(2/3,1)\cup(1,2),
\end{cases}
\end{align}
and for the scheme \eqref{eqn:fdm_impexp},
\begin{align}\label{MLMC_estim_1_impl}
\|u_l - u_{l-1}\|_{L^2(\Omega;L^2(I))} \leq
\begin{cases}
C \mathcal{M}^{1/2} \big(1 + \||u_0|_{BV(I)}\|_{L^2(\Omega)} \big)^{\frac{1}{2}} \Delta x_l^{\frac{1}{4}}, \quad & \lambda\in(0,1),\\[1.5mm]
C \mathcal{M}^{1/2} \big(1 + \||u_0|_{BV(I)}\|_{L^2(\Omega)} \big)^{\frac{1}{2}} \Delta x_l^{\frac{2-\lambda}{4}}, \quad & \lambda\in (1,2).
\end{cases}
\end{align}
By substituting \eqref{MLMC_estim_1_expl} in the estimates of term $T_3$ and summing it over $l=1,2,3,\dots,L$, we end up with
\begin{align*}
T_3 \leq 
\begin{cases}
C\mathcal{M}\Big\{\displaystyle\sum_{l=1}^L M_l^{-1}\Delta x_l^{1/2} \Big\}
\big(1 + \||u_0|_{BV(I)}\|_{L^2(\Omega)} \big), & \quad \lambda\in(0,2/3],\\
C\mathcal{M}\Big\{\displaystyle\sum_{l=1}^L M_l^{-1}\Delta x_l^{\frac{2-\lambda}{2+\lambda}} \Big\}
\big(1 + \||u_0|_{BV(I)}\|_{L^2(\Omega)} \big),  & \quad \lambda\in(2/3,1)\cup(1,2).
\end{cases}
\end{align*}
Finally using $\Delta x_l\leq \Delta x_0\leq \mathcal{O}(1)$ and adding the contribution with the estimates emerging from the terms $T_1$ and $T_2$, we obtain the prescribed error bound \eqref{error bound explicit} for the explicit scheme \eqref{eqn:fdm_explicit}. We can perform a similar analysis for the explicit-implicit scheme \eqref{eqn:fdm_impexp} to obtain the MLMC error bound \eqref{error bound implicit}.
\end{proof}

\subsection{Optimizing the number of samples on each level}
Our analysis to determine Monte Carlo samples sizes $\{M_l\}_{l=0}^{\infty}$ will be based on the error bound \eqref{error bound explicit}-\eqref{error bound implicit}. By optimizing number of samples, we mean to determine the number of samples needed to minimize the computational work with the constraints that the error tolerance is $\varepsilon_{er}$. By adapting the approach in \cite{koley2013multilevel}, our argument will make use of Lagrange multipliers. Theorem \ref{error_bnd_thm} will be instrumental to obtain the following lemma.
%%%%%%%%%%%%%%%%%%%%%%%% Lemma %%%%%%%%%%%%%%%%%%%%%%%%%%%%%%
\begin{lem}
Let the multilevel discretization be given by $\Delta x_l = 3^{-l}\Delta x_0$ for some $\Delta x_0>0$. Consider the work estimates given by \eqref{work_estim_MLMC_explicit} and \eqref{work_estim_MLMC_implicit} for the schemes \eqref{eqn:fdm_explicit} and \eqref{eqn:fdm_impexp} respectively. Let $\Theta$ be the order of convergence of the schemes \eqref{eqn:fdm_explicit} and \eqref{eqn:fdm_impexp}. Based on Corollary \ref{rate_conv_p}, $\Theta$ is precisely given by
\begin{align*}
\Theta = 
\begin{cases}
\text{ for explicit scheme; }\Theta_{EX} = 
\begin{cases}
\frac{1}{4}, \quad \lambda\in (0,2/3],\\[1.5mm]
\frac{2-\lambda}{2(2+\lambda)}, \quad \lambda\in (2/3,1)\cup(1,2);
\end{cases}
\\
\text{ for explicit-implicit scheme; }\Theta_{EI} = 
\begin{cases}
\frac{1}{4}, \quad \lambda\in (0,1),\\[1.5mm]
\frac{2-\lambda}{4},\quad \lambda\in (1,2).
\end{cases}
\end{cases}
\end{align*}
Assume that $L$ and $\Delta x_0$ are chosen such that $\Delta x_L^{2\Theta-r} > \Delta x_0^{-r}$ for some $r \geq 3$. Given an error tolerance $\varepsilon_{er} > 0$, the MLMC-FDM error in a compact form scales as
\begin{align}\label{err_MLMC}
Error_L = \left\| \Ex[u(t,\cdot)] - E^L[u(t,\cdot)] \right\|^2_{L^2(\Omega;L^2(I))} \approx C \left(M_0^{-1} + \Delta x_L^{2\Theta} + \sum_{l=1}^L M_l^{-1}\Delta x_l^{2\Theta} \right),
\end{align}
where the constant $C>0$ is independent of $\Delta x_l$, but depends on $u_0$, $f$, $A$, 
$\mathcal{M}$ and size of the domain $I$.
Furthermore, the optimal sample numbers with respect to the work estimate \eqref{work_estim_MLMC_explicit} and with respect to the error bound \eqref{err_MLMC}, are given by
\begin{align}\label{sample_no_expl}
M_l^{EX} & \backsimeq M_0^{EX} \Delta x_0^{\Theta} 3^{-l \Big(\Theta
+\frac{r^{EX}_{\lambda}}{2} \Big)}, \qquad l=1,2,3,\dots,L
\end{align}
for the explicit scheme \eqref{eqn:fdm_explicit}, where $M_0^{EX}$ is given by
\begin{align}\label{sample_no_expl_0}
M_0^{EX}\backsimeq \left[\frac{1}{\varepsilon_{er}^{EX} - \Delta x_0^{2\Theta}
3^{-2\Theta L}}
\left(1 + \Delta x_0^{\Theta}\sum_{j=1}^L 3^{j\left(\frac{r^{EX}_{\lambda}}{2}-\Theta \right)}
 \right) \right],
\end{align}
and for the scheme \eqref{eqn:fdm_impexp},
\begin{align}\label{sample_no_impl}
M_l^{EI} & \backsimeq \frac{\log(\Delta x_0^{-1})^{\frac{1}{2}} \Delta x_0^{\Theta} 3^{-l\big(\Theta
+\frac{r^{EI}_{\lambda}}{2}\big)}}
{\Big(l\log 3+\log(\Delta x_0^{-1}) \Big)^{\frac{1}{2}}} M_0^{EI}, \qquad l=1,2,3,\dots,L,
\end{align}
where $M_0^{EI}$ is given by
\begin{equation}\label{sample_no_impl_0}
\begin{aligned}
M_0^{EI} & \backsimeq \displaystyle\frac{1}{\log(\Delta x_0^{-1})^{\frac{1}{2}}}\\
& \qquad\left[\displaystyle\frac{1}{\varepsilon_{er}^{EI} 
- \Delta x_0^{2\Theta}3^{-2\Theta L}} 
\left(\log(\Delta x_0^{-1})^{\frac{1}{2}} + \Delta x_0^{\Theta} \displaystyle\sum_{j=1}^L 
3^{j \left(\frac{r^{EI}_{\lambda}}{2}-\Theta \right)}
\Big(j\log 3 +\log(\Delta x_0^{-1})\Big)^{\frac{1}{2}} \right) \right].
\end{aligned}
\end{equation}
Here $\backsimeq$ indicates that this is the number of samples up to a constant which may depend on the data $(u_0,f,A)$ and the domain, but not on the sample sizes on the various levels. Finally, $r_{\lambda}^{EX}$ and $r_{\lambda}^{EI}$ are given by 
\begin{align*}
r_{\lambda}^{EX} = 
\begin{cases}
3 , & \quad \lambda\in (0,1), \\
\lambda +2, & \quad \lambda\in (1,2),
\end{cases}
\end{align*}
\begin{align*}
r_{\lambda}^{EI} = 
\begin{cases}
4 , & \quad \lambda\in (0,1), \\
3 +\lambda, & \quad \lambda\in (1,2)
\end{cases}
\end{align*}
for the schemes \eqref{eqn:fdm_explicit} - \eqref{eqn:fdm_impexp} respectively.

As $L\rightarrow\infty$, the error of the MLMC-FDM algorithm with respect to work is given by
\begin{align}
\|\Ex[u(t,\cdot)] - E^L[u(t,\cdot)]\|^2_{L^2(\Omega;L^2(I))} \leq 
\begin{cases}\label{eqn:MLMC_err_exp}
C \left(W_{L,MLMC}^{EX} \right)^{-\frac{1}{6}}, \quad & \lambda\in(0,2/3],\\[1.5mm]
C\left(W_{L,MLMC}^{EX}\right)^{-\frac{2-\lambda}{3(2+\lambda)}}, \quad & \lambda\in(2/3,1),\\[1.5mm]
C\left(W_{L,MLMC}^{EX}\right)^{-\frac{(2-\lambda)}{(2+\lambda)^2}}, \quad & \lambda\in (1,2)
\end{cases}
\end{align}
for the explicit scheme \eqref{eqn:fdm_explicit}, and 
\begin{align}\label{eqn:MLMC_err_imp}
\|\Ex[u(t,\cdot) - E^L[u(t,\cdot)]\|^2_{L^2(\Omega;L^2(I))}\leq
\begin{cases}
C\left(W^{EI}_{L,MLMC}\Big(\log \left(W^{EI}_{L,MLMC} \right)\Big)^{-1} \right)^{-\frac{1}{8}}, \quad & \lambda\in(0,1),\\[1.5mm]
C\left(W^{EI}_{L,MLMC}
\Big(\log \left(W^{EI}_{L,MLMC}\right)\Big)^{-1} \right)^{-\frac{(2-\lambda)}{2(3+\lambda)}}, \quad & \lambda\in (1,2)
\end{cases}
\end{align}
for the explicit-implicit scheme \eqref{eqn:fdm_impexp}.
\end{lem}
%%%%%%%%%%%%%%% Proof %%%%%%%%%%%%%%%%%%%%%%%%%%%%%%%%%%%%%%%%
\begin{proof}
With the mesh discretization given by $\Delta x_l = 3^{-l} \Delta x_0$, the MLMC-FDM work estimates \eqref{work_estim_MLMC_explicit} and \eqref{work_estim_MLMC_implicit} become
\begin{align*}
W^{EX}_{L,MLMC} & = C\sum_{l=0}^L M_l\Delta x_l^{-r_{\lambda}^{EX}} = C\Delta x_0^{-r_{\lambda}^{EX}}\sum_{l=0}^L M_l 3^{r_{\lambda}^{EX}l},\\
W^{EI}_{L,MLMC} & = C\sum_{l=0}^L M_l\Delta x_l^{-r_{\lambda}^{EI}} \log(\log(\Delta x_l^{-1})) \leq C\Delta x_0^{-r_{\lambda}^{EI}}
\sum_{l=0}^L M_l 3^{r_{\lambda}^{EI}l} \Big(l\log 3 + \log(\Delta x_0^{-1})\Big). 
\end{align*}
Incorporating \eqref{MLMC_discretization}, the bound for the multi-level errors \eqref{error bound explicit} and \eqref{error bound implicit} at level $L$ becomes
\begin{align*}
Error_L & = C \left(M_0^{-1} + \Delta x_L^{2\Theta} + \sum_{l=1}^L M_l^{-1}\Delta x_l^{2\Theta} \right)\\
& = C \left(M_0^{-1} + \Delta x_0^{2\Theta} 3^{-2\Theta L} + \Delta x_0^{2\Theta}\sum_{l=1}^L M_l^{-1} 3^{-2\Theta l} \right).
\end{align*}
Using a Lagrange multiplier $\beta$, we consider the following Lagrangian by incorporating error tolerance 
$\varepsilon_{er}$
\begin{align*}
\mathcal{L} := W_{L,MLMC} -\beta (\varepsilon_{er} - Error_L).
\end{align*}
Consequently, the first order optimality condition provides
\begin{align*}
\frac{\partial \mathcal{L}}{\partial M_l} = 0, \qquad l= 0,1,2,...,L.
\end{align*}
This implies, for the explicit scheme \eqref{eqn:fdm_explicit},
\begin{align*}
\begin{cases}
\Delta x_0^{-r^{EX}_{\lambda}} = \beta^{EX} (M_0^{EX})^{-2} , & l=0,\\
\Delta x_0^{-r^{EX}_{\lambda}} 3^{r^{EX}_{\lambda}l} = \beta^{EX} \Delta x_0^{2\Theta} 
(M_l^{EX})^{-2} 3^{-2\Theta l}, & l=1,2,3,\dots,L;
\end{cases}
\end{align*}
and for the explicit-implicit scheme \eqref{eqn:fdm_impexp},
\begin{align*}
\begin{cases}
\Delta x_0^{-r^{EI}_{\lambda}}\log(\Delta x_0^{-1}) = \beta^{EI} (M_0^{EX})^{-2} , & l=0,\\
\Delta x_0^{-r^{EI}_{\lambda}} 3^{r^{EI}_{\lambda}l} \Big(l\log 3 +\log(\Delta x_0^{-1})\Big) 
= \beta^{EI} \Delta x_0^{2\Theta} (M_l^{EX})^{-2} 3^{-2\Theta l}, & l=1,2,3,\dots,L,
\end{cases}
\end{align*}
where the multiplier $\beta^{EX}$ for the explicit scheme and $\beta^{EI}$ for the explicit-implicit scheme are independent of the level $l$. This leads us to the sample numbers with the scheme \eqref{eqn:fdm_explicit},
\begin{align}\label{eqn:MLMC_err_exp_beta}
\begin{cases}
M_0^{EX} = (\beta^{EX})^{\frac{1}{2}} \Delta x_0^{\frac{r^{EX}_{\lambda}}{2}},\\
M_l^{EX}  = (\beta^{EX})^{\frac{1}{2}} \big(\Delta x_0 3^{-l}\big)^{\Theta +\frac{r^{EX}_{\lambda}}{2}}, \quad l=1,2,3,\dots,L;
\end{cases}
\end{align}
and similarly with the scheme \eqref{eqn:fdm_impexp},
\begin{align}\label{eqn:MLMC_err_imp_beta}
\begin{cases}
M_0^{EI} = \left(\displaystyle\frac{\beta^{EI}}{\log \left(\Delta x_0^{-1}\right)} \right)^{\frac{1}{2}} \Delta x_0^{\frac{r^{EI}_{\lambda}}{2}},\\
M_l^{EI}  = \left(\displaystyle\frac{\beta^{EI}}{l\log 3 +\log(\Delta x_0^{-1})} \right)^{\frac{1}{2}} \big(\Delta x_0 3^{-l}\big)^{\Theta + \frac{r^{EI}_{\lambda}}{2}}, \quad l=1,2,3,\dots,L.
\end{cases}
\end{align}
Using the constraint $Error_L = \varepsilon_{er}^{EX}$ for explicit scheme \eqref{eqn:fdm_explicit} and $Error_L = \varepsilon_{er}^{EI}$ for \eqref{eqn:fdm_impexp}, we deduce the following
\begin{align*}
\varepsilon_{er}^{EX} & \approx \Delta x_0^{2\Theta} 3^{-2\Theta L} + \frac{\Delta x_0^{-\frac{r^{EX}_{\lambda}}{2}}}{(\beta^{EX})^{\frac{1}{2}}} 
\left(1 + \Delta x_0^{\Theta} \sum_{l=1}^L 3^{l \left(\frac{r^{EX}_{\lambda}}{2}
-\Theta \right)} \right),\\
\varepsilon_{er}^{EI} & \approx \Delta x_0^{2\Theta} 3^{-2\Theta L} 
+ \frac{\Delta x_0^{-\frac{r^{EI}_{\lambda}}{2}}}{(\beta^{EI})^{\frac{1}{2}}} 
\left(\log(\Delta x_0^{-1})^{\frac{1}{2}} + \Delta x_0^{\Theta} 
\sum_{l=1}^L 3^{l \left(\frac{r^{EI}_{\lambda}}{2}-\Theta \right)}
\Big(l\log 3 + \log(\Delta x_0^{-1}) \Big)^{\frac{1}{2}} \right)
\end{align*}
which in turn provides the following expressions for the Lagrange multipliers
\begin{align*}
\beta^{EX} & = \frac{1}{\Delta x_0^{r^{EX}_{\lambda}}} \left[\frac{1}{\varepsilon_{er}^{EX} - \Delta x_0^{2\Theta} 3^{-2\Theta L}}
\left( 1 + \Delta x_0^{\Theta} \sum_{l=1}^L 3^{l \big(\frac{r^{EX}_{\lambda}}{2}
-\Theta\big)} \right)\right]^2, \\
\beta^{EI} & = \frac{1}{\Delta x_0^{r^{EI}_{\lambda}}} \Bigg[\frac{1}{\varepsilon_{er}^{EI}
- \Delta x_0^{2\Theta} 3^{-2\Theta L}}
\left( \log(\Delta x_0^{-1})^{\frac{1}{2}} + \Delta x_0^{\Theta} \sum_{l=1}^L 3^{l \big(\frac{r^{EI}_{\lambda}}{2} - \Theta\big)}
\Big(l\log 3 + \log(\Delta x_0^{-1})\Big)^{\frac{1}{2}} \right)\Bigg]^2.
\end{align*}
Using these expression of the multipliers in \eqref{eqn:MLMC_err_exp_beta} and \eqref{eqn:MLMC_err_imp_beta} leads to the optimal number of samples \eqref{sample_no_expl} - \eqref{sample_no_impl_0}.

As a consequence, the work estimates become
\begin{align*}
W^{EX}_{L,MLMC} \backsimeq & \left[\frac{1}{\varepsilon_{er}^{EX} 
- \Delta x_0^{2\Theta}3^{-2\Theta L}} 
\left(1 + \Delta x_0^{\Theta} \sum_{j=1}^L 3^{j\left(\frac{r^{EX}_{\lambda}}{2}-\Theta \right)} \right) \right] \Delta x_0^{-r^{EX}_{\lambda}}
\left(1 + \Delta x_0^{\Theta} \sum_{l=1}^L 
3^{l\left(\frac{r^{EX}_{\lambda}}{2}-\Theta\right)} \right);\\
W^{EI}_{L,MLMC} \backsimeq & \left[\frac{1}{\varepsilon_{er}^{EI} 
- \Delta x_0^{2\Theta}3^{-2\Theta L}} 
\left(\log(\Delta x_0^{-1})^{\frac{1}{2}} + \Delta x_0^{\Theta} 
\sum_{j=1}^L 3^{j \left(\frac{r^{EI}_{\lambda}}{2}-\Theta \right)} 
\Big(j\log 3+\log(\Delta x_0^{-1}) \Big)^{\frac{1}{2}}\right) \right] \\
& \quad \Delta x_0^{-r^{EI}_{\lambda}} \left(\log(\Delta x_0^{-1})^{\frac{1}{2}}
+ \Delta x_0^{\Theta} \sum_{l=1}^L 3^{l\left(\frac{r^{EI}_{\lambda}}{2}-\Theta\right)} 
\Big(l\log 3+\log(\Delta x_0^{-1})\Big)^{\frac{1}{2}}\right).
\end{align*}
Let us observe that for the explicit scheme \eqref{eqn:fdm_explicit},
\begin{align*}
\frac{r^{EX}_{\lambda}}{2}-\Theta =
\begin{cases}
\displaystyle\frac{5}{4}, & \quad \lambda\in (0,2/3];\\
\displaystyle\frac{4(1+\lambda)}{2(2+\lambda)}, & \quad \lambda\in (2/3,1);\\
\displaystyle\frac{\lambda^2 + 5\lambda+2}{2(2+\lambda)}, & \quad \lambda\in (1,2);
\end{cases}
\end{align*}
and subsequently,
\begin{align*}
\frac{r^{EX}_{\lambda}}{2}-\Theta > 0, \text{ for all } \lambda\in (0,2).
\end{align*}
Thus, $W^{EX}_{L,MLMC}$ will be dominated by the terms $3^{L\big(\frac{r^{EX}_{\lambda}}{2}-\Theta\big)}$. Let us choose the error tolerance for the explicit scheme \eqref{eqn:fdm_explicit} as
\begin{align*}
\varepsilon_{er}^{EX}= 2\Delta x_0^{2\Theta} 3^{-2\Theta L}= 2\Delta x_L^{2\Theta}.
\end{align*}
Then the work estimate for explicit MLMC becomes of the order
\begin{align*}
W^{EX}_{L,MLMC} & \backsimeq  \Delta x_L^{-2\Theta} \Delta x_0^{-r^{EX}_{\lambda}} 
\left(1+ \Delta x_0^{\Theta} 
3^{L \big(\frac{r^{EX}_{\lambda}}{2} - \Theta\big)} \right)^2\\
& \backsimeq \Delta x_L^{-2\Theta} \left( \Delta x_0^{-\frac{r^{EX}_{\lambda}}{2}} 
+ \Delta x_0^{\Theta-\frac{r^{EX}_{\lambda}}{2}}
3^{L\big(\frac{r^{EX}_{\lambda}}{2}- \Theta\big)} \right)^2\\
& \backsimeq \Delta x_L^{-2\Theta} \left(\Delta x_0^{-r^{EX}_{\lambda}} 
+ \Delta x_L^{2\Theta - r^{EX}_{\lambda}} \right).
\end{align*}
Assuming that $\Delta x_0$ and $L$ can be chosen such that $\Delta x_L^{2\Theta - r^{EX}_{\lambda}}>\Delta x_0^{- r^{EX}_{\lambda}}$, the work estimate simplifies to
\begin{align}\label{opt_exp_workbnd}
W^{EX}_{L,MLMC} & \backsimeq  \Delta x_L^{-2\Theta}\Delta x_L^{2\Theta - r^{EX}_{\lambda}}
= \Delta x_L^{- r^{EX}_{\lambda}}.
\end{align}
Inserting \eqref{opt_exp_workbnd} into the asymptotic error bound we obtain 
\begin{align*}
\varepsilon_{er}^{EX} \backsimeq \Delta x_L^{2\Theta} \backsimeq  
\left({W^{EX}_{L,MLMC}} \right)^{- \frac{2\Theta}{r^{EX}_{\lambda}}}.
\end{align*}
Incorporating the values of $\Theta$ and $r^{EX}_{\lambda}$ for the explicit scheme \eqref{eqn:fdm_explicit}, we have the following error estimate in terms of work
\begin{align}\label{opt_exp_rate}
\|\Ex[u(t,\cdot) - E^L[u(t,\cdot)]\|^2_{L^2(\Omega;L^2(I))}\leq
\begin{cases}
C \Big(W_{L,MLMC}^{EX} \Big)^{-\frac{1}{6}}, \quad & \lambda\in(0,2/3],\\[1.5mm]
C \Big(W_{L,MLMC}^{EX} \Big)^{-\frac{(2-\lambda)}{3(2+\lambda)}}, \quad & \lambda\in(2/3,1),\\[1.5mm]
C \Big(W_{L,MLMC}^{EX} \Big)^{-\frac{(2-\lambda)}{(2+\lambda)^2}}, \quad & \lambda\in (1,2).
\end{cases}
\end{align}

Next, let us consider the explicit-implicit scheme \eqref{eqn:fdm_impexp} and observe that
\begin{align*}
\frac{r^{EI}_{\lambda}}{2}-\Theta =
\begin{cases}
\displaystyle\frac{7}{4}, & \quad \lambda\in (0,1);\\
\displaystyle\frac{4+ 3\lambda}{4}, & \quad \lambda\in (1,2);
\end{cases}
\end{align*}
and hence $W^{EI}_{L,MLMC}$ will be dominated by the terms $3^{L\big(\frac{r^{EX}_{\lambda}}{2}-\Theta\big)}$.
By choosing the error tolerance $\varepsilon_{er}^{EI} = 2\Delta x_L^{2\Theta}$, we obtain the work estimates of the order
\begin{align*}
W^{EI}_{L,MLMC} \backsimeq & \Delta x_L^{-2\Theta} \Delta x_0^{-r^{EI}_{\lambda}}
\left(\log(\Delta x_0^{-1})^{\frac{1}{2}}
+ \Delta x_0^{\Theta} \Big(L + \log(\Delta x_0^{-1}) \Big)^{\frac{1}{2}} 
3^{L\big(\frac{r^{EI}_{\lambda}}{2} -\Theta\big)} \right)^2\\
\backsimeq & \Delta x_L^{-2\Theta} \left(\log(\Delta x_0^{-1})^{\frac{1}{2}}\Delta x_0^{-\frac{r^{EI}_{\lambda}}{2}}
+ \Delta x_0^{\Theta-\frac{r^{EI}_{\lambda}}{2}}\Big(L + \log(\Delta x_0^{-1}) \Big)^{\frac{1}{2}} 
3^{L \big(\frac{r^{EI}_{\lambda}}{2}-\Theta\big)} \right)^2\\
\backsimeq & \Delta x_L^{-2\Theta} \left(\log(\Delta x_0^{-1}) \Delta x_0^{-r^{EI}_{\lambda}} 
+ \log(\Delta x_L^{-1})\Delta x_L^{2\Theta-r^{EI}_{\lambda}} \right).
\end{align*}
If we choose $\Delta x_0$ and $L$ such that
\begin{align*}
\log(\Delta x_0^{-1}) \Delta x_0^{-r^{EI}_{\lambda}} 
< \log(\Delta x_L^{-1})\Delta x_L^{2\Theta-r^{EI}_{\lambda}},
\end{align*}
then the work is asymptotically dominated by
\begin{align}\label{opt_imp_workbnd}
W^{EI}_{L,MLMC} \backsimeq \log(\Delta x_L^{-1})\Delta x_L^{-r^{EI}_{\lambda}}.
\end{align}
Consequently, by inserting \eqref{opt_imp_workbnd} into the asymptotic error bound, we get
\begin{align*}
\varepsilon_{er}^{EI} \backsimeq \Delta x_L^{2\Theta} \backsimeq \left(W^{EI}_{L,MLMC}
\Big(\log \left(W^{EI}_{L,MLMC} \right)\Big)^{-1} \right)^{-\frac{2\Theta}{r_{\lambda}^{EI}}}.
\end{align*}
Finally, incorporating the values of $\Theta$ and $r_{\lambda}^{EI}$ for the explicit-implicit scheme \eqref{eqn:fdm_impexp}, we have the following error estimate in terms of the work
\begin{align}\label{opt_imp_rate}
\|\Ex[u(t,\cdot) - E^L[u(t,\cdot)]\|^2_{L^2(\Omega;L^2(I))}\leq
\begin{cases}
C\left(W^{EI}_{L,MLMC}
\Big(\log \left(W^{EI}_{L,MLMC} \right)\Big)^{-1} \right)^{-\frac{1}{8}}, \quad & \lambda\in(0,1),
\\[1.5mm]
C\left(W^{EI}_{L,MLMC}
\Big(\log \left(W^{EI}_{L,MLMC} \right)\Big)^{-1} \right)^{-\frac{(2-\lambda)}{2(3+\lambda)}}, \quad & \lambda\in (1,2),
\end{cases}
\end{align}
and the result follows.
\end{proof}

\begin{rem}(Comparison of rates)\\
In the case of $A \equiv 0$ (hyperbolic conservation laws), it is observed in \cite{mishra2016numerical} that the convergence rates, in terms of accuracy vs. work, are considerably reduced for MC-FDM in comparison to the deterministic case. However, the convergence rates significantly improve with multilevel Monte Carlo approach. Similarly, for the case of degenerate convection-diffusion equation ($\lambda=2$) in \cite{koley2013multilevel}, it is demonstrated that the obtained convergence rates of MLMC-FDM is better than single level Monte Carlo, even though these rates are worse when compared with the deterministic schemes. We list some of these theoretical estimates in Table \ref{tab:errvswork_lit}.
\begin{table}[htbp]
\begin{tabular}{|c|c|c|c|}
\hline
\textbf{Model} & \textbf{Base scheme} & \textbf{\begin{tabular}[c]{@{}c@{}}Scaling for \\ MC-FDM\end{tabular}} & \textbf{\begin{tabular}[c]{@{}c@{}}Scaling for \\ MLMC-FDM\end{tabular}} \\ \hline
\begin{tabular}[c]{@{}c@{}}hyperbolic conservation laws \cite{mishra2016numerical}\\ $A \equiv 0$\end{tabular} & Explicit scheme of order $s$ & ${W}^{-\frac{s}{2 + 2s}}$ & ${W}^{-\frac{s}{2 + s}}$                                                                         \\ \hline
\multirow{2}{*}{\begin{tabular}[c]{@{}c@{}}degenerate convection-diffusion \cite{koley2013multilevel}\\ $\lambda=2$\end{tabular}} & \begin{tabular}[c]{@{}c@{}}Explicit scheme with \\ monotone flux\end{tabular} &  ${W}^{-\frac{1}{13}}$ & ${W}^{-\frac{1}{6}}$ \\ \cline{2-4} 
& \begin{tabular}[c]{@{}c@{}}Implicit scheme with \\ monotone flux\end{tabular} & ${\left(\frac{W}{\log(W)}\right)}^{-\frac{1}{7}}$ & ${\left(\frac{W}{\log(W)}\right)}^{-\frac{2}{7}}$ \\ \hline
\end{tabular}
\caption{Existing error vs. work scaling for MC-FDM and MLMC-FDM schemes in some limiting cases. Here $W$ denotes the work estimate for the underlying method used.}
\label{tab:errvswork_lit}
\end{table}
 
In the current paper, error vs. work bounds rates are obtained for various values of $\lambda\in (0,2)$. For the explicit scheme, the convergence rates of MC-FDM are worse than the deterministic scheme (refer to \eqref{MC_error_mod_explicit}), which is expected. However, the rates improved significantly with MLMC-FDM, c.f. \eqref{opt_exp_rate}. The same behaviour is observed for the explicit-implicit scheme, c.f. \eqref{MC_error_mod_implicit} and \eqref{opt_imp_rate}.
\end{rem}

% ------------------------------------------------------------------------------------------------------------------------------------------

\section{Numerical results}\label{results}
We now numerically test the performance of the MLMC-FDMs proposed in this work. We set the underlying target model to be the one-dimensional Buckley-Leverett equation describing a two-phase flow through porous media. Let $u(x,t) \in [0,1]$ represent the water saturation in an oil-water mixture. Then $u(x,t)$ can be modeled by the convection-diffusion equation \eqref{fdcd}, with the numerical flux
\begin{equation}\label{eqn:bl_flux}
f(u) = \frac{\mathcal{P}^w(u)}{\mathcal{P}^w(u) + \mu \mathcal{P}^o(u)}\ , \qquad A(u) = \max(u-\alpha,0)\ , \qquad \alpha \geq 0\ ,
\end{equation}
where $\mathcal{P}^w, \mathcal{P}^o$ are the relative permeability of the water and oil phase, respectively, while $\mu>0$ is the ratio of phase viscosities. In our experiments, we set
\[
\mathcal{P}^w(u) = u^2, \qquad \mathcal{P}^o(u) = (1-u)^2.
\]
Note that the non-local diffusion in \eqref{eqn:bl_flux} does not necessarily approximate the physical diffusion term for two-phase flows. However, our objective is to demonstrate the performance of the proposed methods and validate the expected convergence rates. Thus, we adhere to the choice of $A(u)$ given in \eqref{eqn:bl_flux}. 

Before presenting the numerical results, we briefly discuss a few additional approximations that need to be made for a practical implementation of the various algorithms discussed so far.

\subsection{Finite computational domain and boundary extensions}\label{sec:frac_diss}
Since it is not feasible to work with an infinite number of nodes, we focus on a finite domain with a suitable extension of the solution. In particular, we consider the computational domain to be a symmetric interval $I = [-K,K]$, which is uniformly discretized using $N^x$ cells with a mesh size $\Delta x = 2 K/ N^x$. Based on the definition of the grid points/cell-centers in Section \ref{numerical}, the mesh contains a cell centered at $x_0 = 0$. Thus, we need an odd number of cells in the mesh, i.e., $N^x = 2P+1$, which leads to the following cell-centers and the cell-interfaces on the finite domain
\begin{equation*}
\begin{aligned}
-K + \frac{\Delta x}{2} = x_{-P} < x_{-P+1} < ... x_0 &= 0 < ... < x_{P-1} < x_{P} = K - \frac{\Delta x}{2}, \\
-K = x_{-P-\frac{1}{2}} < x_{-P+\frac{1}{2}} < &... < x_{P-\frac{1}{2}} < x_{P+\frac{1}{2}} = K.
\end{aligned}
\end{equation*}
To approximate the the non-local term \eqref{dif_integralform}, we need to suitably extend the solution beyond the domain $[-K,K]$. For the purpose of this work, we assume that the solution can be extended in a constant manner beyond the original domain, such that
\begin{equation}\label{eqn:boundary_ext}
U_{j} = U_{-P} \quad \forall \quad j < -P \qquad \text{and} \qquad U_{j} = U_{P} \quad \forall \quad j > P.
\end{equation}
While one can argue about the validity of such an extension, especially due to the influence of the non-local term, we adhere to \eqref{eqn:boundary_ext} in order to reduce the computational cost associated with the repeated evaluations of the deterministic samples in the MLMC algorithm. Similar boundary conditions were also considered in \cite{droniou2010numerical}.

If $|i| \leq P$, we have $-P-i \leq 0$ and $P-i \geq 0$. Thus, under the assumption \eqref{eqn:boundary_ext} and using the notation $A_{j}:= A(U_{j})$, the non-local term in the scheme \eqref{eqn:fdm_explicit} (or \eqref{eqn:fdm_impexp}) can be written as
\begin{equation*}
\begin{aligned}
\sum \limits_{j\neq 0} G_j \bigl(A_{i+j} - A_i \bigr) &= \sum \limits_{j < -P-i } G_j \bigl(A_{i+j} - A_i \bigr) +  \sum \limits_{\substack{j = -P-i \\ j\neq 0}}^{j = P-i} G_j \bigl(A_{i+j} - A_i \bigr) +  \sum \limits_{j > P-i } G_j \bigl(A_{i+j} - A_i \bigr) \\
&= (A_{-P} - A_i) \sum \limits_{j < -P-i} G_j +  \sum \limits_{\substack{j = -P-i \\ j\neq 0}}^{j = P-i} G_j \bigl(A_{i+j} - A_i \bigr) +  (A_{P} - A_i) \sum \limits_{j > P-i } G_j  \\
&=  \frac{c_\lambda}{\lambda (\Delta x)^\lambda} \frac{(A_{-P} - A_i)}{(P+i+\frac{1}{2})^\lambda} + \frac{c_\lambda}{\lambda  (\Delta x)^\lambda}  \frac{(A_{P} - A_i)}{(P-i+\frac{1}{2})^\lambda}  +  \sum \limits_{\substack{j = -P-i \\ j\neq 0}}^{j = P-i} G_j \bigl(A_{i+j} - A_i \bigr).
\end{aligned}
\end{equation*}
The above formulation requires values $G_j$ for $ |j| \leq 2P$, which can be pre-computed and stored for a given mesh. 

\subsection{Variance and MLMC error estimation}
In the experiments, we wish to compute the variance of the computed estimated mean. This is achieved by using the following stable algorithm, which was also used in \cite{koley2013multilevel}
\begin{equation*}
\begin{aligned}
V_L &= \sum \limits_{l=1}^L \Delta V_l + V_0, \\
\Delta V_l &= E_{M_l} \left[ (u_l - u_{l-1} - E_{M_l}[u_l - u_{l-1}] )^2\right],\\
V_0 &= E_{M_0} \left[ (u_0 - E_{M_0}[u_0] )^2\right].
\end{aligned}
\end{equation*}
The number of samples used in each level are chosen according to the formulas \eqref{sample_no_expl} - \eqref{sample_no_impl_0}, by setting $p=2$. If the resulting number is not an integer, it is rounded off to the smallest integer greater than this number. 

In order to estimate the error $\|\Ex[u(t,\cdot) - E^L[u(t,\cdot)]\|_{L^2(\Omega;L^2(I))}$, we use the root mean square estimate
\begin{equation}\label{eqn:rms}
\rms = \sqrt{\frac{1}{Q}\sum \limits_{k=1}^Q(\rms_k)^2 },
\end{equation}
where
\[
\rms_k = \frac{\|Z_{\text{ref}}(T,.) - Z_k(T,.) \|_{L^2(I)}}{\|Z_{\text{ref}}(T,.) \|_{L^2(I)}}.
\]
Here, $Z_k$ refers to the computed estimated mean for the index $k$, while $Z_{\text{ref}}$ denotes to the reference mean. The index $k$ refers to independent runs of the MLMC-FDM algorithm, needed to obtain different realizations of the probability space. The sensitivity of the error with respect to the parameter $Q$ has been investigated in \cite{mishra2012sparse, mishra2013multi}. It was noted that $Q=30$ is sufficient for most problems, to remove statistical fluctuations. The reference solution $Z_\text{ref}$ is obtained by: 
\begin{enumerate}
\item Uniformly discretizing the sample space $\Omega$ (which is assumed to be a closed box), with the discretized points denoted by $\{\omega_s\}_s$.
\item Computing the numerical approximation $u_{\Delta}(\omega_s;T,.)$ for each $\omega_s$ on a fine mesh.
\item Applying a trapezoidal quadrature rule to approximate the integral $|\Omega|^{-1}\int_\Omega u(\omega;T,.) \ \ud \omega$, using the points $\{\omega_s\}_s$.
\end{enumerate}

\begin{rem}\label{rem:ref_soln}
The method described above to generate the reference solution makes sense only if the various random parameters are sampled from a uniform distribution, which is the choice we adhere to in this work (see Section \ref{sec:num_MLMC}).
\end{rem}

\begin{rem}\label{rem:error}
Note that the error approximated by \eqref{eqn:rms} corresponds to $\sqrt{Error_L}$, where $Error_L$ is estimated in \eqref{err_MLMC}.
\end{rem}

\subsection{Deterministic simulations}
We consider the parametrised initial condition
\begin{equation}\label{eqn:ic}
u_0(c; x) = \begin{cases} 0.85 \quad &\text{if } -0.5 + c < x < 0\ , \\ 0.1 \quad &\text{otherwise}\ .\end{cases}
\end{equation} 
The local Lax Friedrich flux
\begin{align*}
F(U_i^n,U_{i+1}^n) = \frac{1}{2}(f(U_i^n)+f(U_{i+1}^n)) - \frac{1}{2}\max\big(|f^\prime(U_i^n)|, |f^\prime(U_{i+1}^n)|\big) (U_{i+1}^n - U_i^n),
\end{align*}
is used, with the time-step evaluated using the CFL condition
\begin{equation}\label{eqn:numCFL}
\frac{\Delta t}{\Delta x^{1\vee\lambda}} = CFL  < 1.
\end{equation}
We choose $CFL=0.2$ for all experiments presented in this paper.

We begin by comparing the simulations at time $T=1$ with the schemes \eqref{eqn:fdm_explicit} and \eqref{eqn:fdm_impexp}, by setting $c=0.0$, $\mu = 0.5$, $\alpha = 0.2$ (see \eqref{eqn:bl_flux}) and $K=5$ (the domain $[-5,5]$). The numerical results on a mesh with $N=501$ nodes are shown in Figure \ref{fig:scheme_compare}. The solutions obtained with the explicit and explicit-implicit FDMs are almost indistinguishable. We make two observations from the average run-times listed in Table \ref{tab:deterministic_times}. Firstly, the run-time increases if the exponent $\lambda$ is increased beyond unity. This can be easily understood by looking at the CFL condition \eqref{eqn:numCFL} used to determine $\Delta t$. Secondly, the computational cost is significantly higher with the  explicit-implicit scheme, as we need to solve a non-linear system at each time-step. The solution profiles at various instances of time are shown in Figure \ref{fig:increasing_time}.

\begin{figure}[htbp]
\begin{center}
\subcaptionbox{$\lambda = 0.5$}{\includegraphics[width=0.33\textwidth]{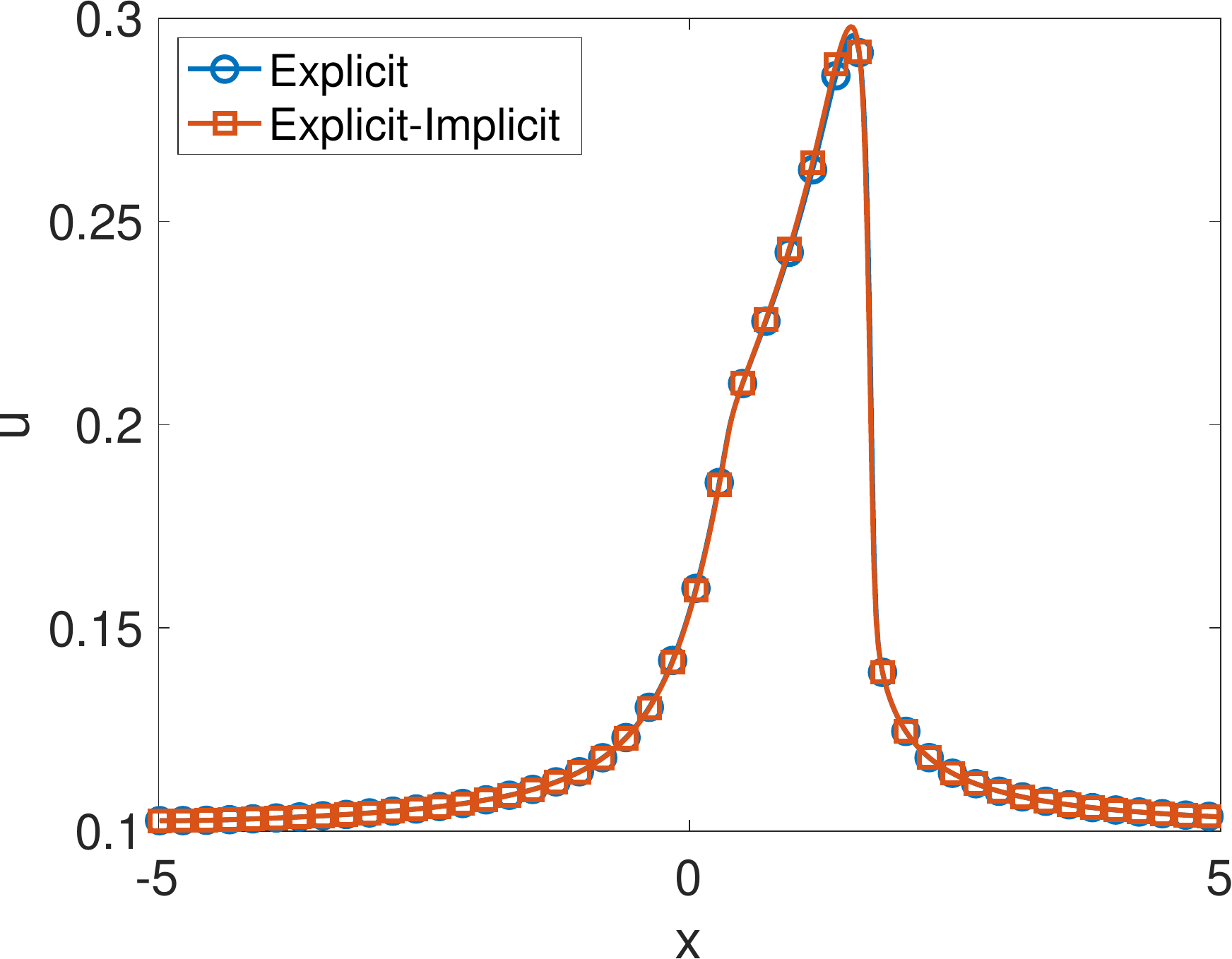}}
\subcaptionbox{$\lambda = 1.5$}{\includegraphics[width=0.33\textwidth]{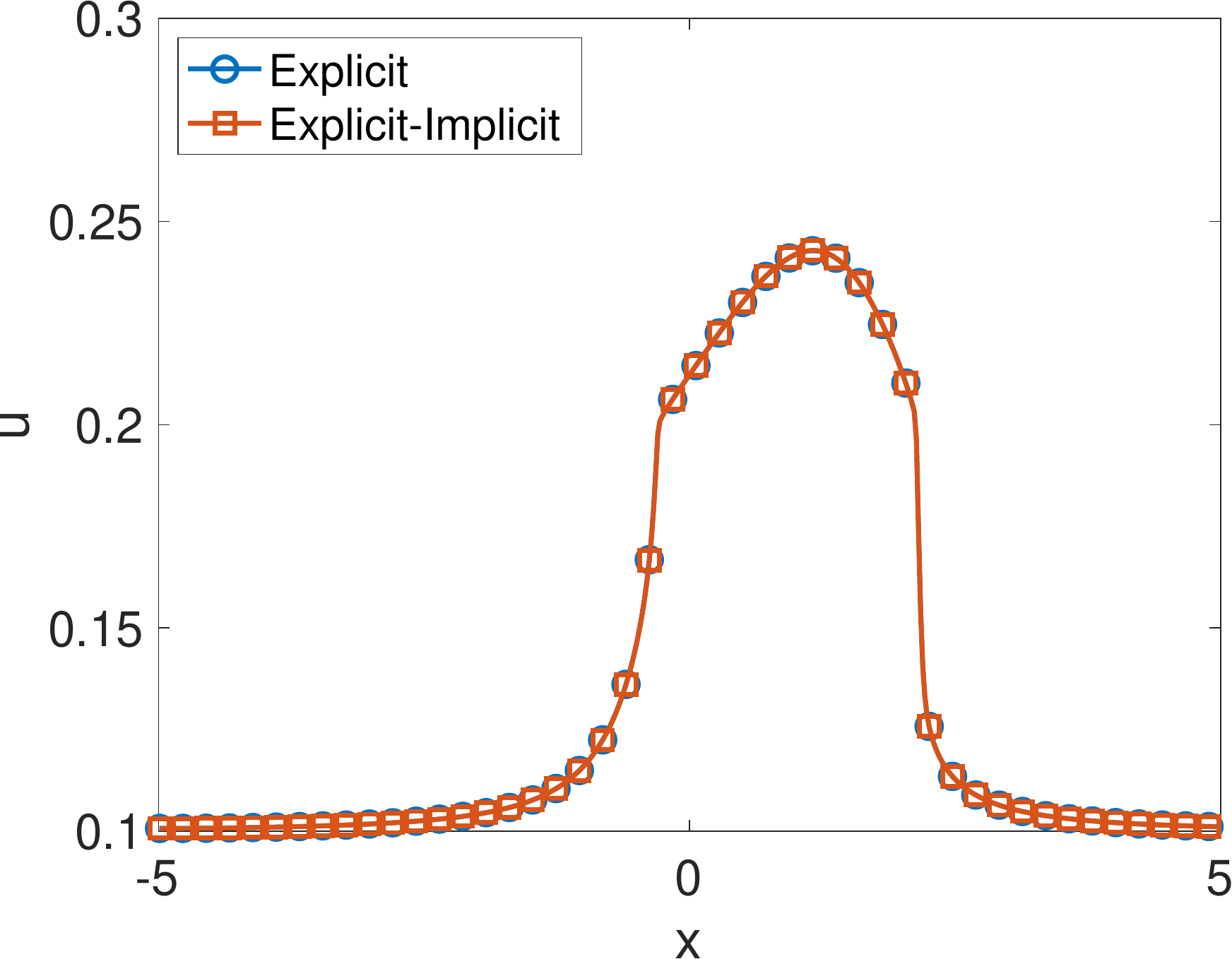}}
\caption{Buckley-Leverett problem evaluated at $T=1$ with schemes \eqref{eqn:fdm_explicit} and \eqref{eqn:fdm_impexp}. The solution is obtained on a mesh with 501 cells, and parameters $c = 0.0$, $\mu = 0.5$ and $\alpha = 0.2$.}
\label{fig:scheme_compare}
\end{center}
\end{figure}

\begin{table}[htbp]
\begin{tabular}{|c|c|c|}
\hline
\textbf{Scheme}                             & $\lambda$ & \textbf{Run time(s)} \\ \hline
\multirow{2}{*}{Explicit}          & 0.5       &     \num{1.17e-1}       \\ \cline{2-3} 
                                   & 1.5       &    \num{1.09e+0}        \\ \hline
\multirow{2}{*}{Explicit-implicit} & 0.5       &   \num{3.78e+1}         \\ \cline{2-3} 
                                   & 1.5       &    \num{2.72e+2}       \\ \hline
\end{tabular}
\caption{Average run times for deterministic simulations on a mesh with $N=501$ cells and parameters $c = 0.0$, $\mu = 0.5$, $\alpha = 0.2$}
\label{tab:deterministic_times}
\end{table}

\begin{figure}[htbp]
\begin{center}
\subcaptionbox{$\lambda = 0.5$}{\includegraphics[width=0.33\textwidth]{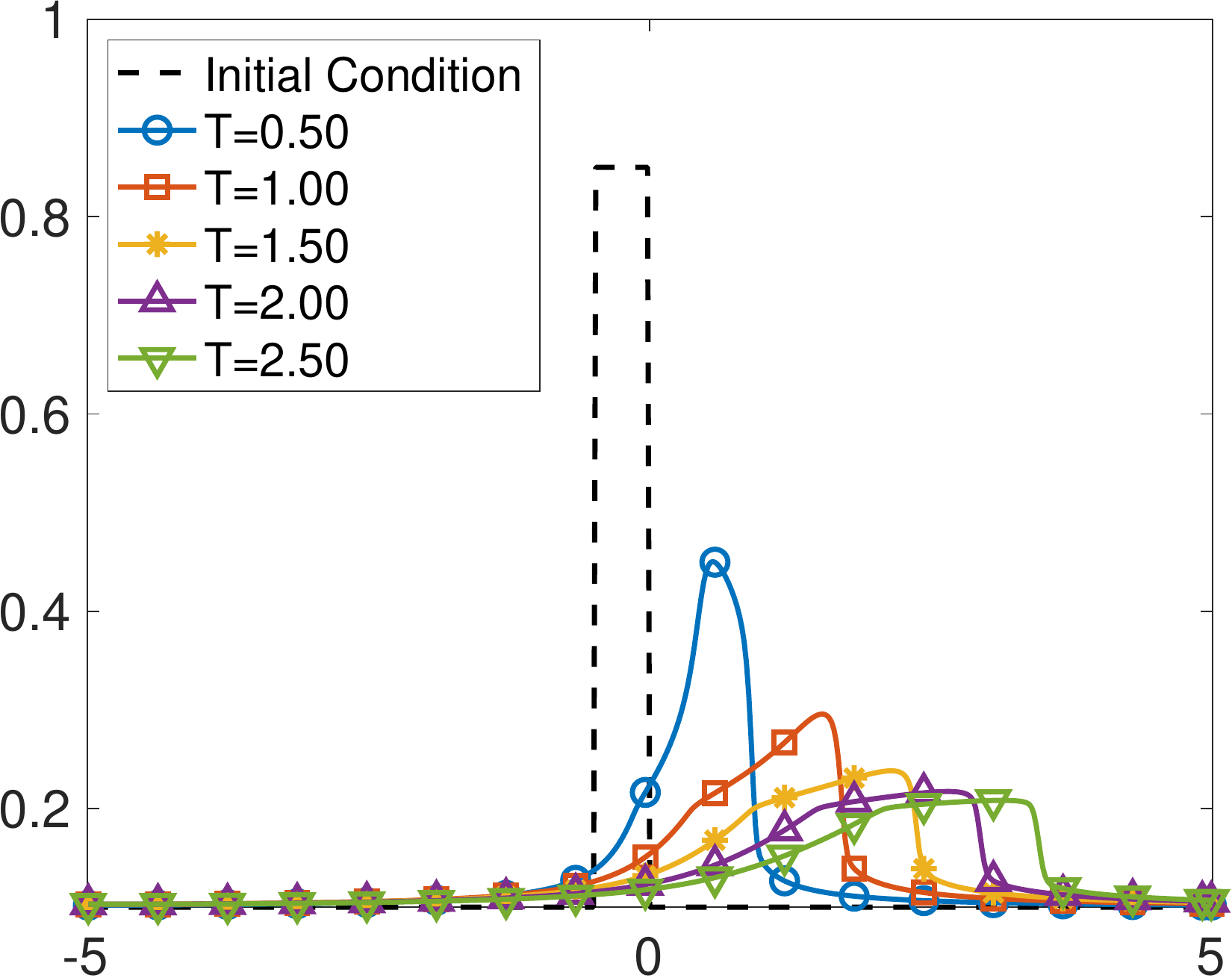}}
\subcaptionbox{$\lambda = 1.5$}{\includegraphics[width=0.33\textwidth]{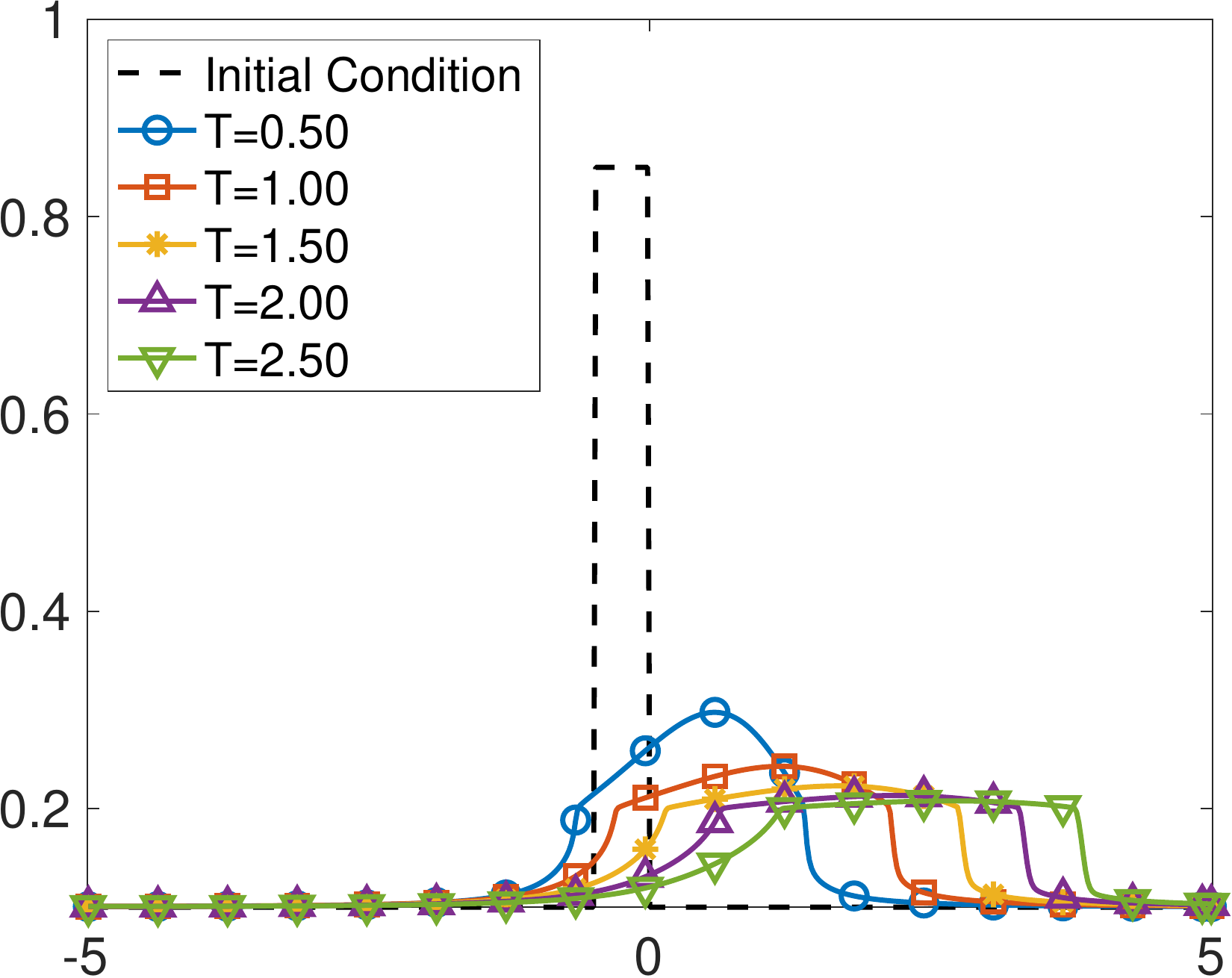}}
\caption{Buckley-Leverett problem evaluated at $T=0.5,1,1.5,2$ and $2.5$ with the scheme \eqref{eqn:fdm_explicit}. The solution is obtained on a mesh with 501 cells, and parameters $c = 0.0$, $\mu = 0.5$ and $\alpha = 0.2$.}
\label{fig:increasing_time}
\end{center}
\end{figure}

Next, we analyse the effect of the fractional exponent on the solution. We take $c = 0.0$, $\mu = 0.5$, $\alpha = 0.2$, $N=501$ and simulate the solution using the scheme \eqref{eqn:fdm_explicit} for varying values of $\lambda$. As can be seen in Figure \ref{fig:lambda_compare}, the solution has sharp features resembling a shock for smaller values of $\lambda$. The diffusion term becomes stronger as $\lambda$ is increased from 0 to 2. To represent all exponent partitions considered in \eqref{eqn:MLMC_err_exp} and \eqref{eqn:MLMC_err_imp}, we choose $\lambda = 0.5, 0.75$ and $1.5$ for the Monte-Carlo simulations in the next section. In Figure \ref{fig:mesh_ref}, we plot the solution for these three exponents on each of the mesh levels to be considered in the MLMC algorithm, and on the mesh used to generate the reference solution.

\begin{rem}\label{rem:asymptotics}
In the limit $\lambda \downarrow 0$ or $\lambda \uparrow 2$, the non-local diffusion term in the model \eqref{fdcd} converges to a source term or a Laplacian diffusion term respectively. The schemes considered in the present work are not expected to preserve these asymptotic properties. While asymptotic preserving schemes are available (see \cite{droniou2014asymptotics}), it is not possible to obtain rigorous theoretical convergence and work estimates for such schemes at present.
\end{rem}

\begin{figure}[htbp]
\begin{center}
\subcaptionbox{$\lambda = 0.1$}{\includegraphics[width=0.32\textwidth]{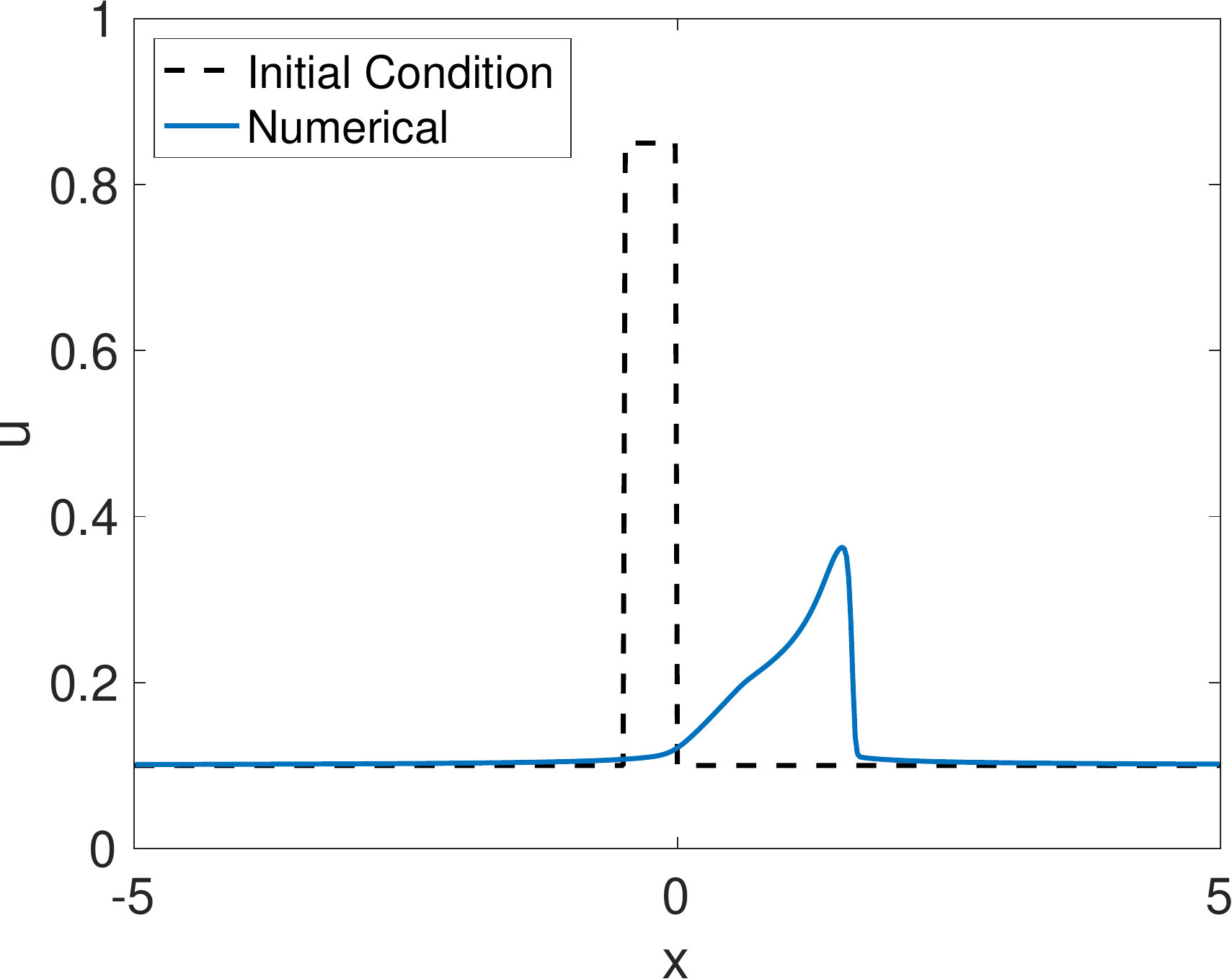}}
\subcaptionbox{$\lambda = 0.5$}{\includegraphics[width=0.32\textwidth]{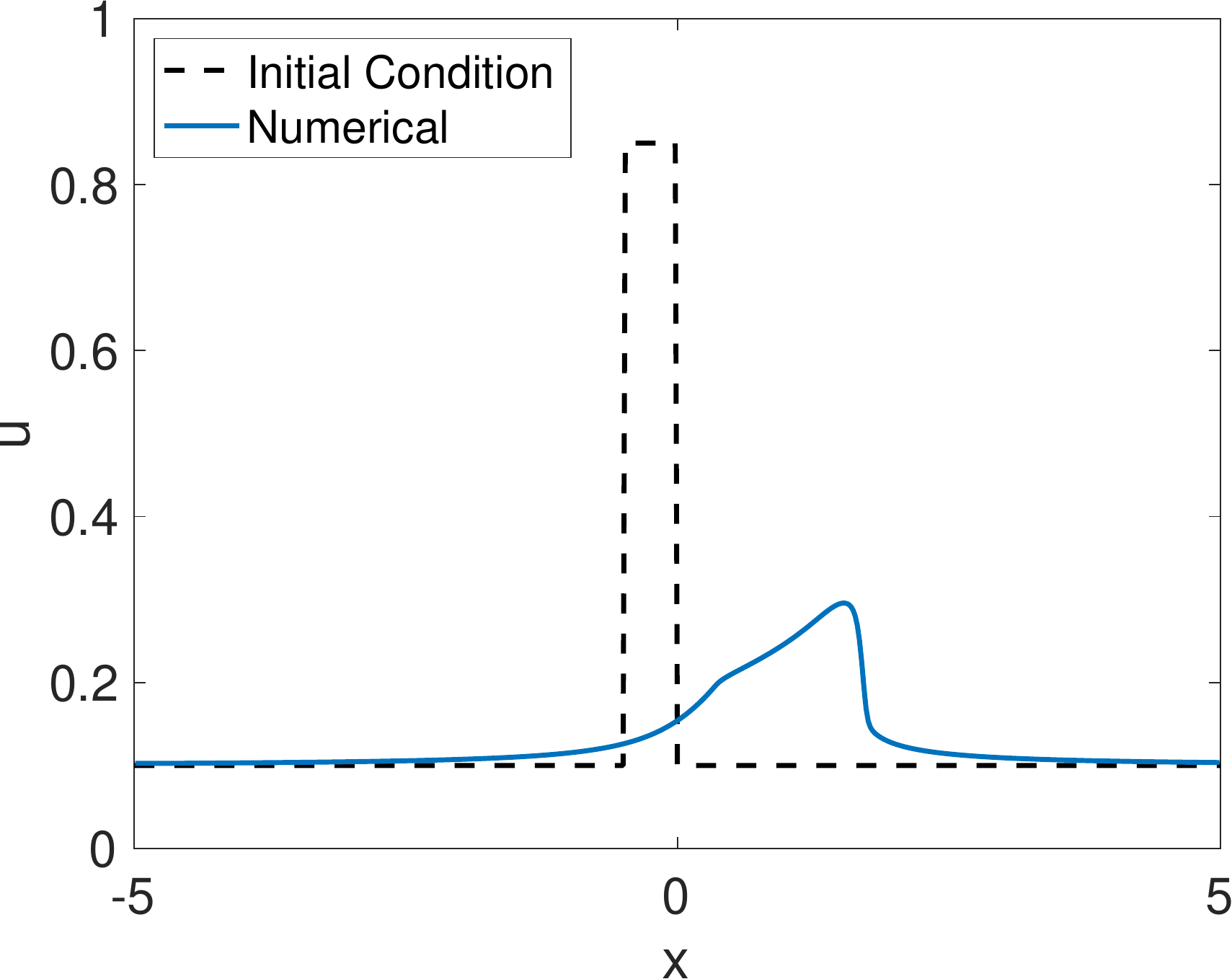}}
\subcaptionbox{$\lambda = 0.9$}{\includegraphics[width=0.32\textwidth]{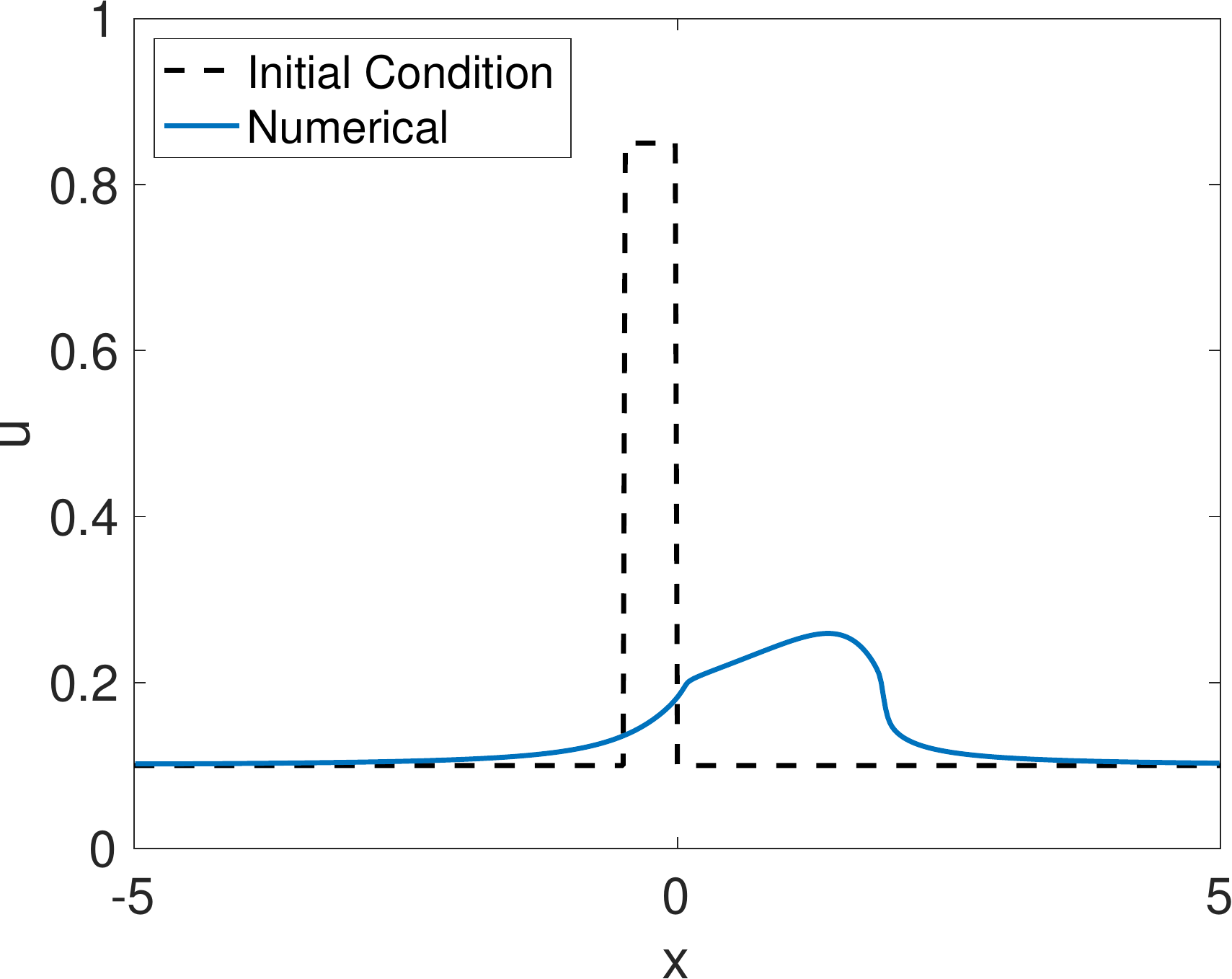}}\\
\subcaptionbox{$\lambda = 1.0$}{\includegraphics[width=0.32\textwidth]{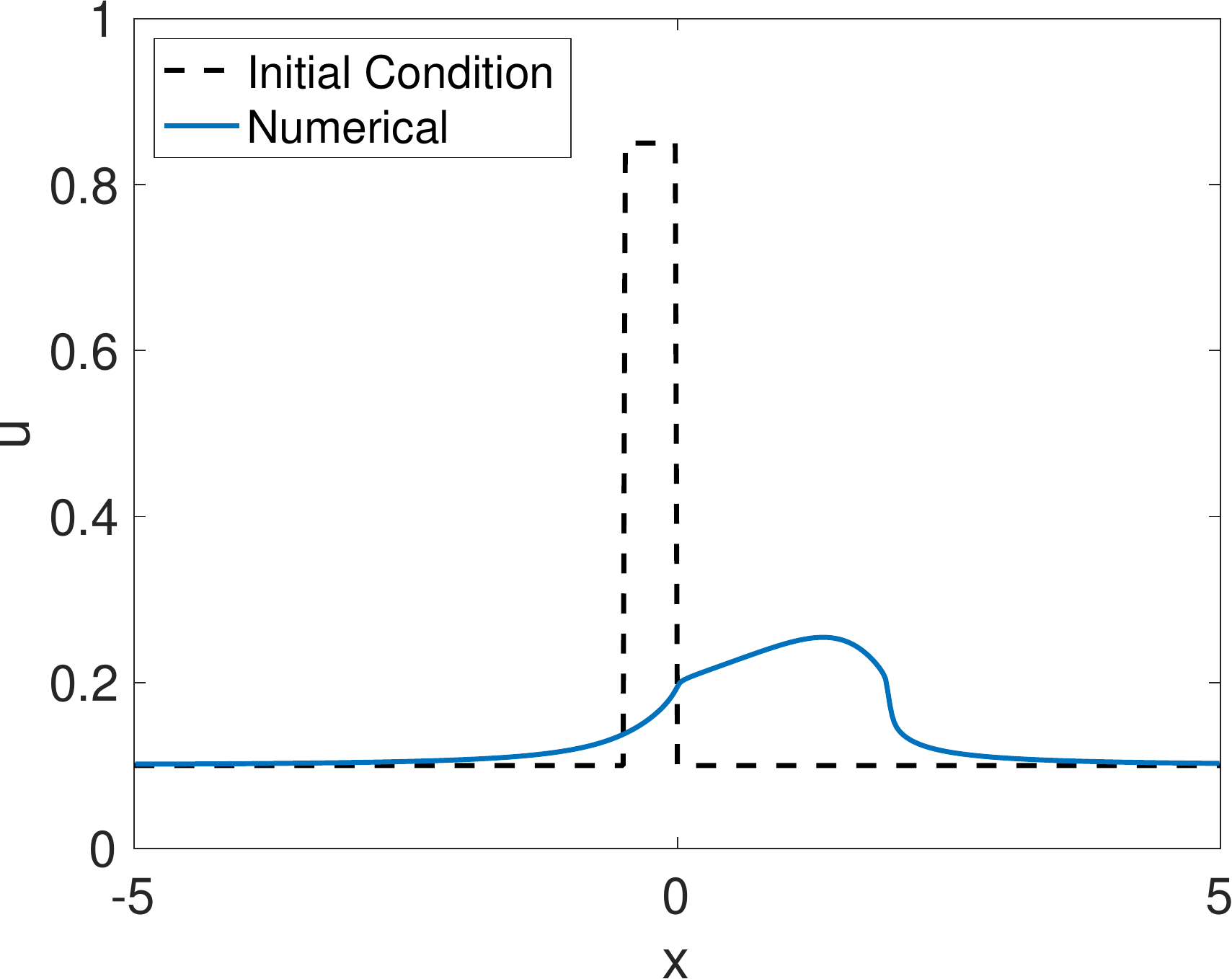}}
\subcaptionbox{$\lambda = 1.5$}{\includegraphics[width=0.32\textwidth]{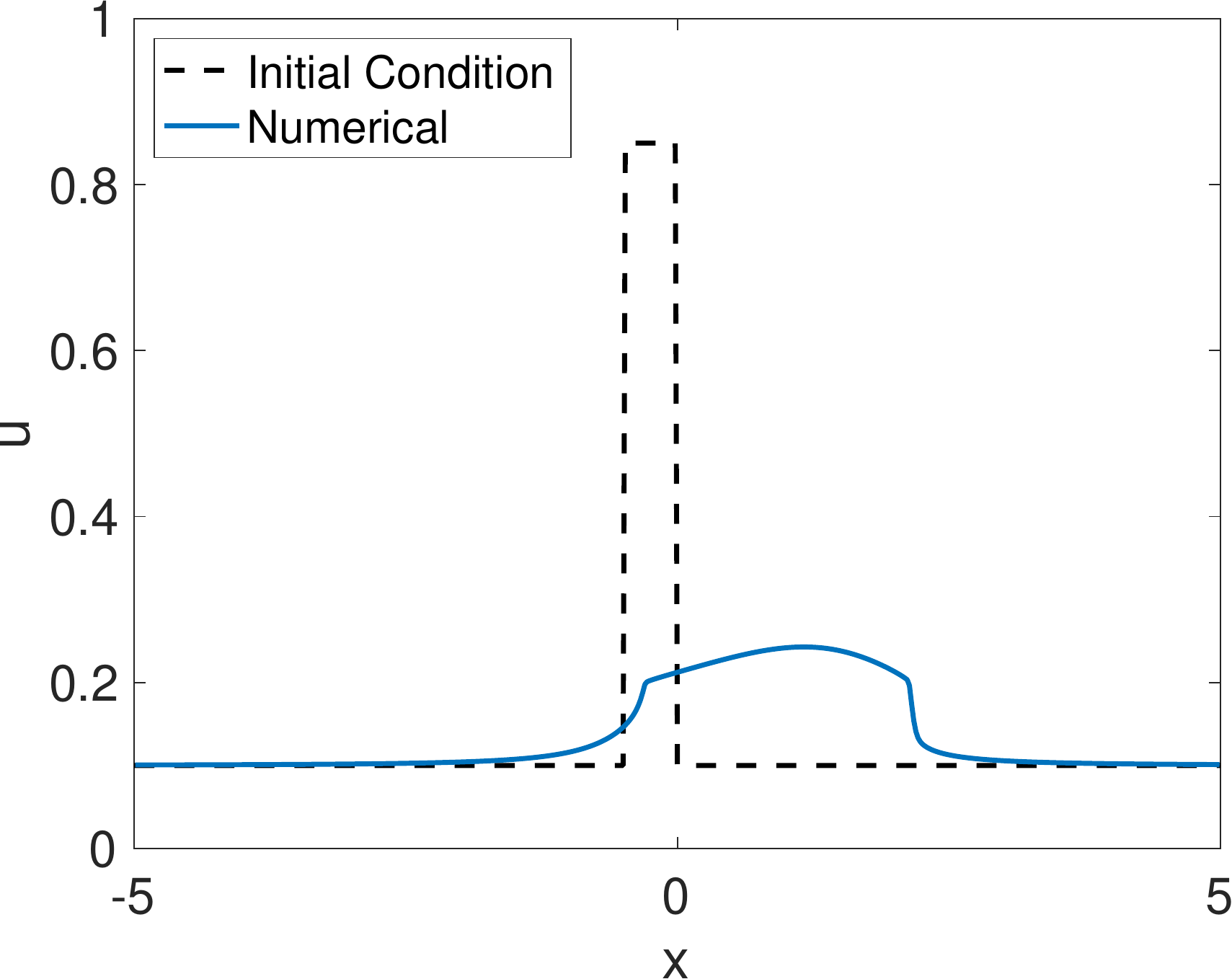}}
\subcaptionbox{$\lambda = 1.9$}{\includegraphics[width=0.32\textwidth]{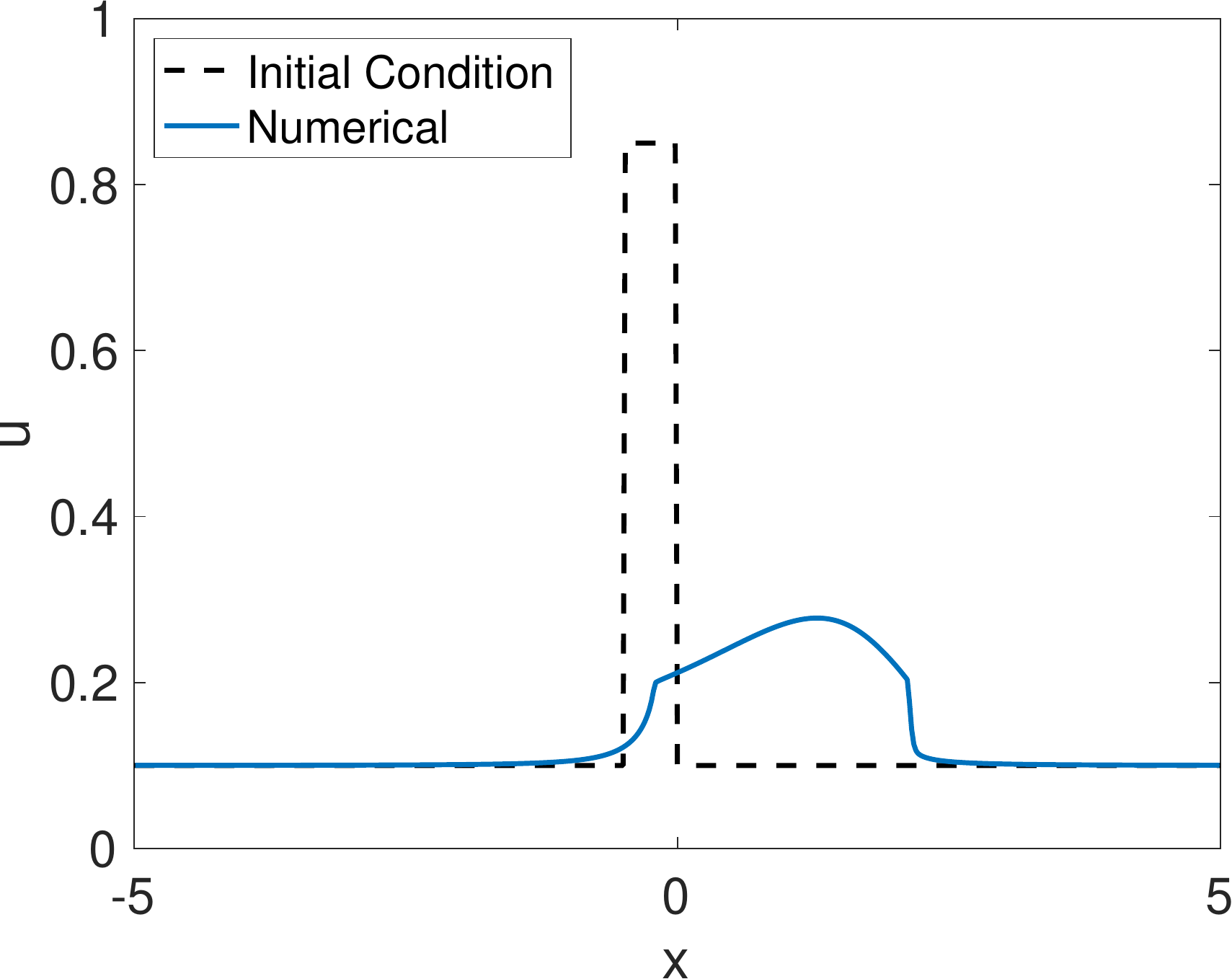}}
\caption{Buckley-Leverett problem evaluated at $T=1$ with scheme \eqref{eqn:fdm_explicit}. The solution is evaluated on a mesh with 501 cells, and parameters $c = 0.0$, $\mu = 0.5$ and $\alpha = 0.2$.}
\label{fig:lambda_compare}
\end{center}
\end{figure}

\begin{figure}[htbp]
\begin{center}
\subcaptionbox{$\lambda = 0.5$}{\includegraphics[width=0.32\textwidth]{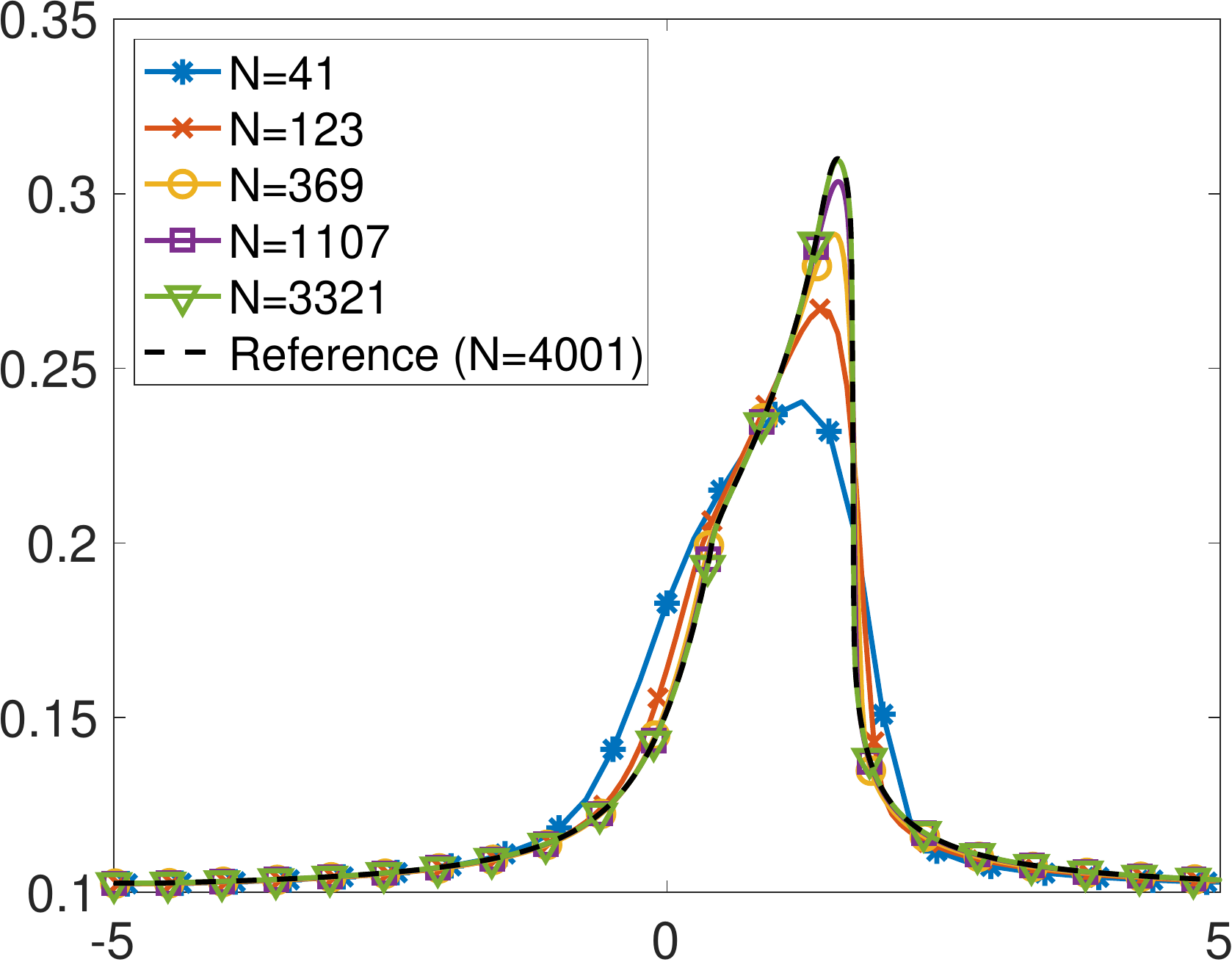}}
\subcaptionbox{$\lambda = 0.75$}{\includegraphics[width=0.32\textwidth]{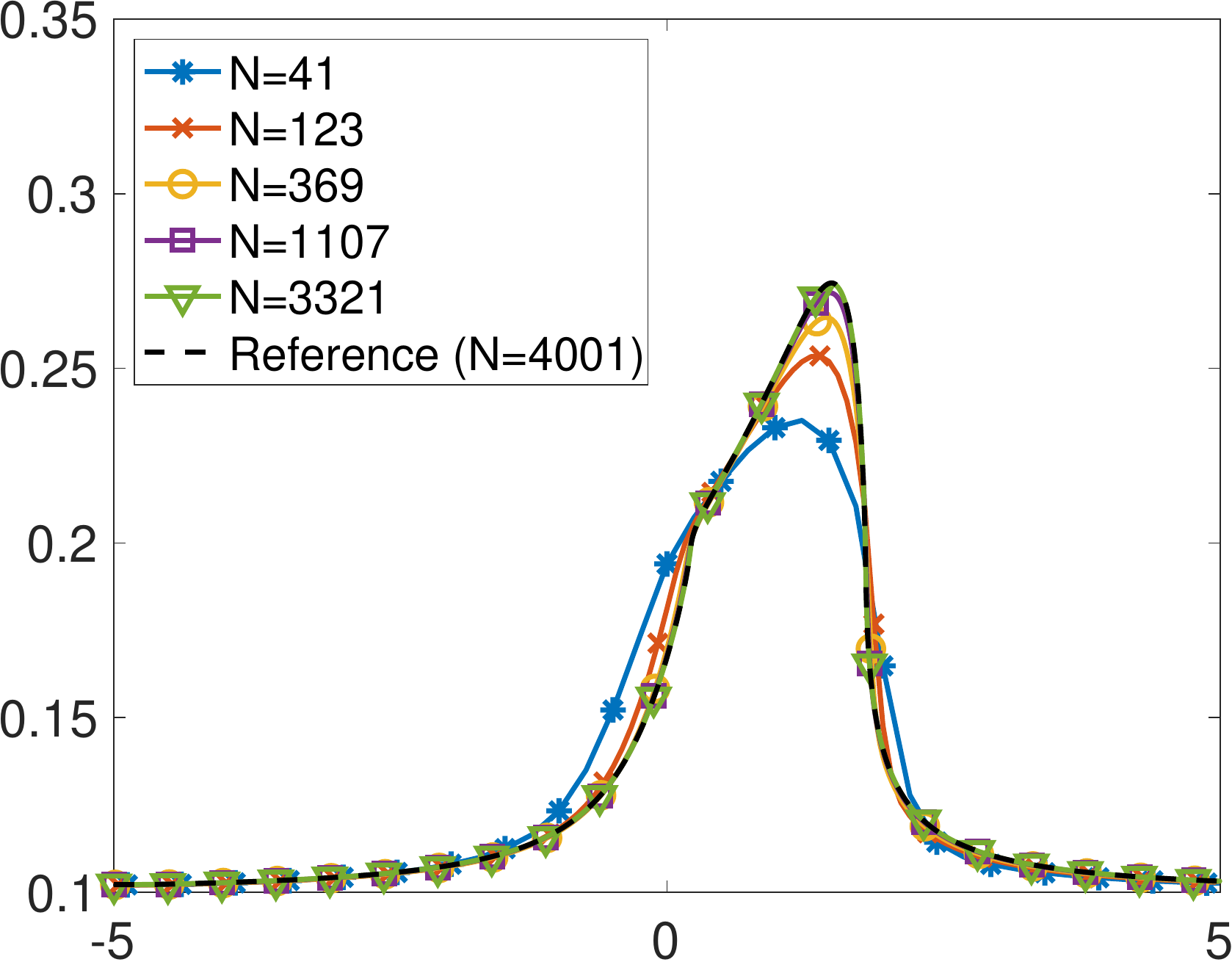}}
\subcaptionbox{$\lambda = 1.5$}{\includegraphics[width=0.32\textwidth]{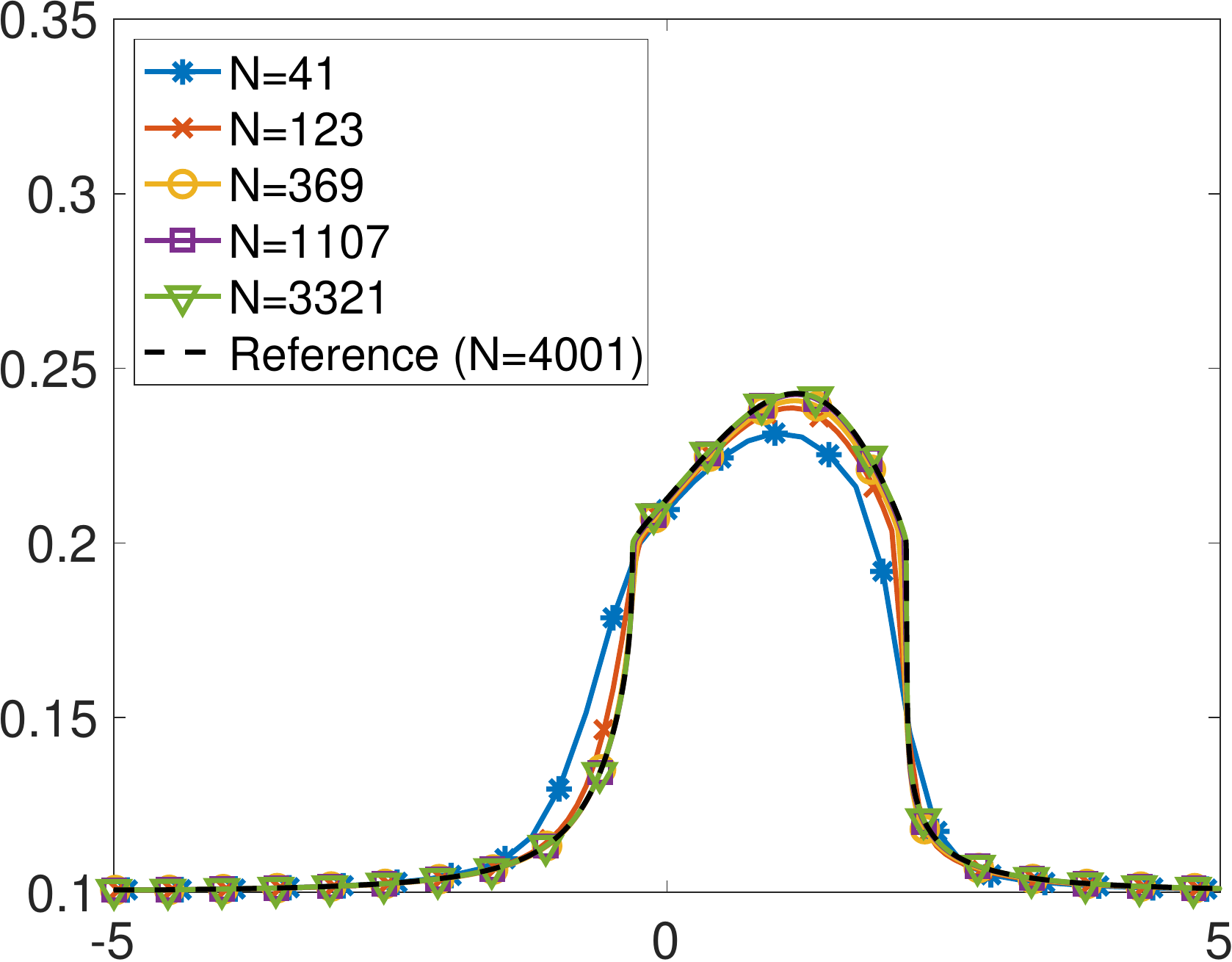}}\\
\subcaptionbox{$\lambda = 0.5$ (zoomed)}{\includegraphics[width=0.32\textwidth]{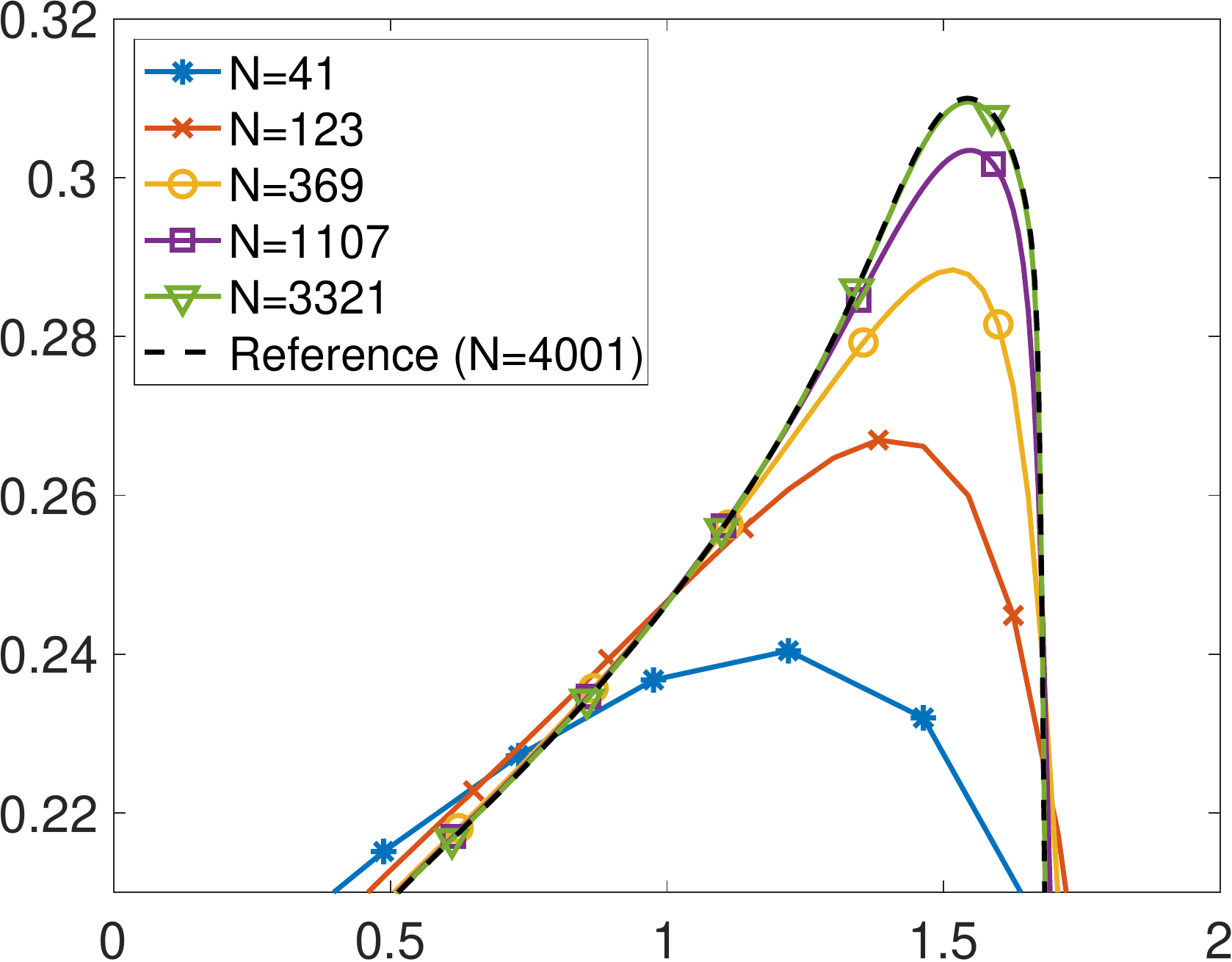}}
\subcaptionbox{$\lambda = 0.75$ (zoomed)}{\includegraphics[width=0.32\textwidth]{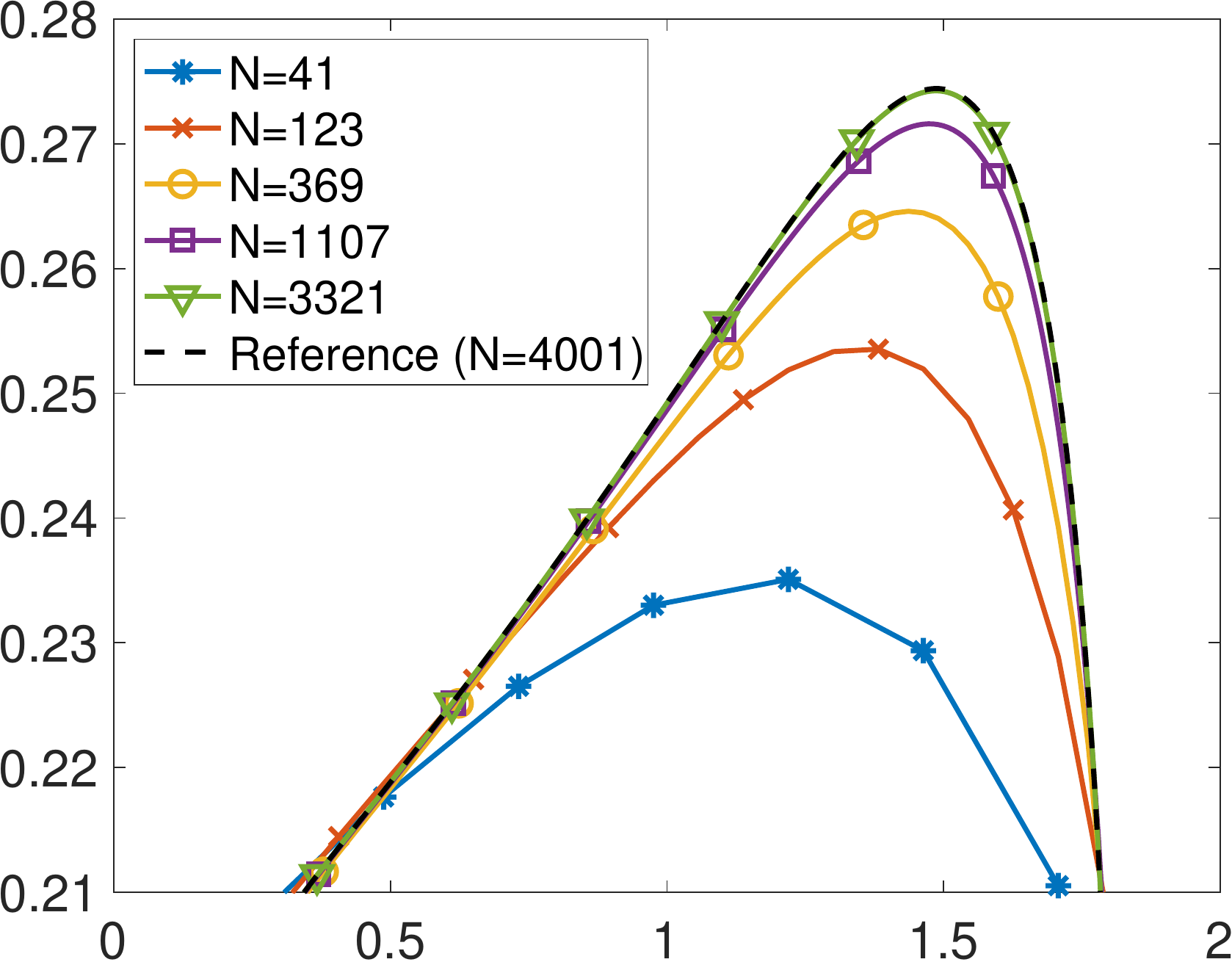}}
\subcaptionbox{$\lambda = 1.5$ (zoomed)}{\includegraphics[width=0.32\textwidth]{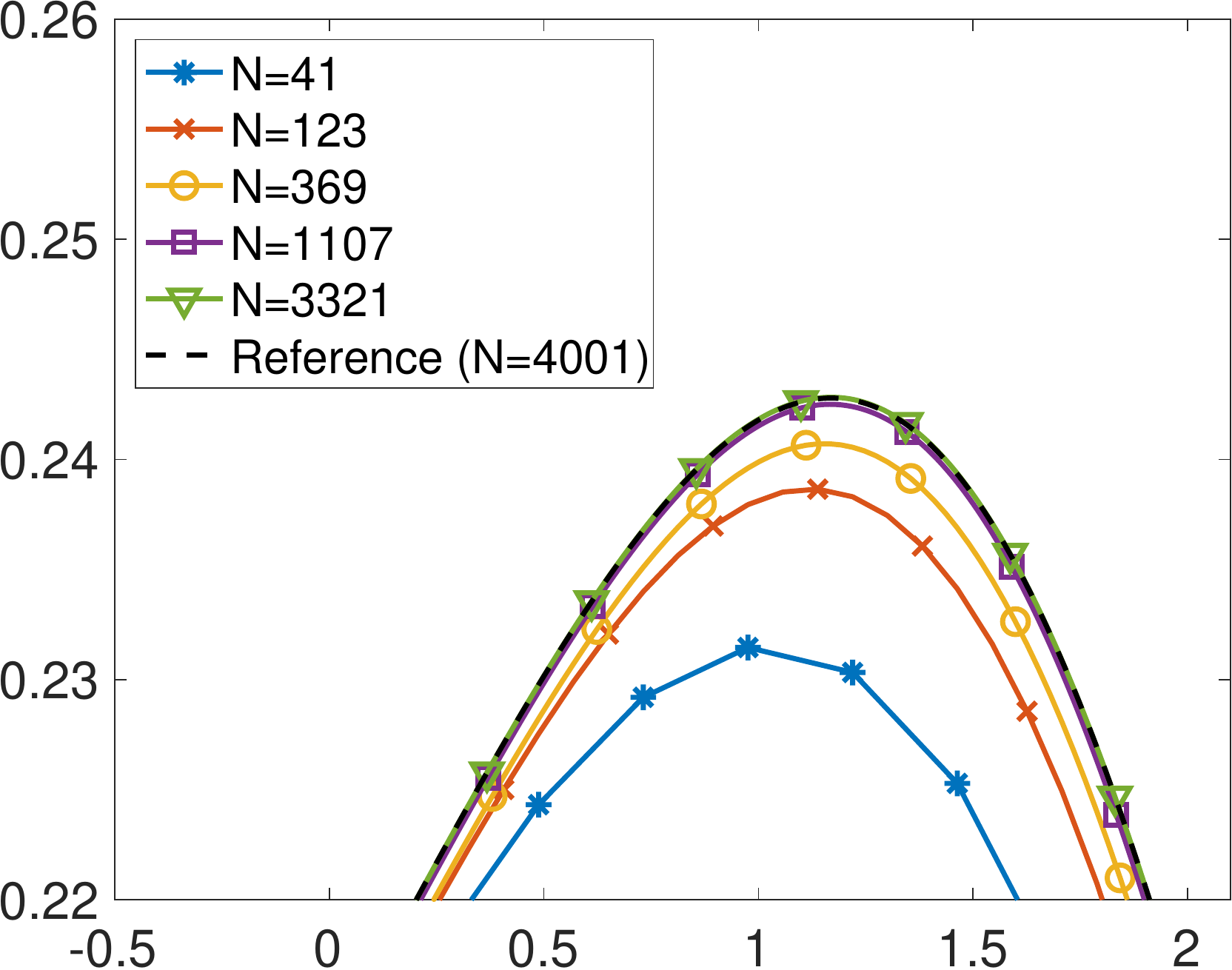}}
\caption{Buckley-Leverett problem evaluated at $T=1$ with scheme \eqref{eqn:fdm_explicit}. The solution is evaluated with parameters $c = 0.0$, $\mu = 0.5$ and $\alpha = 0.2$.}
\label{fig:mesh_ref}
\end{center}
\end{figure}

\subsection{MLMC simulations}\label{sec:num_MLMC}
We now demonstrate the performance of the MLMC-FDM algorithms. We introduce uncertainty in the initial condition, the flux, and the dissipation term by choosing $c \sim \mathcal{U}(0,0.1)$, $\mu \sim  \mathcal{U}(0.3,0.7)$ and $\alpha \sim  \mathcal{U}(0,0.4)$, respectively. Note that, a similar argument as stated in the proof of Theorem~\ref{random_entropy} reveals that the random numerical solution is measurable since the solution map is a composition of a measurable and a continuous map. The number of samples for each mesh level is chosen using \eqref{sample_no_expl}-\eqref{sample_no_expl_0} for the explicit scheme and \eqref{sample_no_impl}-\eqref{sample_no_impl_0} for the explicit-implicit scheme, with the error tolerance set to $\varepsilon_{er} = 2\Delta x_L^{2\Theta}$. Figure \ref{fig:exp_scheme_MLMC} shows the statistical quantities evaluated with the explicit MLMC-FDM algoruthm for $\lambda = 0.5$, $0.75$ and $1.5$, with 41 cells in the coarsest mesh and $L=4$. The solid line represents the estimated mean, while the dashed lines represent the estimated mean $\pm$ the standard deviation. The shaded region between the two dashed lines is referred to as the deviation band. As was seen in the deterministic experiments, the solution tends to be more diffused as the exponent $\lambda$ is increased. This explains why the deviation band in Figure \ref{fig:exp_scheme_MLMC} broadens with increasing $\lambda$. 

\begin{figure}[!htbp]
\begin{center}
\subcaptionbox{$\lambda = 0.5$}{\includegraphics[width=0.33\textwidth]{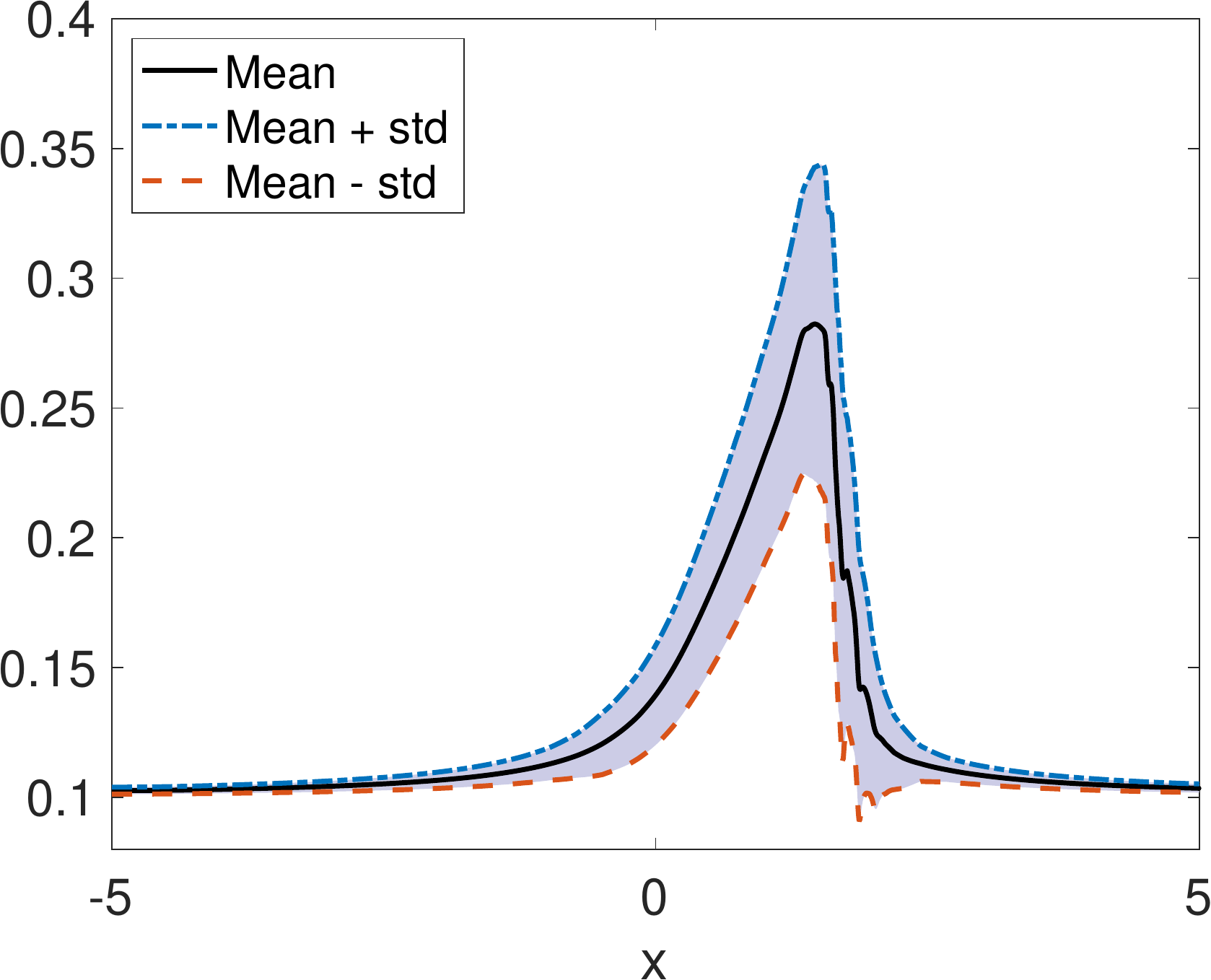}} \hfill
\subcaptionbox{$\lambda = 0.75$}{\includegraphics[width=0.33\textwidth]{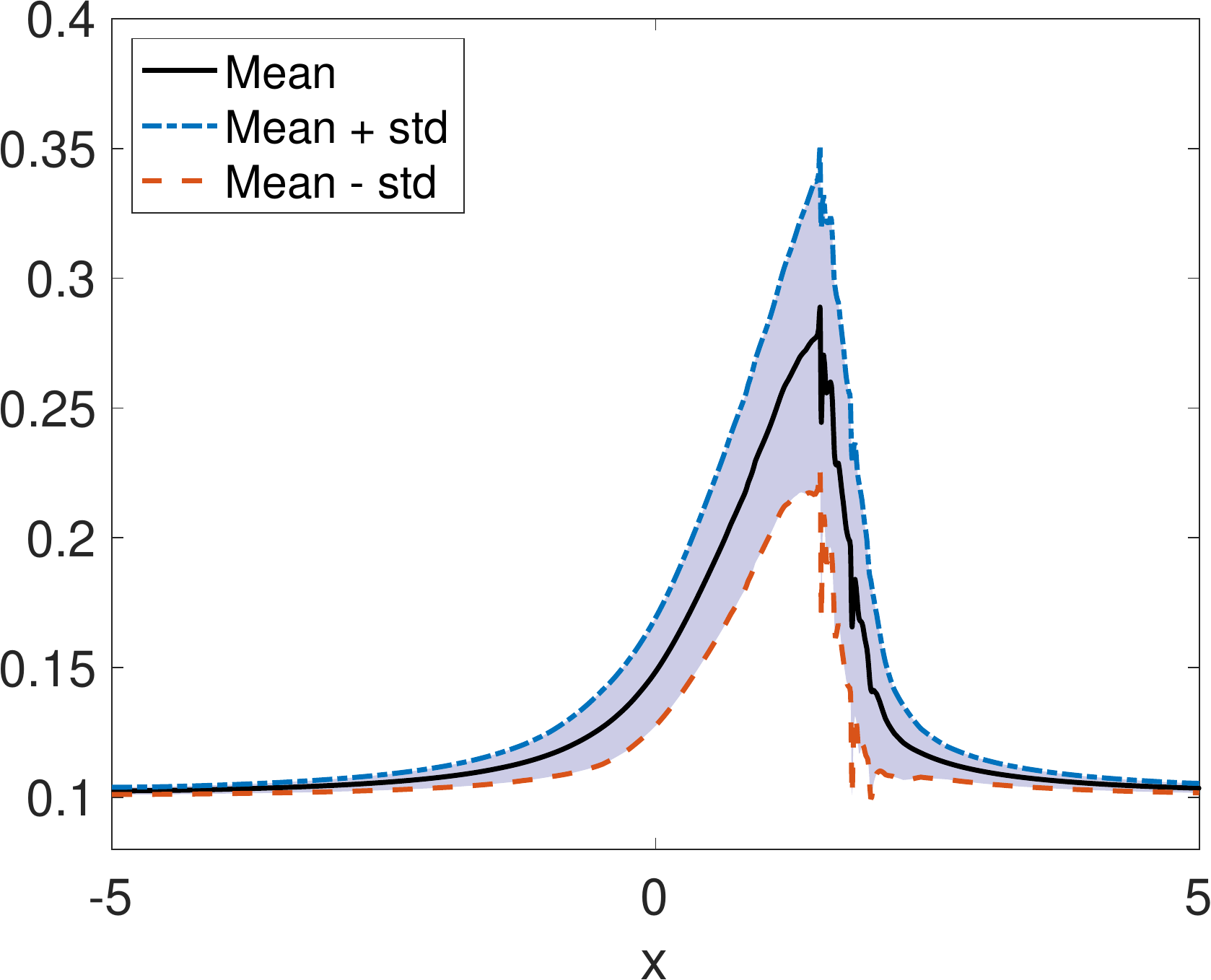}}\hfill
\subcaptionbox{$\lambda = 1.5$}{\includegraphics[width=0.33\textwidth]{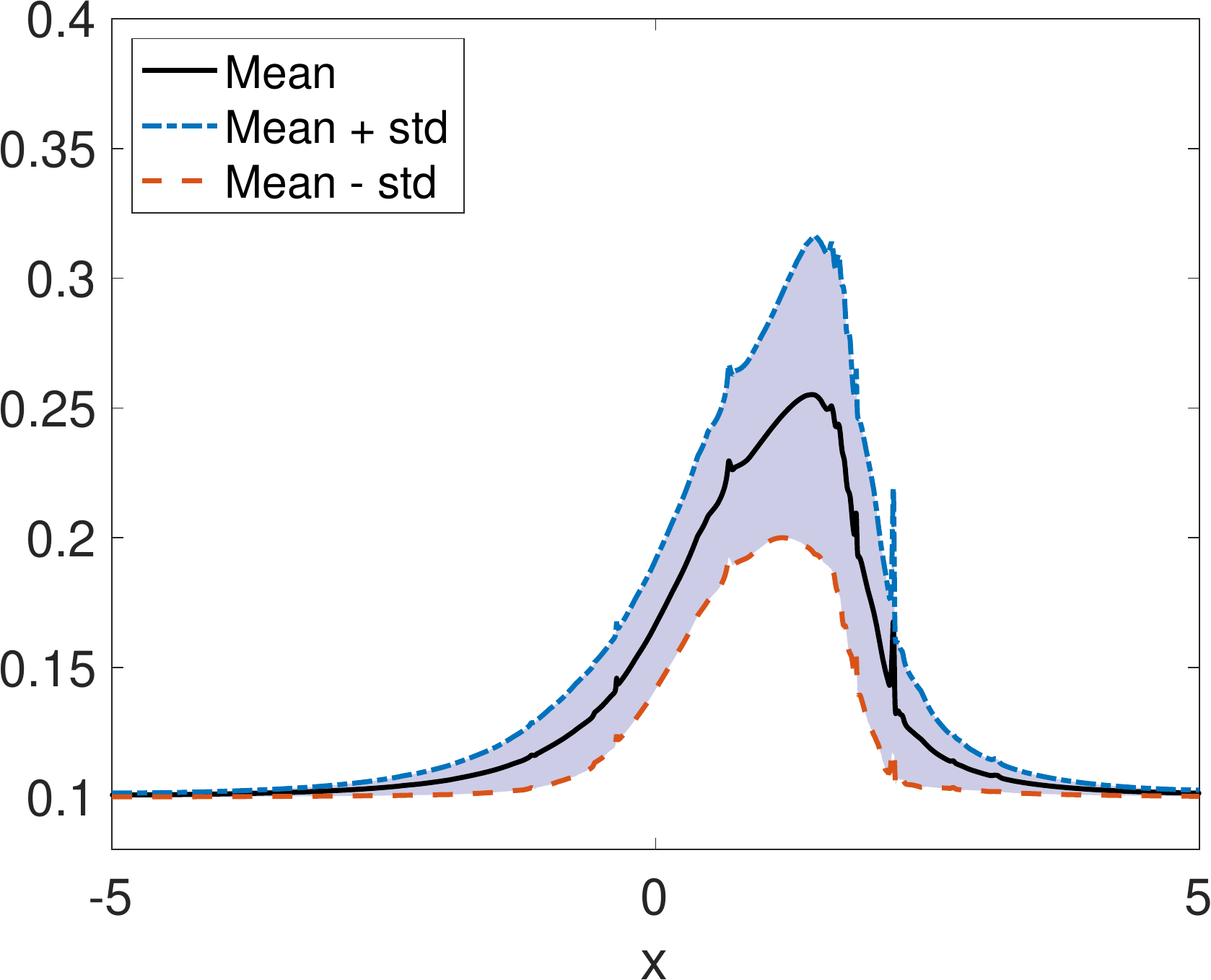}}
\caption{Buckley-Leverett problem evaluated at $T=1$ with explicit MLMC-FDM. The mean estimator $E^4[u(T,.)]$ evaluated with $L=4$ levels, with the dashed lines representing $E^4[u(T,.)] \ \pm$ standard deviation.}
\label{fig:exp_scheme_MLMC}
\end{center}
\end{figure}

Tables \ref{tab:0p5_ex}-\ref{tab:1p5_ex} show the estimated $\mathcal{RMS}$ errors evaluated using \eqref{eqn:rms} with the explicit scheme, as a function of the number of MLMC level $L$, with the coarsest mesh having $N^x_0 = 41$ cells and the finest mesh having $N^x_L$ cells.  We compute the $\mathcal{RMS}$ decay rate $r_1>0$ with respect to the finest mesh size $\Delta x_L$. In view of Remark \ref{rem:error} and the prescription of the error tolerance $\varepsilon_{er} = 2\Delta x_L^{2\Theta}$ to determine the samples in each level, the theoretical estimate of this rate is $r_1 = \Theta$. Based on the estimate \eqref{eqn:MLMC_err_exp}, we also compute and compare the $\mathcal{RMS}$ decay rate $r_2>0$ with respect to the work done, i.e., $\mathcal{RMS} \sim$ $(\text{work})^{-r_2}$. The work done is estimated in term of the total CPU run time (in seconds) for the $Q=30$ revaluations. We observe that the rates are better than those predicted by theory, for all three values of $\lambda$ considered in the experiments. Tables \ref{tab:0p5_im}-\ref{tab:1p5_im} show the estimated $\mathcal{RMS}$ errors with the explicit-implicit scheme. We have limited the MLMC experiments with the explicit-implicit scheme to $L=3$, as the cost of generating each deterministic sample is very high. Based on \eqref{eqn:MLMC_err_imp}, $r_2$ is computed under the assumption that $\mathcal{RMS} \sim$ $(\text{work}/\log(\text{work}))^{-r_2}$.  As was observed with the explicit scheme, the rates are much better than the theoretical ones. Furthermore, we note that the number of samples required and the run times are significantly larger for a given $L>1$, as compared to the MLMC simulations using the explicit scheme.

\begin{table}[!htbp]
\centering
\begin{tabular}{|c|l|l|l|l|l|l|}
\hline
\textbf{L}            & \textbf{1} & \textbf{2} & \textbf{3} & \textbf{4} & \textbf{$r$} & \textbf{expected $r$}\\ \hline
\textbf{\begin{tabular}[c]{@{}c@{}}Number \\of \\Samples\end{tabular}}      &     \begin{tabular}[c]{@{}l@{}}$M_0 = 14$\\ $M_1=2$\end{tabular}       &    \begin{tabular}[c]{@{}l@{}}$M_0 = 90$\\ $M_1=10$ \\ $M_2=2$\end{tabular} &       \begin{tabular}[c]{@{}l@{}}$M_0 = 611$\\ $M_1=63$ \\ $M_2=10$ \\ $M_3 = 2$\end{tabular}     &  \begin{tabular}[c]{@{}l@{}}$M_0 = 4169$\\ $M_1=429$ \\ $M_2=63$ \\ $M_3 = 10$ \\ $M_4 = 2$\end{tabular} &     &      \\ \hline
$\mathbf{N^x_L}$ &  123   & 369   & 1107 & 3321 & 0.484 & 0.25\\ \hline
$\mathbf{\mathcal{RMS}}$          &       \num{8.948e-2}     &  \num{4.967e-2}          &   \num{2.976e-2}         &      \num{1.806e-2} &       &     \\ \hline
\textbf{Run time(s)}  &   \num{2.524e-1}      &   \num{3.550e+0}        &   \num{6.051e+1}      &  \num{1.298e+3} & 0.186 &  0.083  \\ \hline
\end{tabular}
\caption{$\mathcal{RMS}$ vs. $L$ for $\lambda=0.5$, with the explicit scheme. The coarsest mesh has $N^x_0 = 41$ cells, while the finest mesh has $N^x_L$ cells for MLMC algorithm with $L$ levels.}
\label{tab:0p5_ex}
\end{table}

\begin{table}[!htbp]
\centering
\begin{tabular}{|c|l|l|l|l|l|l|}
\hline
\textbf{L}            & \textbf{1} & \textbf{2} & \textbf{3} & \textbf{4} & \textbf{$r$} & \textbf{expected $r$}\\ \hline
\textbf{\begin{tabular}[c]{@{}c@{}}Number \\of \\Samples\end{tabular}}      &     \begin{tabular}[c]{@{}l@{}}$M_0 = 13$\\ $M_1=2$\end{tabular}       &    \begin{tabular}[c]{@{}l@{}}$M_0 = 82$\\ $M_1=9$ \\ $M_2=2$\end{tabular} &       \begin{tabular}[c]{@{}l@{}}$M_0 = 544$\\ $M_1=60$ \\ $M_2=9$ \\ $M_3 = 2$\end{tabular}     &  \begin{tabular}[c]{@{}l@{}}$M_0 = 3623$\\ $M_1=395$ \\ $M_2=60$ \\ $M_3 = 9$ \\ $M_4 = 2$\end{tabular} &     &      \\ \hline
$\mathbf{N^x_L}$ &  123   & 369   & 1107 & 3321 & 0.267 & 0.227\\ \hline
$\mathbf{\mathcal{RMS}}$          &       \num{1.022e-1}     &  \num{6.925e-2}          &   \num{4.846e-2}         &      \num{4.327e-2} &       &    \\ \hline
\textbf{Run time(s)}  &    \num{2.459e-1}     &  \num{3.322e+0}         &   \num{5.842e+1}      & \num{1.262e+3}  & 0.102 &  0.076 \\ \hline
\end{tabular}
\caption{$\mathcal{RMS}$ vs. $L$ for $\lambda=0.75$, with the explicit scheme. The coarsest mesh has $N^x_0 = 41$ cells, while the finest mesh has $N^x_L$ cells for MLMC algorithm with $L$ levels.}
\label{tab:0p75_ex}
\end{table}

\begin{table}[!htbp]
\centering
\begin{tabular}{|c|l|l|l|l|l|l|}
\hline
\textbf{L}            & \textbf{1} & \textbf{2} & \textbf{3} & \textbf{4} & \textbf{$r$} & \textbf{expected $r$}\\ \hline
\textbf{\begin{tabular}[c]{@{}c@{}}Number \\of \\Samples\end{tabular}}      &     \begin{tabular}[c]{@{}l@{}}$M_0 = 10$\\ $M_1=2$\end{tabular}       &    \begin{tabular}[c]{@{}l@{}}$M_0 = 72$\\ $M_1=9$ \\ $M_2=2$\end{tabular} &       \begin{tabular}[c]{@{}l@{}}$M_0 = 532$\\ $M_1=65$ \\ $M_2=9$ \\ $M_3 = 2$\end{tabular}     &  \begin{tabular}[c]{@{}l@{}}$M_0 = 3933$\\ $M_1=481$ \\ $M_2=65$ \\ $M_3 = 9$ \\ $M_4 = 2$\end{tabular} &     &      \\ \hline
$\mathbf{N^x_L}$ &  123   & 369   & 1107 & 3321 &0.534  & 0.071 \\ \hline
$\mathbf{\mathcal{RMS}}$          &     \num{8.996e-2}       &     \num{5.430e-2}       &    \num{2.990e-2}       &   \num{1.574e-2}    &       &    \\ \hline
\textbf{Run time(s)}  &    \num{5.895e-1}     &    \num{1.508e+1}       &    \num{4.901e+2}     &  \num{1.955e+4}  & 0.169 &  0.020 \\ \hline
\end{tabular}
\caption{$\mathcal{RMS}$ vs. $L$ for $\lambda=1.5$, with the explicit scheme. The coarsest mesh has $N^x_0 = 41$ cells, while the finest mesh has $N^x_L$ cells for MLMC algorithm with $L$ levels.}
\label{tab:1p5_ex}
\end{table}

\begin{table}[!htbp]
\centering
\begin{tabular}{|c|l|l|l|l|l|}
\hline
\textbf{L}            & \textbf{1} & \textbf{2} & \textbf{3} & \textbf{$r$} & \textbf{expected $r$}\\ \hline
\textbf{\begin{tabular}[c]{@{}c@{}}Number \\of \\Samples\end{tabular}}      &     \begin{tabular}[c]{@{}l@{}}$M_0 = 26$\\ $M_1=2$\end{tabular}       &    \begin{tabular}[c]{@{}l@{}}$M_0 = 365$\\ $M_1=17$ \\ $M_2=2$\end{tabular} &       \begin{tabular}[c]{@{}l@{}}$M_0 = 4950$\\ $M_1=221$ \\ $M_2=16$ \\ $M_3 = 2$\end{tabular}    &     &      \\ \hline
$\mathbf{N^x_L}$ &  123   & 369   & 1107 & 0.606 & 0.25 \\ \hline
$\mathbf{\mathcal{RMS}}$          &   \num{8.958e-2}    &   \num{4.443e-2}             &  \num{2.364e-2}  &       &    \\ \hline
\textbf{Run time(s)}  & \num{1.939e+1}  &  \num{9.401e+2}  & \num{7.382e+4}  &  0.161 &  0.063 \\ \hline
\end{tabular}
\caption{$\mathcal{RMS}$ vs. $L$ for $\lambda=0.5$, with the explicit-implicit scheme. The coarsest mesh has $N^x_0 = 41$ cells, while the finest mesh has $N^x_L$ cells for MLMC algorithm with $L$ levels.}
\label{tab:0p5_im}
\end{table}

\begin{table}[!htbp]
\centering
\begin{tabular}{|c|l|l|l|l|l|}
\hline
\textbf{L}            & \textbf{1} & \textbf{2} & \textbf{3} & \textbf{$r$} & \textbf{expected $r$}\\ \hline
\textbf{\begin{tabular}[c]{@{}c@{}}Number \\of \\Samples\end{tabular}}      &     \begin{tabular}[c]{@{}l@{}}$M_0 = 26$\\ $M_1=2$\end{tabular}       &    \begin{tabular}[c]{@{}l@{}}$M_0 = 365$\\ $M_1=17$ \\ $M_2=2$\end{tabular} &       \begin{tabular}[c]{@{}l@{}}$M_0 = 4950$\\ $M_1=221$ \\ $M_2=16$ \\ $M_3 = 2$\end{tabular}    &     &      \\ \hline
$\mathbf{N^x_L}$ &  123   & 369   & 1107 & 0.343 & 0.25\\ \hline
$\mathbf{\mathcal{RMS}}$          &   \num{9.559e-2}    &    \num{6.037e-2}            &  \num{4.496e-2}   &       &     \\ \hline
\textbf{Run time(s)}  & \num{2.121e+1}  &  \num{1.117e+3}  &  \num{7.591e+4} & 0.092 & 0.063  \\ \hline
\end{tabular}
\caption{$\mathcal{RMS}$ vs. $L$ for $\lambda=0.75$, with the explicit-implicit scheme. The coarsest mesh has $N^x_0 = 41$ cells, while the finest mesh has $N^x_L$ cells for MLMC algorithm with $L$ levels.}
\label{tab:0p75_im}
\end{table}

\begin{table}[!htbp]
\centering
\begin{tabular}{|c|l|l|l|l|l|}
\hline
\textbf{L}            & \textbf{1} & \textbf{2} & \textbf{3} & \textbf{$r$} & \textbf{expected $r$}\\ \hline
\textbf{\begin{tabular}[c]{@{}c@{}}Number \\of \\Samples\end{tabular}}      &     \begin{tabular}[c]{@{}l@{}}$M_0 = 24$\\ $M_1=2$\end{tabular}       &    \begin{tabular}[c]{@{}l@{}}$M_0 = 384$\\ $M_1=18$ \\ $M_2=2$\end{tabular} &       \begin{tabular}[c]{@{}l@{}}$M_0 = 5971$\\ $M_1=277$ \\ $M_2=17$ \\ $M_3 = 2$\end{tabular}    &     &      \\ \hline
$\mathbf{N^x_L}$ &  123   & 369   & 1107 & 0.538 & 0.125 \\ \hline
$\mathbf{\mathcal{RMS}}$          &   \num{7.994e-2}    &     \num{4.599e-2}           &  \num{2.452e-2}   &       &     \\ \hline
\textbf{Run time(s)}  &  \num{7.059e1}  &  \num{5.612e3}  & \num{4.890e5}  &  0.133 &  0.028 \\ \hline
\end{tabular}
\caption{$\mathcal{RMS}$ vs. $L$ for $\lambda=1.5$, with the explicit-implicit scheme. The coarsest mesh has $N^x_0 = 41$ cells, while the finest mesh has $N^x_L$ cells for MLMC algorithm with $L$ levels.}
\label{tab:1p5_im}
\end{table}

Finally, we numerically demonstrate that the MLMC-FDM algorithm is superior to the MC-FDM algorithm. In order to do this, we compute the $\mathcal{RMS}$ error with the explicit MLMC scheme for $L=3$, $\lambda = 0.5$ and $N_0^x \in \{ 31,41,51,61,71,81\}$. For each $N_0^x$, we compute the statistics using the explicit MC scheme on the finest level with $N^x_L = 3^LN_0^x$ cells. We choose the number of samples using \eqref{eqn:MCsamples}. In order to have a reasonable number of samples, we take $C=2$ in \eqref{eqn:MCsamples}. The $\mathcal{RMS}$ error for the MC-FDM algorithm is also computed using \eqref{eqn:rms} with $Q=30$, where $Z_k$ is the computed estimated mean \eqref{MC_estimate_defn}. We plot the variation in $\mathcal{RMS}$ error as a function of the mesh size $N^x_L$ in Figure \ref{fig:MC_MLMC}(A). While both algorithms show a similar decay, the MLMC algorithm leads to smaller errors on a given mesh. However, the work done by the MC algorithm to achieve the same level of $\mathcal{RMS}$ error as the MLMC algorithm is significantly higher (almost 100 times), as shown in Figure \ref{fig:MC_MLMC}(B).

\begin{figure}[!htbp]
\begin{center}
\subcaptionbox{$\mathcal{RMS}$ vs $N^x_L$}{\includegraphics[width=0.45\textwidth]{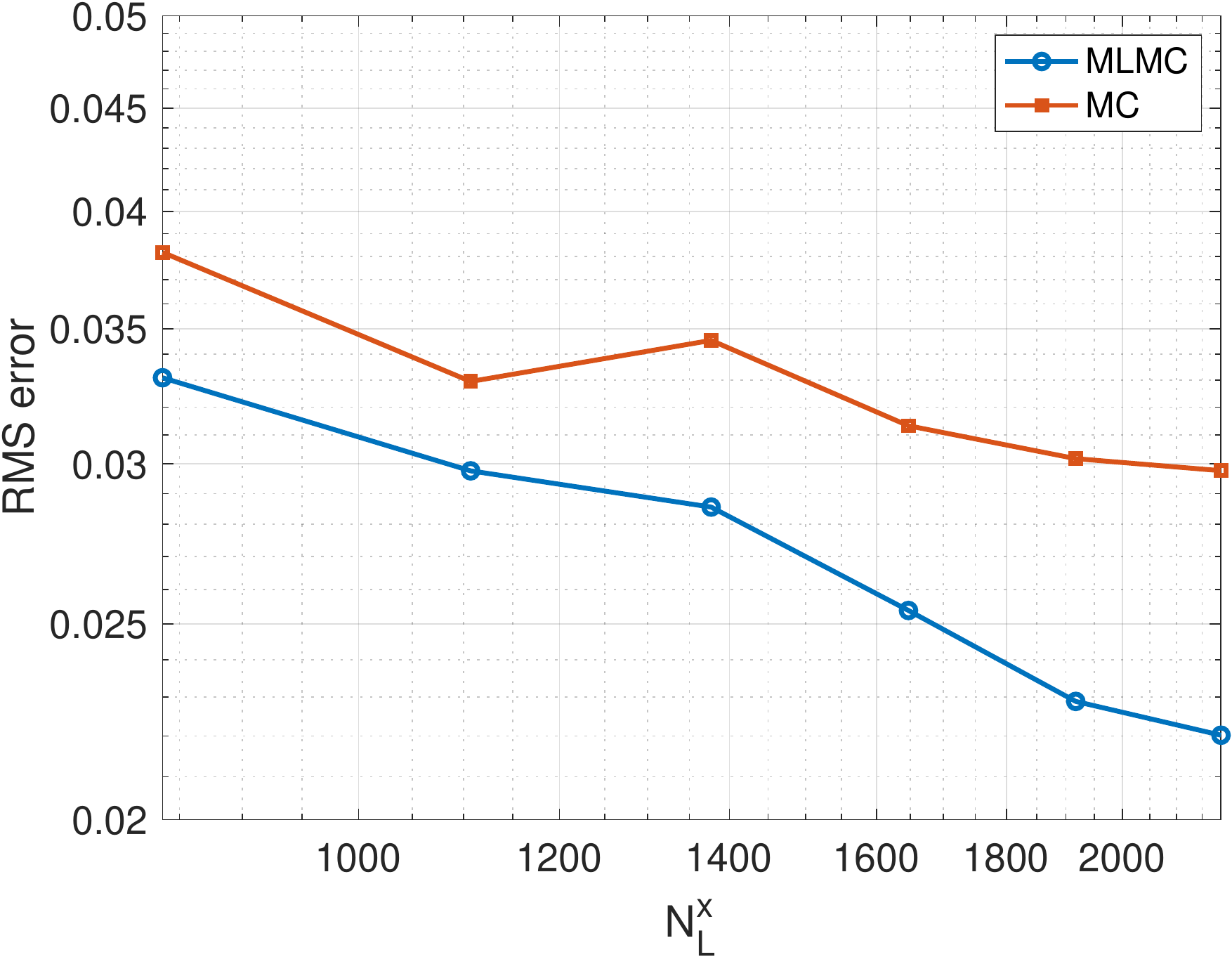}} \hfill
\subcaptionbox{$\mathcal{RMS}$ vs runtime}{\includegraphics[width=0.45\textwidth]{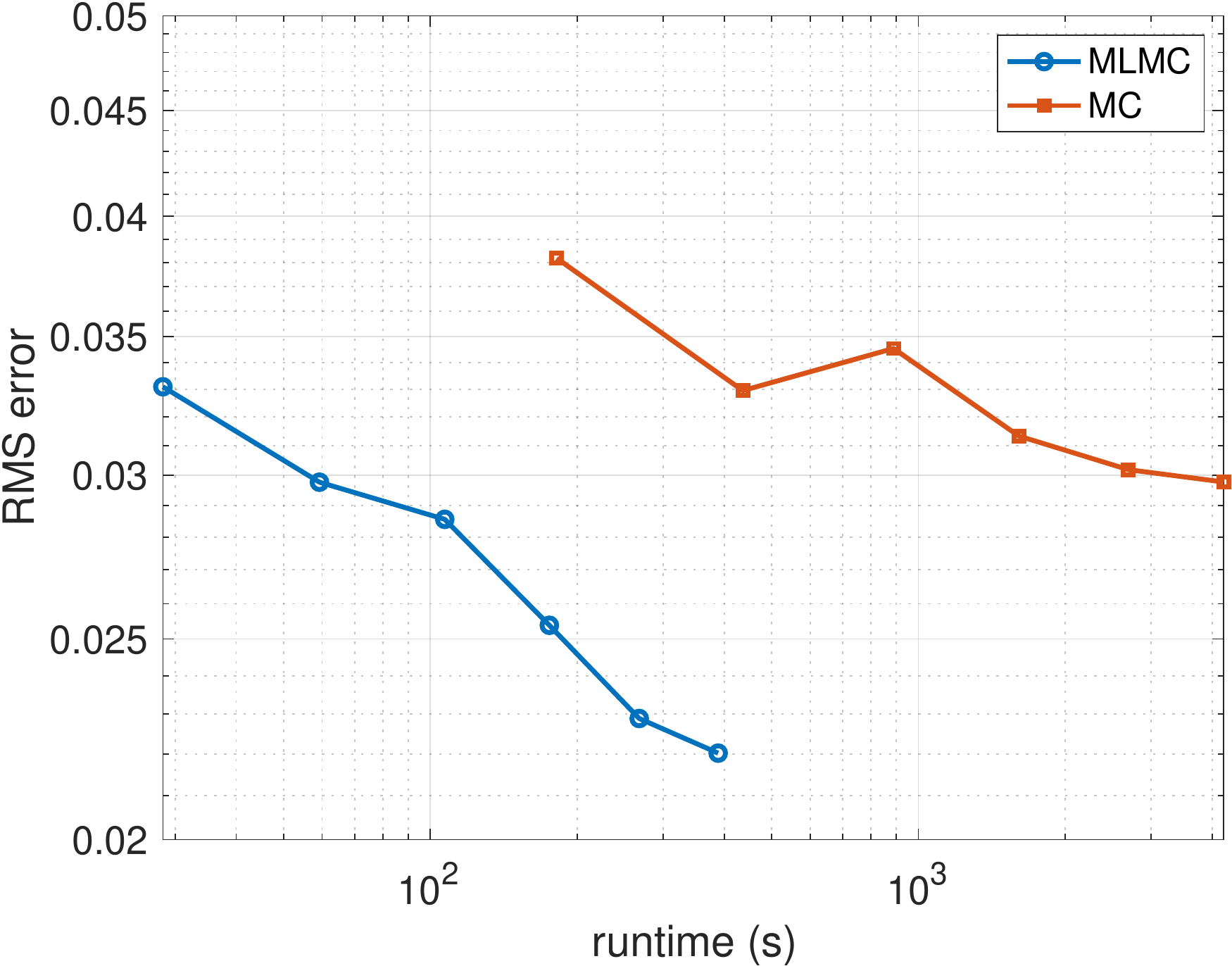}}\hfill
\caption{Comparing the performance of the MC-FDM and MLMC-FDM algorithms. The problem is solved using the explicit scheme for $\lambda=0.5$. The plots are shown in the log-log scale.}
\label{fig:MC_MLMC}
\end{center}
\end{figure}

% ------------------------------------------------------------------------------------------------------------------------------------------

\section{Conclusion}
The proper notion of random entropy solution for degenerate non-linear non-local conservation laws in several space dimension with uncertain initial data and random fluxes is formulated and its well-posedness is demonstrated. We propose a new class of MLMC methods and prove them to be convergent.
MLMC-FDMs are designed in such a way that it maintains the same accuracy vs. work bounds as of deterministic FDM. We have observed that the obtained rates in MLMC are much improved than the single level MC. Hence, MLMC-FDMs are faster than MC-FDMs at comparable accuracy.
We have presented several numerical experiments with Buckley-Leverett in one space dimension that reinforce the theory. The MLMC-FDM algorithms are implemented for various values of $\lambda$ and 
${\mathcal{RMS}}$ error is calculated. It is observed that the numerical convergence rates are better than the theoretical rates for both explicit and explicit-implicit schemes.

% ------------------------------------------------------------------------------------------------------------------------------------------

\section*{Acknowledgements}
U.K acknowledges the support of the Department of Atomic Energy,  Government of India, under project no.$12$-R$\&$D-TFR-$5.01$-$0520$, and India SERB Matrics grant MTR/$2017/000002$.

% ------------------------------------------------------------------------------------------------------------------------------------------

\end{document}